\documentclass[11pt,letterpaper]{article}
\usepackage[T1]{fontenc}
\usepackage[utf8]{inputenc}
\usepackage{lmodern}
\usepackage[left=1in,right=1in,top=0.85in,bottom=0.85in]{geometry}
\usepackage{microtype}
\usepackage{amsmath,amssymb,amsthm,mathtools,bm}
\usepackage{booktabs,tabularx,array}
\usepackage{enumitem}
\usepackage{graphicx}
\usepackage[section]{placeins}
\usepackage{flafter}
\usepackage{authblk}
\usepackage[font=small,labelfont=bf,skip=5pt]{caption}
\usepackage{titlesec}
\titlespacing*{\section}{0pt}{2.2ex plus .6ex minus .3ex}{1.1ex plus .2ex}
\titlespacing*{\subsection}{0pt}{1.7ex plus .5ex minus .2ex}{.7ex plus .2ex}
\titlespacing*{\paragraph}{0pt}{1.1ex plus .3ex minus .2ex}{.7em}
\AtBeginDocument{%
  \setlength{\abovedisplayskip}{8pt plus 2pt minus 2pt}%
  \setlength{\belowdisplayskip}{8pt plus 2pt minus 2pt}%
  \setlength{\abovedisplayshortskip}{3pt plus 1pt}%
  \setlength{\belowdisplayshortskip}{5pt plus 2pt minus 2pt}%
}
\usepackage[dvipsnames]{xcolor}
\usepackage[authoryear,round]{natbib}
\usepackage{xurl}
\usepackage{hyperref}
\usepackage{bookmark}
\graphicspath{{figures/}{./}{Figures/}}
\hypersetup{
 colorlinks=true,
 citecolor=NavyBlue,
 linkcolor=MidnightBlue,
 urlcolor=RoyalBlue,
 pdftitle={Calibrated alpha-Barycenters for Multiobjective f-Divergence Optimization},
 pdfauthor={Dier Tang, Peijie Li, Xiaoyuan Zhang},
 bookmarksdepth=2
}
\renewcommand{\bibfont}{\small}
\makeatletter
\renewcommand*{\l@section}{\@dottedtocline{1}{0em}{2.2em}}
\makeatother
\setlist[enumerate]{itemsep=2pt,topsep=4pt}
\numberwithin{equation}{section}

\newtheorem{theorem}{Theorem}[section]
\newtheorem{corollary}[theorem]{Corollary}
\newtheorem{proposition}[theorem]{Proposition}
\newtheorem{lemma}[theorem]{Lemma}
\newtheorem{assumption}[theorem]{Assumption}
\theoremstyle{definition}
\newtheorem{definition}[theorem]{Definition}
\theoremstyle{remark}

\DeclareMathOperator{\softmax}{softmax}
\DeclareMathOperator{\ri}{ri}
\newcommand{\R}{\mathbb{R}}

\newcommand{\Acal}{\mathcal{A}}
\newcommand{\Ecal}{\mathcal{E}}

\newcommand{\KL}{\operatorname{KL}}

\newcommand{\op}{\mathrm{op}}
\newcommand{\eps}{\varepsilon}
\newcommand{\inner}[2]{\left\langle #1,#2\right\rangle}
\newcommand{\norm}[1]{\left\lVert #1\right\rVert}
\newcommand{\abs}[1]{\left\lvert #1\right\rvert}
\newcommand{\dd}{\,\mathrm{d}}

\newcommand{\EeightSamplesPDF}{E8_generated_samples.pdf}
\newcommand{\EeightImagesPDF}{E8_image_diagnostics.pdf}
\newcommand{\EeightTradeoffsPDF}{E8_latent_tradeoffs.pdf}
\newcommand{\EeightCostPDF}{E8_optimization_cost.pdf}
\newcommand{\EeightAuditsPDF}{E8_quadrature_checks.pdf}
\newcommand{\EeightRepresentationPDF}{E8_representation_checks.pdf}
\newcommand{\email}[1]{\textit{\href{mailto:#1}{#1}}}

\title{Calibrated $\alpha$-Barycenters for Multiobjective\\
       $f$-Divergence Optimization}

\author[1]{Dier Tang}
\author[1]{Peijie Li}
\author[2]{Xiaoyuan Zhang}

\affil[1]{The University of Hong Kong \\
  {\normalfont\small Email: \email{tangde\_math@connect.hku.hk}, \email{lipeijie@hku.hk}}}

\affil[2]{Zhongguancun Academy \\
  {\normalfont\small Email: \email{zhangxiaoyuan@bza.edu.cn}}}

\date{}

\begin{document}
\maketitle
\begin{abstract}
Matching one distribution to many targets can require repeated evaluations of every target. CAL-$\alpha$-MODM aggregates a nonlinear predictor into one barycenter, calibrates its residual derivatives, and checks steps against the original $f$-divergence objective. Under smoothness assumptions, calibration controls parameter-step model error without requiring targets to be close to the model. Exact-oracle stationarity and local convergence guarantees are established. At common matching, canonical geometry cancels the leading quadratic iteration error, giving cubic convergence with second-order calibration and sufficiently accurate inner solves. High-precision experiments support this prediction. Sensitivity-based initialization also accelerates preference continuation along regular stationary branches. On 300-target Jensen--Shannon problems, CAL-p2 achieves median paired speedups of $3.15\times$ over L-BFGS and $1.10\times$ over second-order Direct-TR, across 23/25 and 24/25 objective-comparable stationary pairs. The savings arise when cheaper predictor queries offset calibration overhead.

\vspace{2mm}

\vspace{2mm}

\vspace{2mm}
\noindent \textbf{2012 ACM CCS}: G.1.6, I.2.6
\end{abstract}

\begingroup
\small
\tableofcontents
\endgroup
\clearpage

\section{Introduction}
\label{sec:intro}

Multiple distributional targets arise in ensemble compression and
generative modeling \citep{zhang2025}. Given $P_1,\ldots,P_m$ and
$\mathcal M=\{Q_\theta:\theta\in\Theta\subseteq\R^d\}$,
fix an orientation $\sigma\in\{-,+\}$ and consider
\begin{equation}
 \min_{\theta\in\Theta}
 \bigl(F_1^\sigma(\theta),\ldots,F_m^\sigma(\theta)\bigr),\qquad
 F_i^\sigma(\theta)=
 \begin{cases}
 D_f(P_i\|Q_\theta),&\sigma=-,\\
 D_f(Q_\theta\|P_i),&\sigma=+.
 \end{cases}
 \label{eq:moo}
\end{equation}
All objectives share $f$ and $\sigma$. Positive preferences specify
the trade-off. We study fixed scalarizations with exact oracles,
whose repeated evaluations can require revisiting every target.

Our key idea is to reuse one barycenter for predictor queries and
restore the true objective's local derivatives by residual calibration.
Exact $\alpha$-divergence scalarizations aggregate, but generic $f$
need not: likelihood ratios need not approach one as parameter steps
shrink, so local divergence geometry alone is insufficient.

We Taylor-correct the true objective minus the predictor and check
steps against the true scalarization, following trust-region model
management \citep{alexandrov1998,booker1999}. Aggregation lowers
predictor-query cost. Its nonlinear remainder can also reduce local
iteration error beyond a matched quadratic Taylor model.

\paragraph{Contributions.}
\begin{enumerate}[leftmargin=1.5em,itemsep=0.12em,topsep=0.2em]
\item \textbf{Compressed models with controlled accuracy.}
Exact $\alpha$-aggregation and residual calibration give parameter-step error bounds without target proximity, with stationarity and certified-iteration complexity guarantees.
\item \textbf{Geometry beyond the matched Taylor jet.}
The residual third derivative determines the leading local error coefficient for second-order calibration. Canonical geometry cancels it at common matching, giving a conditional cubic rate without third-order calibration. E10 audits this.
\item \textbf{Measured many-target crossover.}
For 300 JS targets, CAL-p2 gives median paired speedups of $3.15\times$ over L-BFGS and $1.10\times$ over second-order Direct-TR, with 23/25 and 24/25 objective-comparable stationary pairs. Latent targets demonstrate preference-controlled generation.
\end{enumerate}

\paragraph{Organization.}
Sections~\ref{sec:background}--\ref{sec:related} introduce the setting and related work.
Sections~\ref{sec:models}--\ref{sec:theory} develop the models, algorithm, and guarantees.
Extensions and implementation are covered in Sections~\ref{app:continuation}--\ref{app:implementation}, followed by experiments in Section~\ref{sec:experiments}.
The appendices give the remaining proofs and experimental protocols.

\section{Problem formulation}
\label{sec:background}

Let the distributions have positive densities relative to a common measure $\mu$. For convex $f$, normalize $f(1)=f'(1)=0$ and set $a_f=f''(1)>0$. Define \citep{csiszar1967,amari2016}
\begin{equation}
 D_f(P\|Q)=\int q(x)f\!\left(\frac{p(x)}{q(x)}\right)\dd\mu(x).
 \label{eq:fdiv}
\end{equation}
A preference $\lambda\in\Delta_m:=\{\lambda\in\R^m:\lambda_i\geq0,\ \sum_i\lambda_i=1\}$ gives
\begin{equation}
 F_\lambda^\sigma(\theta)=\sum_i\lambda_iF_i^\sigma(\theta).
 \label{eq:scalarizations}
\end{equation}
Write $D_g^-(P,Q)=D_g(P\|Q)$ and $D_g^+(P,Q)=D_g(Q\|P)$ for any generator $g$. The preference $\lambda$ and $\sigma$ remain fixed unless explicitly varied.

Throughout, the model family may be restricted: an aggregated target need
not lie in $\mathcal M$. Unless stated otherwise, optimization takes place
in Euclidean parameter coordinates with exact objective and derivative
oracles. The convergence statements distinguish first-order stationarity
from global minimization and from coverage of the Pareto set.

\section{Related work}
\label{sec:related}
\label{app:extended-related}

\paragraph{Divergence objectives in learning.}
$f$-divergences underlie variational classification objectives
\citep{novello2024fdiv}, convex losses \citep{roulet2025losses},
alignment--diversity trade-offs \citep{go2023alignment}, and preference
optimization \citep{han2025fpo}. They are also used in sequence-level
distillation \citep{wen2023fdistill}, reverse-KL distillation
\citep{gu2024minillm}, and distillation on student-generated sequences
\citep{agarwal2024gkd}. Domain adaptation \citep{acuna2021fdal} and robust
offline reinforcement learning \citep{tang2025robust} use related
discrepancies, although their adversarial formulations differ from our
fixed-target scalarizations.

\paragraph{Local statistical geometry.}
Universal feature analysis relates local inference to low-dimensional
dependence representations \citep{huang2024universalfeature}.
\citet{xu2024neural} study neural feature learning in function space,
where squared dependence norms correspond to $\chi^2$ divergence.
\citet{tang2026sharp} analyze spectral representation learning for local
tasks in $\chi^2$ neighborhoods. Task covariance determines transfer
regret and motivates a correction to spectral learning. Here, local
divergence geometry is used to construct optimization predictors for
fixed distributional targets.

\paragraph{Aggregation and centroids.}
\citet{zhang2025} characterize KL and inverse-KL Pareto sets in full
exponential families and develop multiobjective variational autoencoders.
Bregman centroid identities reduce weighted objectives to one target
\citep{banerjee2005,nielsen2009}. Power-mean annealing paths admit
$\alpha$-divergence characterizations \citep{masrani2021qpaths}, while
transformed Bregman divergences characterize quasi-arithmetic means
\citep{brekelmans2024annealing}. Related constructions yield broader
divergence families \citep{nielsen2022quasi} and approximate Jeffreys
centers \citep{nielsen2024proxy}. We consider general smooth $f$ and
restricted model families, using the aggregate as a predictor even when
its density cannot be represented by $Q_\theta$.

\paragraph{Multiobjective optimization and preferences.}
CAGrad \citep{liu2021cagrad}, Nash-MTL \citep{navon2022bargaining}, and
FAMO \citep{liu2023famo} adapt task influence. Smooth Tchebycheff
scalarization \citep{lin2024tchebycheff} and FERERO
\citep{chen2024ferero} support nonlinear or constrained preferences.
Our aggregation identity instead assumes a fixed linear scalarization.
For repeated preferences, hypernetworks learn preference-to-parameter
maps \citep{navon2021pareto}, and Pareto set learning approximates
solution sets in Bayesian optimization \citep{lin2022pareto}.
E6 uses implicit differentiation \citep{blondel2022implicit} to
initialize a common corrector along regular stationary branches.

\paragraph{Calibration and model management.}
Trust-region model management establishes convergence for optimization
with approximations \citep{alexandrov1998,booker1999}.
Evaluation-complexity theory \citep{cartis2022evaluation} and local
higher-order analysis \citep{doikov2022local} provide the corresponding
global and local perspectives. We specialize these tools to aggregated
nonlinear predictors with calibrated residual derivatives.
Stochastic trust-region methods with inexact restoration address
sampled finite sums \citep{bellavia2023sirtr}. Extending our results to
sampled divergences would require derivative and acceptance-test error
controls beyond the deterministic bounds proved here.

\section{Calibrated barycenter predictors}
\label{sec:models}

\subsection{Canonical local \texorpdfstring{$\alpha$}{alpha} geometry}

For $\alpha\in\R$, define
\begin{equation}
 D_\alpha(P\|Q)=\frac{4}{1-\alpha^2}
 \left(1-\int p^{(1-\alpha)/2}q^{(1+\alpha)/2}\dd\mu\right),
 \label{eq:alpha-div}
\end{equation}
with continuous limits $D_{-1}(P\|Q)=\KL(P\|Q)$ and
$D_1(P\|Q)=\KL(Q\|P)$.  For $\alpha\neq\pm1$, an affine-normalized generator is
\begin{equation}
 f_\alpha(u)=\frac{4}{1-\alpha^2}(1-u^{(1-\alpha)/2})
 +\frac{2}{1+\alpha}(u-1),
 \label{eq:normalized-alpha-generator}
\end{equation}
with endpoint limits $f_{-1}(u)=u\log u-u+1$ and
$f_1(u)=-\log u+u-1$.  It satisfies
$f_\alpha''(1)=1$ and $f_\alpha'''(1)=-(3+\alpha)/2$.

\begin{theorem}[Fourth-order local normal form]
\label{thm:normal-form}
Let $0<\delta<1$, $f\in C^4([1-\delta,1+\delta])$, and define
\begin{equation}
 \alpha_f:=-3-2\frac{f'''(1)}{f''(1)},
 \qquad h_f:=f-a_ff_{\alpha_f}.
 \label{eq:alpha-f}
\end{equation}
Then $h_f^{(j)}(1)=0$ for $j=0,1,2,3$.  If $p/q\in[1-\delta,1+\delta]$ almost everywhere, then
\begin{equation}
 \abs{D_f(P\|Q)-a_fD_{\alpha_f}(P\|Q)}
 \leq \frac{M_{f,\delta}}{24}\int q\abs{p/q-1}^4\dd\mu,
 \label{eq:normal-bound}
\end{equation}
where $M_{f,\delta}=\sup_{u\in[1-\delta,1+\delta]}\abs{h_f^{(4)}(u)}$.
\end{theorem}
\begin{proof}
The normalization gives $f(1)=f'(1)=0$.  From the derivatives of $f_\alpha$ and~\eqref{eq:alpha-f},
\begin{equation}
 a_ff_{\alpha_f}''(1)=f''(1),\qquad
 a_ff_{\alpha_f}'''(1)=-a_f\frac{3+\alpha_f}{2}=f'''(1).
\end{equation}
Therefore $h_f=f-a_ff_{\alpha_f}$ has a zero of order at least four at $u=1$.  Taylor's theorem gives
$h_f(u)=h_f^{(4)}(\xi_u)(u-1)^4/24$ for some $\xi_u$ between $1$ and $u$.  Substitute $u=p/q$, multiply by $q$, integrate, and apply the triangle inequality to obtain~\eqref{eq:normal-bound}.
\end{proof}

The bound controls likelihood-ratio displacement and does not, by itself, imply model accuracy as a parameter step tends
to zero.

\subsection{Exact aggregation}

Let $r_\alpha=(1-\alpha)/2$ and $s_\alpha=(1+\alpha)/2$.  For nonzero exponents define
\begin{align}
 c^-_{\lambda,\alpha}
 &=\frac{(\sum_i\lambda_ip_i^{r_\alpha})^{1/r_\alpha}}{Z^-_{\lambda,\alpha}},
 &Z^-_{\lambda,\alpha}&=\int(\sum_i\lambda_ip_i^{r_\alpha})^{1/r_\alpha}\dd\mu,
 \label{eq:cminus}\\
 c^+_{\lambda,\alpha}
 &=\frac{(\sum_i\lambda_ip_i^{s_\alpha})^{1/s_\alpha}}{Z^+_{\lambda,\alpha}},
 &Z^+_{\lambda,\alpha}&=\int(\sum_i\lambda_ip_i^{s_\alpha})^{1/s_\alpha}\dd\mu.
 \label{eq:cplus}
\end{align}
For a zero exponent, normalize the geometric mean $\prod_i p_i^{\lambda_i}$. We assume finite, strictly positive normalizers and finite divergences wherever the optimization model is evaluated.

\begin{theorem}[Exact one-target aggregation]
\label{thm:aggregation}
For $\alpha\neq\pm1$ and $0<Z^\sigma_{\lambda,\alpha}<\infty$,
\begin{align}
 \sum_i\lambda_iD_\alpha(P_i\|Q)
 &=C^-_{\lambda,\alpha}+\omega^-_{\lambda,\alpha}
 D_\alpha(c^-_{\lambda,\alpha}\|Q),
 &\omega^-_{\lambda,\alpha}&=(Z^-_{\lambda,\alpha})^{r_\alpha},
 \label{eq:agg-minus}\\
 \sum_i\lambda_iD_\alpha(Q\|P_i)
 &=C^+_{\lambda,\alpha}+\omega^+_{\lambda,\alpha}
 D_\alpha(Q\|c^+_{\lambda,\alpha}),
 &\omega^+_{\lambda,\alpha}&=(Z^+_{\lambda,\alpha})^{s_\alpha}.
 \label{eq:agg-plus}
\end{align}
The constants $C^\sigma_{\lambda,\alpha}$ do not depend on $Q$.  The identities extend to $\alpha=\pm1$ whenever all displayed terms admit finite continuous limits.
\end{theorem}
\begin{proof}
Let $K_\alpha=4/(1-\alpha^2)$ and
$H^-(x)=\sum_i\lambda_ip_i(x)^{r_\alpha}$.  Then
\begin{equation}
 \sum_i\lambda_iD_\alpha(P_i\|Q)
 =K_\alpha\left(1-\int H^-(x)q(x)^{s_\alpha}\dd\mu(x)\right).
\end{equation}
By~\eqref{eq:cminus},
$H^-=(Z^-_{\lambda,\alpha})^{r_\alpha}(c^-_{\lambda,\alpha})^{r_\alpha}$, and hence
\begin{align}
 \sum_i\lambda_iD_\alpha(P_i\|Q)
 &=K_\alpha(1-\omega^-_{\lambda,\alpha})
 +\omega^-_{\lambda,\alpha}D_\alpha(c^-_{\lambda,\alpha}\|Q).
\end{align}
The reverse identity follows by taking $H^+(x)=\sum_i\lambda_ip_i(x)^{s_\alpha}$ and exchanging the arguments.
At the KL endpoints, write $\bar p=\sum_i\lambda_ip_i$ and
$g=Z_g^{-1}\prod_i p_i^{\lambda_i}$, with
$Z_g=\int\prod_i p_i^{\lambda_i}\dd\mu$. Direct expansion gives
\begin{align}
 \sum_i\lambda_i\KL(P_i\|Q)
 &=\sum_i\lambda_i\KL(P_i\|\bar P)+\KL(\bar P\|Q),\label{eq:kl-arithmetic}\\
 \sum_i\lambda_i\KL(Q\|P_i)
 &=\KL(Q\|G)-\log Z_g.\label{eq:kl-geometric}
\end{align}
Here $\bar P$ and $G$ have densities $\bar p$ and $g$. Finiteness is assumed for the displayed terms. These are the $\alpha=-1$ identities in the two orientations. The case $\alpha=1$ exchanges them. Their scale factors are one, as also obtained from the stated finite continuous limits.
\end{proof}

The positive factor $\omega^\sigma$ does not change the predictor's minimizers over a fixed model family, but is required to reproduce the multi-target predicted reductions. Omitting it generally changes the acceptance ratio.

\subsection{Derivative calibration}

Fix $\lambda$ and $\sigma$ and write $F=F_\lambda^\sigma$.  Define the structured predictor
\begin{equation}
 \Acal(\theta):=a_f\sum_i\lambda_iD_{\alpha_f}^\sigma(P_i,Q_\theta),
 \label{eq:surrogate}
\end{equation}
which equals a constant plus one weighted divergence to $c^\sigma_{\lambda,\alpha_f}$ by Theorem~\ref{thm:aggregation}. This identity compresses $\Acal$, while $F$ remains the objective to be optimized. We analyze the Euclidean unconstrained case. At iteration $t$, define
$\widehat F_t(v)=F(\theta_t+v)$, $\widehat\Acal_t(v)=\Acal(\theta_t+v)$, and
$\widehat\Ecal_t=\widehat F_t-\widehat\Acal_t$.

\begin{definition}[Order-$p$ calibrated model]
\label{def:model}
For $p\in\{1,2,3\}$,
\begin{equation}
 m_t^{(p)}(v)=\widehat\Acal_t(v)
 +\sum_{j=0}^p\frac{1}{j!}\nabla^j\widehat\Ecal_t(0)[v]^j.
 \label{eq:model}
\end{equation}
\end{definition}
Let $T_pG(v):=\sum_{j=0}^p\nabla^jG(0)[v]^j/j!$ denote the Taylor
polynomial at zero. Equivalently,
\begin{equation}
 m_t^{(p)}(v)=T_p\widehat F_t(v)
 +\bigl(\widehat\Acal_t(v)-T_p\widehat\Acal_t(v)\bigr).
 \label{eq:taylor-remainder}
\end{equation}
Calibration matches the true order-$p$ jet (value and derivatives).
The choice of $\alpha$ affects only the higher-order nonlinear remainder.
E7 tests its finite-step accuracy, E10 its local iteration error, and
E9 its computational cost.

\begin{theorem}[Model order]
\label{thm:model-order}
Suppose $\widehat F_t,\widehat\Acal_t\in C^{p+1}$ on a neighborhood of the trial ball and
$\sup_{\norm v\leq\Delta_t}\norm{\nabla^{p+1}\widehat\Ecal_t(v)}_\op\leq L_{p+1}$. For $\norm v\leq\Delta_t$,
\begin{align}
 \nabla^jm_t^{(p)}(0)&=\nabla^j\widehat F_t(0),\quad 0\leq j\leq p,\label{eq:matching}\\
 \abs{\widehat F_t(v)-m_t^{(p)}(v)}&\leq\frac{L_{p+1}}{(p+1)!}\norm v^{p+1},\label{eq:model-value}\\
 \norm{\nabla\widehat F_t(v)-\nabla m_t^{(p)}(v)}&\leq\frac{L_{p+1}}{p!}\norm v^p.\label{eq:model-grad}
\end{align}
\end{theorem}
\begin{proof}
Because $m_t^{(p)}=\widehat\Acal_t+T_p\widehat\Ecal_t$ and
$\widehat F_t=\widehat\Acal_t+\widehat\Ecal_t$, derivative matching follows from the Taylor polynomial.  Taylor's theorem with integral remainder gives
\begin{equation}
 \abs{\widehat\Ecal_t(v)-T_p\widehat\Ecal_t(v)}
 \leq\frac{L_{p+1}}{(p+1)!}\norm v^{p+1}.
\end{equation}
Differentiation commutes with Taylor truncation: for $0\leq j\leq p$, $\nabla^jT_p\widehat\Ecal_t$ is the degree-$(p-j)$ Taylor polynomial of $\nabla^j\widehat\Ecal_t$. Applying the integral remainder to this tensor-valued function bounds its error by $L_{p+1}\norm v^{p+1-j}/(p+1-j)!$. The case $j=1$ gives~\eqref{eq:model-grad}.
\end{proof}

Theorem~\ref{thm:model-order} does not require target proximity. Near
matching, likelihood-ratio moments further bound the residual derivatives.

\subsection{Residual derivatives near matching}
\label{sec:calibration-size}

\begin{proposition}[Residual derivative bounds]\label{prop:residual-tensors}
Under the normalization and smoothness assumptions of Theorem~\ref{thm:normal-form}, let $R^-(\theta)=D_f(P\|Q_\theta)-a_fD_{\alpha_f}(P\|Q_\theta)$. Suppose differentiation under the integral is valid, $p/q_\theta\in[1-\delta,1+\delta]$, and the score and log-density Hessian are uniformly bounded. Then
\begin{align}
 \norm{\nabla R^-(\theta)}&\leq C_1\int q_\theta\abs{p/q_\theta-1}^3\dd\mu,\\
 \norm{\nabla^2R^-(\theta)}_\op&\leq C_2\int q_\theta\abs{p/q_\theta-1}^2\dd\mu.
\end{align}
Analogous bounds hold for $R^+(\theta)=D_f(Q_\theta\|P)-a_fD_{\alpha_f}(Q_\theta\|P)$ when $q_\theta/p\in[1-\delta,1+\delta]$.
If $\nabla^3\log q_\theta$ is also uniformly bounded and third-order
differentiation under the integral is valid, then
\begin{equation}
 \norm{\nabla^3R^\sigma(\theta)}_\op
 \leq C_3\int q_\theta\abs{u_\sigma-1}\dd\mu,
 \qquad u_-=p/q_\theta,\quad u_+=q_\theta/p.
 \label{eq:residual-third-bound}
\end{equation}
The constants depend on the density-derivative bounds, $\delta$, and
$M_{f,\delta}$, not on the parameter step.
\end{proposition}

The proof is given in Appendix~\ref{app:calibration-size}.
These likelihood-ratio bounds complement the step-based guarantees of
Theorem~\ref{thm:model-order}. They do not replace residual calibration
when the targets are heterogeneous.

\section{CAL-\texorpdfstring{$\alpha$}{alpha}-MODM}
\label{sec:algorithm}

\subsection{Certified trust-region steps}

At a certified iteration, approximately minimize $m_t^{(p)}$ over $\norm v\leq\Delta_t$ and form
\begin{equation}
 \operatorname{pred}_t=m_t^{(p)}(0)-m_t^{(p)}(s_t),\quad
 \operatorname{ared}_t=F(\theta_t)-F(\theta_t+s_t),\quad
 \rho_t=\frac{\operatorname{ared}_t}{\operatorname{pred}_t}.
 \label{eq:ratio}
\end{equation}
Let $g_t:=\nabla F(\theta_t)=\nabla m_t^{(p)}(0)$.  We require the approximate trust-region solution to achieve a fixed fraction of Cauchy decrease:
\begin{equation}
 \operatorname{pred}_t\geq\kappa_c\norm{g_t}
 \min\left\{\Delta_t,\frac{\norm{g_t}}{\kappa_H}\right\},
 \qquad \kappa_c\in(0,1],
 \label{eq:cauchy}
\end{equation}
where $\kappa_H>0$ uniformly bounds $\norm{\nabla^2m_t^{(p)}(v)}_\op$ on the trust region. A safeguarded gradient step gives $\kappa_c=1/2$ (Section~\ref{app:tr-assumptions}). Larger values require verified decrease.

For $0<\eta_1<\eta_2<1$, $0<\gamma_1<1<\gamma_2$, use the explicit radius update
\begin{equation}
 \Delta_{t+1}=\begin{cases}
  \gamma_1\Delta_t,&\rho_t<\eta_1,\\
  \min\{\gamma_2\Delta_t,\Delta_{\max}\},&\rho_t\geq\eta_2,\\
  \Delta_t,&\text{otherwise}.
 \end{cases}
 \label{eq:radius-update}
\end{equation}

\begin{center}
\fbox{\begin{minipage}{\dimexpr\linewidth-2\fboxsep-2\fboxrule\relax}
\textbf{Algorithm 1: CAL-$\alpha$-MODM}
\begin{enumerate}[leftmargin=1.4em,itemsep=0.05em,topsep=0.15em]
 \item Given $\theta_0$ and $0<\Delta_0\leq\Delta_{\max}$, normalize $f$, compute $(a_f,\alpha_f)$, and construct $(c^\sigma_{\lambda,\alpha_f},\omega^\sigma_{\lambda,\alpha_f})$.
 \item At the certified point $\theta_t$, evaluate $F$ and the residual jet through order $p$.  Stop only if $\norm{\nabla F(\theta_t)}\leq\eps$.
 \item Build~\eqref{eq:model} and compute $\norm{s_t}\leq\Delta_t$ satisfying~\eqref{eq:cauchy}.
 \item Evaluate the true scalarized objective at $\theta_t+s_t$.  Set $\theta_{t+1}=\theta_t+s_t$ if $\rho_t\geq\eta_1$ and $\theta_{t+1}=\theta_t$ otherwise.
 \item Update $\Delta_{t+1}$ by~\eqref{eq:radius-update} and repeat.
\end{enumerate}
\end{minipage}}
\end{center}

\subsection{Calibration order and reuse}

\paragraph{Calibration order.}
Order $p=1$ requires residual values and gradients. Order $p=2$ adds
Hessians and permits local quadratic convergence under
Theorem~\ref{thm:local}. Order $p=3$ permits local cubic convergence but
is more expensive. A p1 inner solver may still use predictor Hessians.
With $p=2$, the retained nonlinear remainder can improve local accuracy
without a true third-derivative tensor (Proposition~\ref{prop:leading-error}).

\paragraph{Inner solves and reuse.}
E8/E9 use regularized Newton iterations, trust-region projection,
Armijo backtracking, and a steepest-descent comparator.
Section~\ref{app:implementation} gives the settings and relates the
comparator to Cauchy decrease. The residual jet is reused across inner
queries and rejected trials at a fixed anchor. A smaller radius tightens
the error bound. Accepted steps require fresh calibration.

\paragraph{Computational savings.}
Cached aggregate values on a reusable support or grid remove the target
sum from predictor queries. New nodes may require all targets to be
evaluated again. Calibration and acceptance still use the true objective,
so predictor savings must offset this cost. The aggregate is reused for
a fixed preference and rebuilt when the preference changes.

\subsection{Safeguarded lazy calibration}
\label{app:lazy}

After a successful certified step $s_t$, set
$z_t=\theta_t+s_t$ and retain $g_t=\nabla F(\theta_t)$.
Before the next calibration, allow at most $K_{\rm lazy}$ extra trials
using the fixed-origin, value-matched predictor
\[
 m_{z_t}^{(0)}(v)=\Acal(z_t+v)+F(z_t)-\Acal(z_t).
\]
A trial from $z_{t,j}=z_t+v_j$ queries this model at $v_j+u_j$.
It is not recentered without recalibration. Accept a lazy increment
only if, for fixed positive constants,
\begin{align}
 F(z_{t,j})-F(z_{t,j}+u_j)&\geq c_{\rm lazy}\norm{u_j}^2,\notag\\
 \sum_j\norm{u_j}&\leq
 \min\{\kappa_{\rm path}\norm{g_t},\kappa_{\rm move}\norm{s_t}\}.
 \label{eq:lazy-safeguards}
\end{align}
The sum is over accepted increments within this block. Its endpoint
becomes $\theta_{t+1}$ and receives a fresh residual jet and true-gradient
stopping test. No lazy block follows a rejected certified step, and
lazy trials do not alter the radius update~\eqref{eq:radius-update}.
The step-length budget links uncalibrated motion to certified progress.
An anchor-gradient budget alone does not provide this link.

\begin{corollary}[Lazy-search complexity]
\label{cor:lazy}
Under Assumption~\ref{ass:tr}, including Lipschitz continuity on the lazy
segments, the preceding block protocol and~\eqref{eq:lazy-safeguards}
give vanishing gradients along the complete accepted sequence of a
zero-tolerance run, or finite exact stationarity. With tolerance
$\eps>0$, the certified jet complexity is $O(\eps^{-2})$ and the true
scalarized value complexity is $O((K_{\rm lazy}+1)\eps^{-2})$.
\end{corollary}

E4 uses a local curvature estimate and its original path budget, without
the additional step-length budget above. Its stationarity results are
empirical and are not inferred from this corollary.

The proof of Corollary~\ref{cor:lazy} is in Appendix~\ref{app:lazy-proof}.

\section{Convergence and Pareto guarantees}
\label{sec:theory}

\subsection{Exact-oracle assumptions and sufficient decrease}
\label{app:tr-assumptions}

\begin{assumption}[Trust-region regularity]
\label{ass:tr}
Fix $\lambda$ and $\sigma$, set $\Theta=\R^d$, use exact objective and derivative oracles, and assume:
\begin{enumerate}[label=(A\arabic*),leftmargin=1.8em,itemsep=0.1em]
 \item The level set $\{\theta:F(\theta)\leq F(\theta_0)\}$ is compact, $F\geq F_{\inf}$ there, and $\nabla F$ is $L_g$-Lipschitz, with $L_g>0$, on a neighborhood containing every trial segment.
 \item The objective and predictor satisfy Theorem~\ref{thm:model-order} on every trial ball, with a uniform residual-derivative bound for $\Delta_t\leq\Delta_{\max}$.
 \item $\sup_{t,\norm v\leq\Delta_t}\norm{\nabla^2m_t^{(p)}(v)}_\op\leq\kappa_H$.
 \item The computed step has $\norm{s_t}\leq\Delta_t$ and satisfies~\eqref{eq:cauchy} for fixed $\kappa_c\in(0,1]$.
 \item The acceptance decision uses~\eqref{eq:ratio}, with $\operatorname{pred}_t>0$, and the radius follows~\eqref{eq:radius-update} for fixed $0<\eta_1<\eta_2<1$ and $0<\gamma_1<1<\gamma_2$.
\end{enumerate}
\end{assumption}

Condition~\eqref{eq:cauchy} is attainable for the nonlinear model itself.
For $g_t\ne0$, let $\tau_t=\min\{\Delta_t/\norm{g_t},1/\kappa_H\}$
and $s_t^{\rm C}=-\tau_tg_t$. The Hessian bound along this segment gives
\[
 m_t^{(p)}(0)-m_t^{(p)}(s_t^{\rm C})
 \geq\tau_t\norm{g_t}^2-\tfrac12\kappa_H\tau_t^2\norm{g_t}^2
 \geq\tfrac12\norm{g_t}\min\{\Delta_t,\norm{g_t}/\kappa_H\}.
\]
Thus any feasible step with at least this model decrease satisfies
\eqref{eq:cauchy} with $\kappa_c=1/2$. Section~\ref{app:implementation}
describes a backtracking safeguard when $\kappa_H$ is not known.

For every $p\geq1$, Theorem~\ref{thm:model-order} implies fully linear bounds
\begin{align}
 \abs{F(\theta_t+v)-m_t^{(p)}(v)}&\leq\kappa_f\Delta_t^2,\\
 \norm{\nabla F(\theta_t+v)-\nabla m_t^{(p)}(v)}&\leq\kappa_g\Delta_t
\end{align}
for $\norm v\leq\Delta_t$, choosing positive constants $\kappa_f,\kappa_g$. For $p=2$ the sharper orders are $(3,2)$, and for $p=3$ they are $(4,3)$. An exact stationary endpoint is extended constantly when discussing asymptotic rates. Proofs of the global and local convergence results are in Appendix~\ref{app:trust}. The global result concerns the unconstrained Euclidean iteration. A bounded affine probability chart is covered by the local result only while its boundary safeguard is inactive.

\subsection{Global convergence and Pareto stationarity}

\begin{theorem}[Global convergence and complexity]
\label{thm:global}
Under Assumption~\ref{ass:tr}, for fixed $\lambda$ and orientation, the zero-tolerance version of Algorithm~1 either stops at an exact stationary point or satisfies
\begin{equation}
 \lim_{t\to\infty}\norm{\nabla F_\lambda^\sigma(\theta_t)}=0.
\end{equation}
For stopping tolerance $\eps>0$, there are problem-dependent constants $C_0,C_1>0$ such that an iterate with
$\norm{\nabla F_\lambda^\sigma}\leq\eps$ is obtained within
\begin{equation}
 N(\eps)\leq C_0+C_1\eps^{-2}
 \label{eq:complexity}
\end{equation}
certified iterations and true scalarized trial evaluations. With explicit summation, each scalarized value or order-$p$ jet call requires $m$ component calls. The constants in~\eqref{eq:complexity} may depend on $m$.
\end{theorem}

This is the standard smooth nonconvex trust-region exponent \citep{conn2000,cartis2010}. Higher $p$ improves local accuracy, not the global exponent or certified component-oracle factor $m$.

\begin{corollary}[Pareto stationarity]
\label{cor:pareto}
For $C^1$ components and $\lambda\in\ri(\Delta_m)$, every accumulation point in Theorem~\ref{thm:global} is Pareto stationary with a strictly positive multiplier. If each component is convex, that point is globally Pareto optimal.
\end{corollary}

\subsection{Local convergence order}

\begin{theorem}[Local order]
\label{thm:local}
Let $p\in\{2,3\}$ and let $\theta^\star$ be a nondegenerate local minimizer:
$\nabla F(\theta^\star)=0$ and
$\nabla^2F(\theta^\star)\succeq\mu I$ for some $\mu>0$. Suppose $F$ and $\Acal$ are $C^{p+1}$.  Let $s_p(\theta)$ be the unique local stationary branch of $m_\theta^{(p)}$ furnished by Lemma~\ref{lem:local-branch}.  Assume the subproblem solver returns steps in the local neighborhood of Lemma~\ref{lem:local-branch} satisfying
\begin{equation}
 \norm{\nabla m_t^{(p)}(s_t)}\leq\kappa_s\norm{\nabla F(\theta_t)}^p.
 \label{eq:inner}
\end{equation}
For Algorithm~1 with zero stopping tolerance and fixed $\Delta_0\in(0,\Delta_{\max}]$, there exists
a neighborhood $U_{\Delta_0}$ of $\theta^\star$ such that, for $\theta_0\in U_{\Delta_0}$, the local steps are eventually
untruncated and accepted, and
\begin{equation}
 \norm{\theta_{t+1}-\theta^\star}
 \leq C_p\norm{\theta_t-\theta^\star}^p.
 \label{eq:qorder}
\end{equation}
The guaranteed local order is at least quadratic for $p=2$ and cubic for $p=3$.
\end{theorem}

The forcing condition~\eqref{eq:inner} provides a sufficient inner-accuracy requirement for our local-rate guarantee, in the spirit of inexact Newton methods \citep{dennis1974,dembo1982}. Lazy safeguards appear in Section~\ref{app:lazy}. Common-generator Bregman objectives are zero-residual special cases (Section~\ref{app:bregman}).

\subsection{Leading error and canonical cancellation}

\begin{proposition}[Leading local error and canonical cancellation]
\label{prop:leading-error}
Let $F,\Acal\in C^4$ near $\theta^\star$, with $\nabla F(\theta^\star)=0$
and $H_\star=\nabla^2F(\theta^\star)\succ0$. Set $\Ecal=F-\Acal$.
For order-two calibrated steps in the local neighborhood of
Theorem~\ref{thm:local}, impose the stronger forcing condition
$\norm{\nabla m_\theta^{(2)}(s)}=O(\norm{\nabla F(\theta)}^3)$.
Once untruncated and accepted, with $e=\theta-\theta^\star$,
\begin{equation}
 e_+:=\theta+s-\theta^\star
 =\tfrac12 H_\star^{-1}\nabla^3\Ecal(\theta^\star)[e,e]
   +O(\norm e^3).
 \label{eq:leading-local-error}
\end{equation}
If every positive-weight target equals $Q_{\theta^\star}$, canonical
$\alpha_f$ gives $\nabla^3\Ecal(\theta^\star)=0$ in either orientation,
provided the divergence integrals can be differentiated through order
four. The local rate is then at least cubic with $p=2$.
\end{proposition}

Here $\nabla^3\Ecal[e,e]$ is the vector whose inner product with $v$
is $\nabla^3\Ecal[v,e,e]$.
For a quadratic Taylor model the corresponding tensor is $\nabla^3F$.
The nonlinear predictor subtracts $\nabla^3\Acal$ from this coefficient.
Common matching is sufficient for cancellation. Heterogeneous targets
need not give a cubic rate. The stronger forcing condition uses only
model gradients and Hessians.

\label{app:leading-error}
\begin{proof}
Write $x=\theta$, $y=x+s$, $e=x-\theta^\star$, and
$r=\nabla m_x^{(2)}(s)$. On a fixed sufficiently small neighborhood,
$C^4$ regularity bounds the fourth derivatives. Taylor expansion of
the residual gradient gives
\begin{equation}
 \nabla F(y)=r+\tfrac12\nabla^3\Ecal(x)[s,s]+O(\norm s^3).
 \label{eq:leading-gradient-identity}
\end{equation}
The local stationary model branch from Lemma~\ref{lem:local-branch}
is $O(\norm e)$, and local strong convexity bounds the distance of
an inexact step from that branch by $C\norm r$. Since
$\norm{\nabla F(x)}=O(\norm e)$, cubic forcing gives
$r=O(\norm e^3)$ and $s=O(\norm e)$.
Equation~\eqref{eq:leading-gradient-identity}, together with local
invertibility of $\nabla F$, then yields $y-\theta^\star=O(\norm e^2)$.
Consequently $s=-e+O(\norm e^2)$ and
\[
 \nabla^3\Ecal(x)[s,s]
 =\nabla^3\Ecal(\theta^\star)[e,e]+O(\norm e^3).
\]
Expanding $\nabla F(y)=H_\star(y-\theta^\star)
+O(\norm{y-\theta^\star}^2)$ in
\eqref{eq:leading-gradient-identity} proves
\eqref{eq:leading-local-error}. These same local steps satisfy the
$p=2$ forcing condition and are eventually accepted as in the proof
of Theorem~\ref{thm:local}.

For canonical $\alpha_f$, each component of $\Ecal$ is an integral
with residual generator $h_f=f-a_ff_{\alpha_f}$.
At $P_i=Q_{\theta^\star}$, its density ratio equals one.
Every term in the first three parameter derivatives of either
$q_\theta h_f(p_i/q_\theta)$ or $p_i h_f(q_\theta/p_i)$ contains
$h_f^{(j)}(1)$ for some $j\leq3$, which vanishes by
Theorem~\ref{thm:normal-form}. Differentiating under the integral and
summing over positive-weight targets gives
$\nabla^3\Ecal(\theta^\star)=0$ and hence
$\norm{e_+}=O(\norm e^3)$. For a finite positive categorical model,
the interchange is automatic for a $C^4$ parameterization near an
interior point. No third derivative is needed by the order-two
optimizer: it appears only in the analysis and offline audit.
\end{proof}

\begin{table}[!htbp]
\caption{Distinct notions of order used in the paper.}
\label{tab:orders}
\centering
\small
\setlength{\tabcolsep}{5pt}
\renewcommand{\arraystretch}{1.04}
\setlength{\abovecaptionskip}{0pt}
\setlength{\belowcaptionskip}{4pt}
\begin{tabularx}{\linewidth}{@{}l l X@{}}
\toprule
Notion & Order & Meaning \\
\midrule
Likelihood-ratio normal form & $4$ & Error as $p/q\to1$ \\
Calibrated model $p$ & $(p+1,p)$ & Value/gradient error in the step \\
Global complexity & $O(\eps^{-2})$ & Certified iterations to first-order stationarity \\
Local iteration & At least Q-$p$ & Error recursion for $p=2,3$ \\
\bottomrule
\end{tabularx}
\end{table}
\section{Preference continuation}
\label{app:continuation}

This extension uses standard implicit differentiation for smooth
scalarizations \citep{blondel2022implicit}. It does not depend on
$\alpha$-aggregation. CAL supplies the common certified corrector.

Fix an orientation $\sigma$.
Assume that each $F_i^\sigma$ is $C^3$ near $\theta^\star$,
that $\lambda^\star\in\ri(\Delta_m)$, and that
\[
 \nabla F_{\lambda^\star}^\sigma(\theta^\star)=0,
 \qquad
 \nabla^2F_{\lambda^\star}^\sigma(\theta^\star)
 \text{ is nonsingular}.
\]
Parameterize the simplex by
$\lambda_m=1-\sum_{j<m}\lambda_j$.
The implicit function theorem yields a locally unique
$C^2$ stationary branch $\theta^\star(\lambda)$.
Writing
$H_\lambda=
\nabla^2F_\lambda^\sigma(\theta^\star(\lambda))$,
its derivatives satisfy
\begin{equation}
 \frac{\partial\theta^\star}{\partial\lambda_j}
 =-H_\lambda^{-1}\bigl(\nabla F_j^\sigma(\theta^\star)-\nabla F_m^\sigma(\theta^\star)\bigr),\quad j<m.
 \label{eq:preference-path}
\end{equation}
The corresponding first-order predictor has parameter error
$O(\norm{\delta\lambda}^2)$ for nearby feasible preferences
and initializes the corrector used in E6.
The expansion assumes exact branch points and sensitivities. Numerical errors and singular-Hessian crossings require separate analysis.

\paragraph{Predictor--corrector use.}
For a feasible increment with $\mathbf1^\top\delta\lambda=0$, the
coordinate-free form is
\[
 \theta_{\rm pred}=\theta^\star(\lambda)
 -H_\lambda^{-1}\sum_i\delta\lambda_i\nabla F_i^\sigma(\theta^\star(\lambda)).
\]
At the new preference, rebuild the aggregate and apply the calibrated
corrector using the new true scalarization. The predictor accelerates
initialization along a regular branch. It does not certify branch
coverage or remove the need for a true-gradient stopping test.

\section{Extension to Bregman objectives}
\label{app:bregman}

Let $\phi$ be differentiable and strictly convex on an
open convex domain $\Omega$, with
$x_1,\ldots,x_m,y\in\Omega$ and $\lambda\in\Delta_m$.
For the dual-centroid identity, assume that
$\sum_i\lambda_i\nabla\phi(x_i)\in\nabla\phi(\Omega)$,
so that $\widetilde x_\lambda$ exists and is unique.

For $B_\phi(x,y)=\phi(x)-\phi(y)-\inner{\nabla\phi(y)}{x-y}$,
\begin{align}
 \sum_i\lambda_iB_\phi(x_i,y)
 &=C^-_\lambda+B_\phi(\bar x_\lambda,y),
 &\bar x_\lambda&=\sum_i\lambda_ix_i,\label{eq:breg-primal}\\
 \sum_i\lambda_iB_\phi(y,x_i)
 &=C^+_\lambda+B_\phi(y,\widetilde x_\lambda),
 &\nabla\phi(\widetilde x_\lambda)&=\sum_i\lambda_i\nabla\phi(x_i).
 \label{eq:breg-dual}
\end{align}
Expanding the definition proves~\eqref{eq:breg-primal}. Collecting the linear terms in $y$ proves~\eqref{eq:breg-dual}.  For a common-generator objective, choosing the exact aggregate, including its constant, as the predictor gives $\Acal=F$, $\Ecal=0$, and vanishing calibration tensors. Theorems~\ref{thm:global} and~\ref{thm:local} then transfer whenever their Euclidean assumptions hold.

For the variable-first heterogeneous objective $\sum_i\lambda_iB_{\phi_i}(v,a_i)$, define
\begin{equation}
 \Phi_\lambda=\sum_i\lambda_i\phi_i,
 \qquad b_\lambda=\sum_i\lambda_i\nabla\phi_i(a_i).
\end{equation}
If $\Phi_\lambda$ is Legendre and
$a_\lambda=(\nabla\Phi_\lambda)^{-1}(b_\lambda)$ exists, then
\begin{equation}
 \sum_i\lambda_iB_{\phi_i}(v,a_i)
 =C_\lambda+B_{\Phi_\lambda}(v,a_\lambda).
\end{equation}
Indeed, both sides equal $\Phi_\lambda(v)-\inner{b_\lambda}{v}$ up to a constant. A derivative-matching Bregman template can therefore replace the $\alpha_f$ module under the same regularity, model-decrease, and acceptance assumptions. For heterogeneous generators, this algebraic identity alone does not guarantee that evaluating $\Phi_\lambda$ is cheaper than evaluating its component sum.

\section{Implementation and computational scope}
\label{app:implementation}

\subsection{The computational regime}
Group queries by returned derivative order $r$ (values, gradients, or higher jets):
\begin{equation}
 T_{\rm CAL}=T_{\rm setup}
 +\sum_r n_r^F\bar c_r^F
 +\sum_r n_r^{\Acal}\bar c_r^{\Acal}
 +T_{\rm other}.
 \label{eq:cost-accounting}
\end{equation}
Here $n_r^F,n_r^{\Acal}$ count true-objective and predictor calls,
$\bar c_r^F,\bar c_r^{\Acal}$ are their mean costs, and $T_{\rm other}$
includes subproblem algebra and control overhead. Setup is charged once
per preference. A net gain requires true-oracle savings to exceed setup,
predictor-query, and model-management costs. E9 measures this crossover
as the target count grows.

\subsection{Nonlinear inner solves used in E8/E9}
The calibrated model is generally nonquadratic. The implementation
starts at $v=0$, shifts its current Hessian by
$\max\{0,10^{-7}-\lambda_{\min}(\nabla^2m(v))\}I$, and computes a
Newton direction, falling back to the negative model gradient if
needed. Trial iterates are projected onto the trust-region ball.
Armijo backtracking uses coefficient $10^{-4}$ and at most 22 trials.
The best model decrease is retained together with a steepest-descent
comparator. E8/E9 cap the inner loop at 12 iterations and use relative
gradient tolerance $0.01$ with absolute floor $10^{-9}$.
These settings target practical stationarity. E5 separately enforces
the stronger local forcing condition~\eqref{eq:inner}.

\subsection{A Cauchy comparator without a supplied uniform bound}
Write $g=\nabla m(0)\neq0$. The E8/E9 comparator initializes
\[
 \alpha_0=\min\left\{\frac{\Delta}{\norm g},
              \frac{1}{\max\{1,\norm{\nabla^2m(0)}_\op\}}\right\}
\]
and halves $\alpha$ while seeking
\begin{equation}
 m(-\alpha g)-m(0)\leq-\tfrac14\alpha\norm g^2.
 \label{eq:backtracking-comparator}
\end{equation}
The center Hessian supplies an initial scale, not a certified bound
over the ball. To connect this test to the theory, suppose
$\norm{\nabla^2m(v)}_\op\leq\kappa_H$ on the ball, taking
$\kappa_H\geq1$ without loss of generality. Taylor's inequality ensures
\eqref{eq:backtracking-comparator} whenever
$\alpha\leq3/(2\kappa_H)$. In exact arithmetic, uncapped halving
therefore terminates with
\[
 \alpha\geq\tfrac12\min\{\Delta/\norm g,1/\kappa_H\},\qquad
 m(0)-m(-\alpha g)\geq\tfrac18\norm g
                   \min\{\Delta,\norm g/\kappa_H\}.
\]
Any returned step with at least this model decrease satisfies
\eqref{eq:cauchy} with $\kappa_c=1/8$. The algorithm need not know the
numerical value of $\kappa_H$. The stored E8/E9 implementations cap
this comparator search at 40 trials and retain a decreasing candidate.
Finite caps and arithmetic floors are implementation safeguards, not
independent verification of the uniform-bound assumption. E4 uses its
stated conservative constant, while E5 uses an analytic ballwise bound.

\subsection{Parameter dimension and matrix-free extensions}
\label{app:scalability}
Target count $m$, parameter dimension $d$, and sample-space dimension $D$ impose distinct costs. Dense residual Hessians need $O(d^2)$ storage and dense Newton factorization $O(d^3)$ work. Order $p=3$ also requires third derivatives. Tensor storage is not inherent: for $p=2$,
\begin{equation}
 \nabla_v^2m_\theta^{(2)}(v)u
 =\nabla^2\Acal(\theta+v)u
  +[\nabla^2F(\theta)-\nabla^2\Acal(\theta)]u.
 \label{eq:matrix-free-model}
\end{equation}
Exact Hessian--vector products permit matrix-free inner solves under the same decrease and forcing conditions, but repeated true residual products can revisit all targets and negate caching savings. CAL-p1 permits a safeguarded first-order inner solver using a cached residual gradient. Such large-scale implementations are untested. Approximate derivatives need additional error controls.

\subsection{Integration dimension and sampled oracles}
Aggregation holds for integrable densities independently of discretization. For $\alpha\ne\pm1$, the forward predictor integral can use $\sum_i\lambda_ip_i^{r_\alpha}$ without normalizing the pseudo-target, but efficiency requires a reusable representation. Tensor grids have $k^D$ nodes for $k$ nodes per coordinate. Fresh Monte Carlo nodes may require reevaluating all targets. E8/E9 test two-dimensional grids with refinement audits, not high-dimensional integration scalability or the theorems' exact continuous oracles. Target minibatching also requires error controls for noisy derivatives and acceptance tests. Stochastic rates are not established here.

\section{Numerical experiments}
\label{sec:experiments}

The experiments test model fidelity (E1--E5, E7), local iteration error
(E10), total cost (E9), learned targets (E8), and repeated preferences
(E6). Table~\ref{tab:mechanism-main} summarizes the controlled evidence.
Appendix~\ref{app:numerics} gives the protocols. Forward/reverse denote
$\sigma=-/+$.

\subsection{Experimental overview and generator conventions}
\label{sec:generators}
The labels JS, KLChi, and KL4 denote
\begin{align*}
 f_{\mathrm{JS}}(u)
 &=u\log u-(u+1)\log\!\left(\frac{u+1}{2}\right),\\
 f_{\mathrm{KLChi}}(u)
 &=u\log u-u+1+\tfrac12(u-1)^2,\\
 f_{\mathrm{KL4}}(u)
 &=u\log u-u+1+(u-1)^4.
\end{align*}
Their $(a_f,\alpha_f)$ pairs are
$(1/2,0)$, $(2,-2)$, and $(1,-1)$, respectively.
JS uses the unhalved normalization:
$D_{f_{\mathrm{JS}}}(P\|Q)
=\KL(P\|M)+\KL(Q\|M)$, where $M=(P+Q)/2$.

\subsection{Why compression needs calibration}
\label{sec:mechanisms}

\paragraph{Geometry, exact aggregation, and step accuracy (E1--E3).}
E1 recovers residual orders $(4,3,2)$ near matched distributions,
supporting the canonical local geometry (Figure~\ref{fig:e1-scaling}).
E2 verifies 15,680 aggregation identities, with maximum scaled error
$3.77\times10^{-15}$. Its repeated-query benchmark measures the cost
of constructing and reusing an exact aggregate: at $m=256$,
end-to-end speedups are $2.46$--$3.14\times$, with break-even counts
of 5--100 queries (Figure~\ref{fig:e2-cost}). These are aggregation
microbenchmarks, rather than full optimization timings.

E3 recovers value/directional-derivative orders $(p+1,p)$ even with
median maximum absolute log-likelihood ratio $3.94$
(Figure~\ref{fig:e3-model-scaling}). Thus the predictor is exactly
compressible, and calibration controls parameter-step errors beyond
nearby targets.

\begin{figure}[!htbp]
  \centering
  \includegraphics[width=0.902\linewidth]{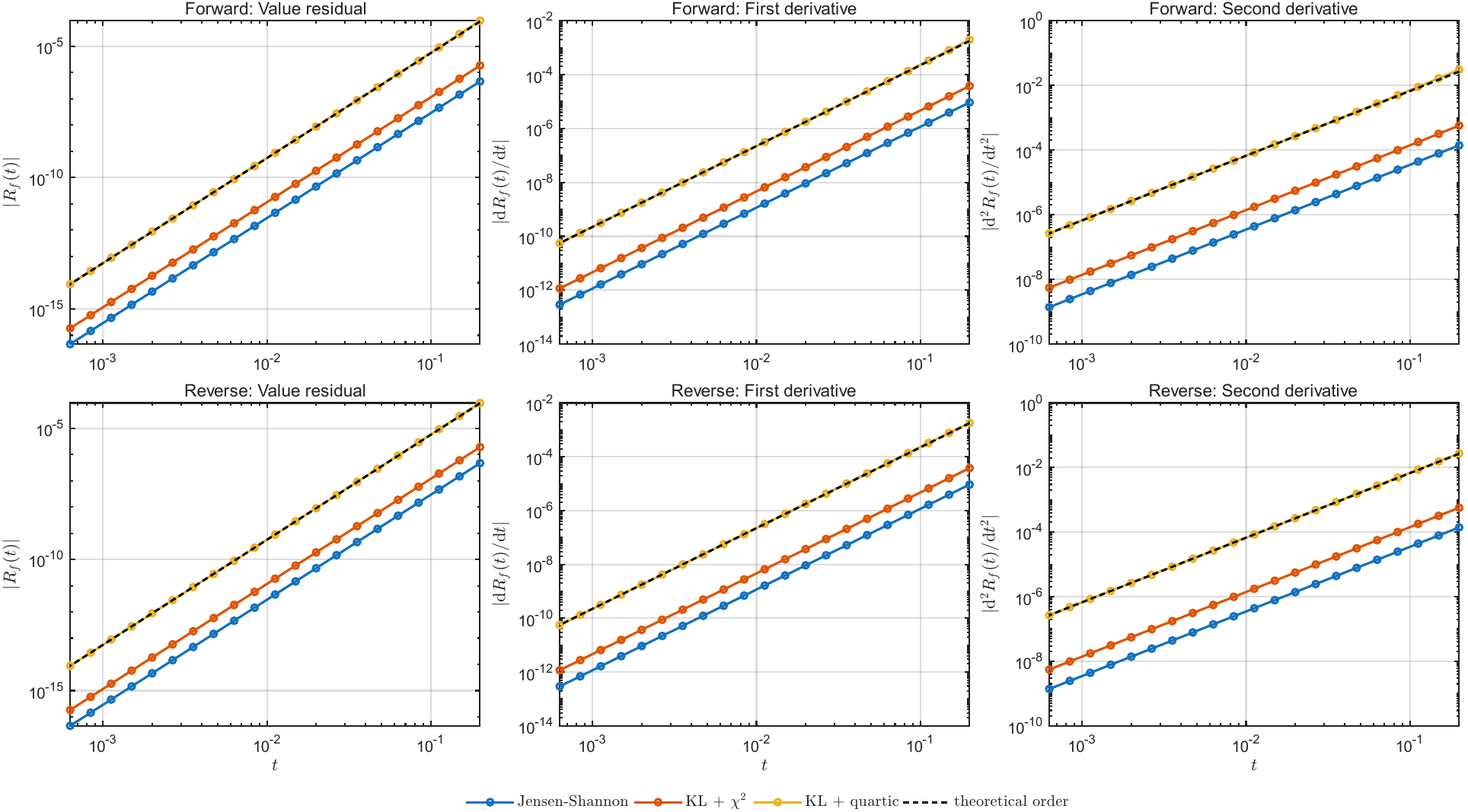}
  \caption{Local normal-form scaling in the forward and reverse orientations.  Curves show medians over 20 seeds. Dashed references have slopes $4$, $3$, and $2$.  Both axes are logarithmic.}
  \label{fig:e1-scaling}
\end{figure}

\begin{figure}[!htbp]
  \centering
  \includegraphics[width=0.902\linewidth]{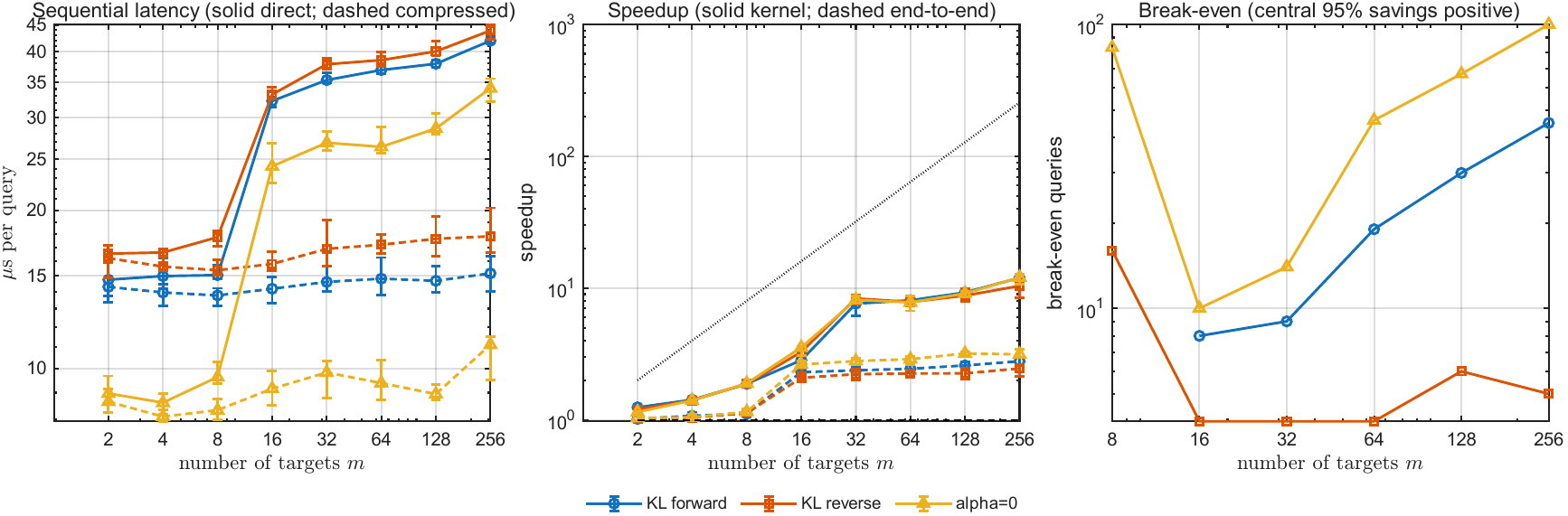}
  \caption{E2 amortized cost, $K=2048$, 30 paired repeats. Left: direct (solid) and aggregated (dashed) time per query, with IQRs. Middle: target-kernel (solid) and end-to-end (dashed) speedups, with $m$-fold and unit references. Right: break-even query counts. Gaps indicate a nonpositive 2.5th percentile of paired savings.}
  \label{fig:e2-cost}
\end{figure}

\begin{figure}[!htbp]
\centering
\includegraphics[width=0.902\linewidth]{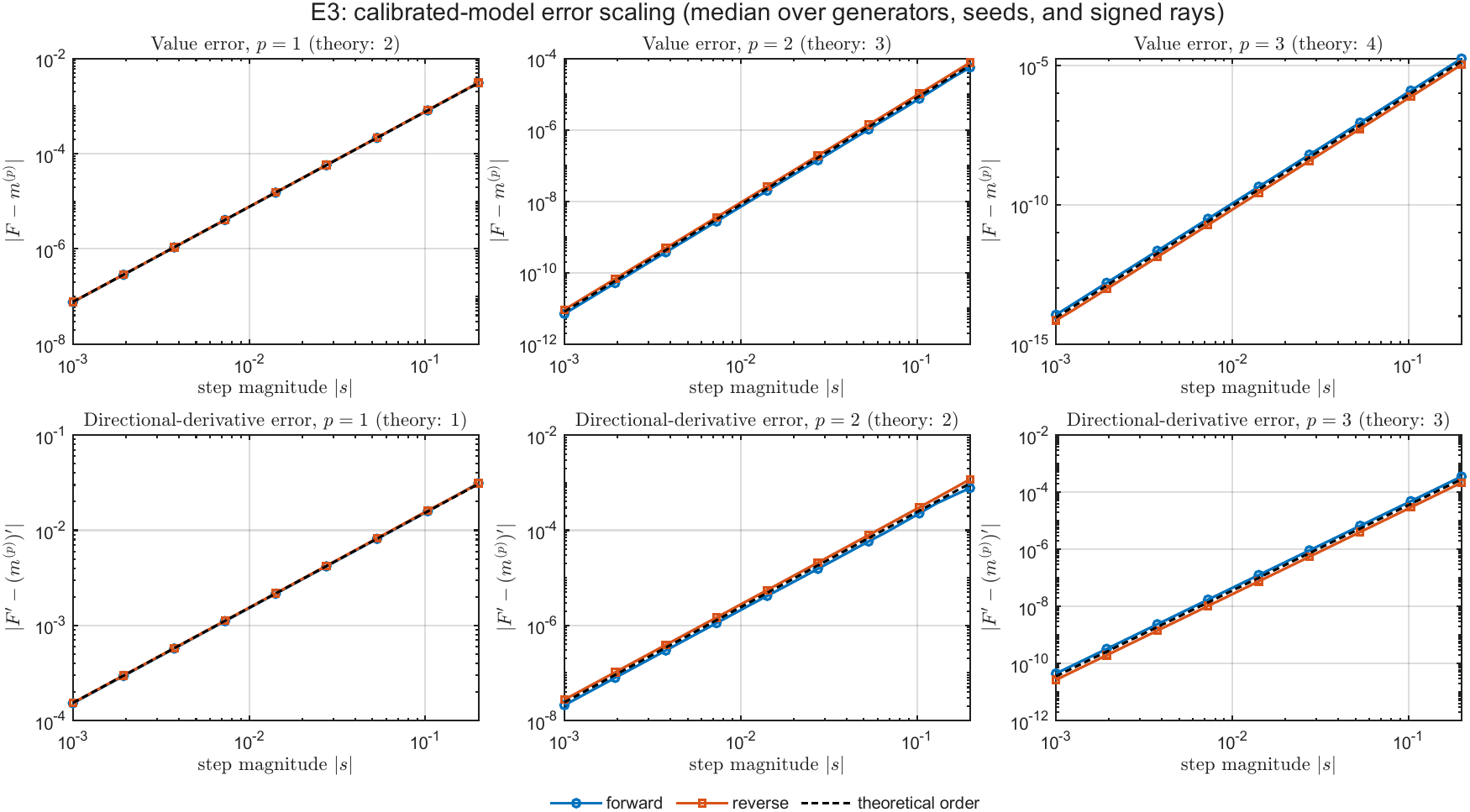}
\caption{Step-local calibrated-model error scaling.  Columns correspond to $p=1,2,3$. Rows show value and directional-derivative errors.  Curves are medians over three generators, 20 seeds, and two signed rays. Dashed references have slopes $p+1$ and $p$.}
\label{fig:e3-model-scaling}
\end{figure}

\paragraph{Why residual derivatives matter.}
In E4, each calibrated method meets the true-gradient tolerance in
1,800/1,800 runs, versus 180/1,800 for the uncalibrated predictor. In E7, the
corresponding counts are 480/480 and 80/480. Direct methods also succeed
throughout. Figure~\ref{fig:e4-robustness-cost} reports the E4 success and cost comparisons. These comparisons show why a compressible approximation
needs calibration to serve the specified divergence.

\begin{figure}[!htbp]
  \centering
  \includegraphics[width=0.902\linewidth]{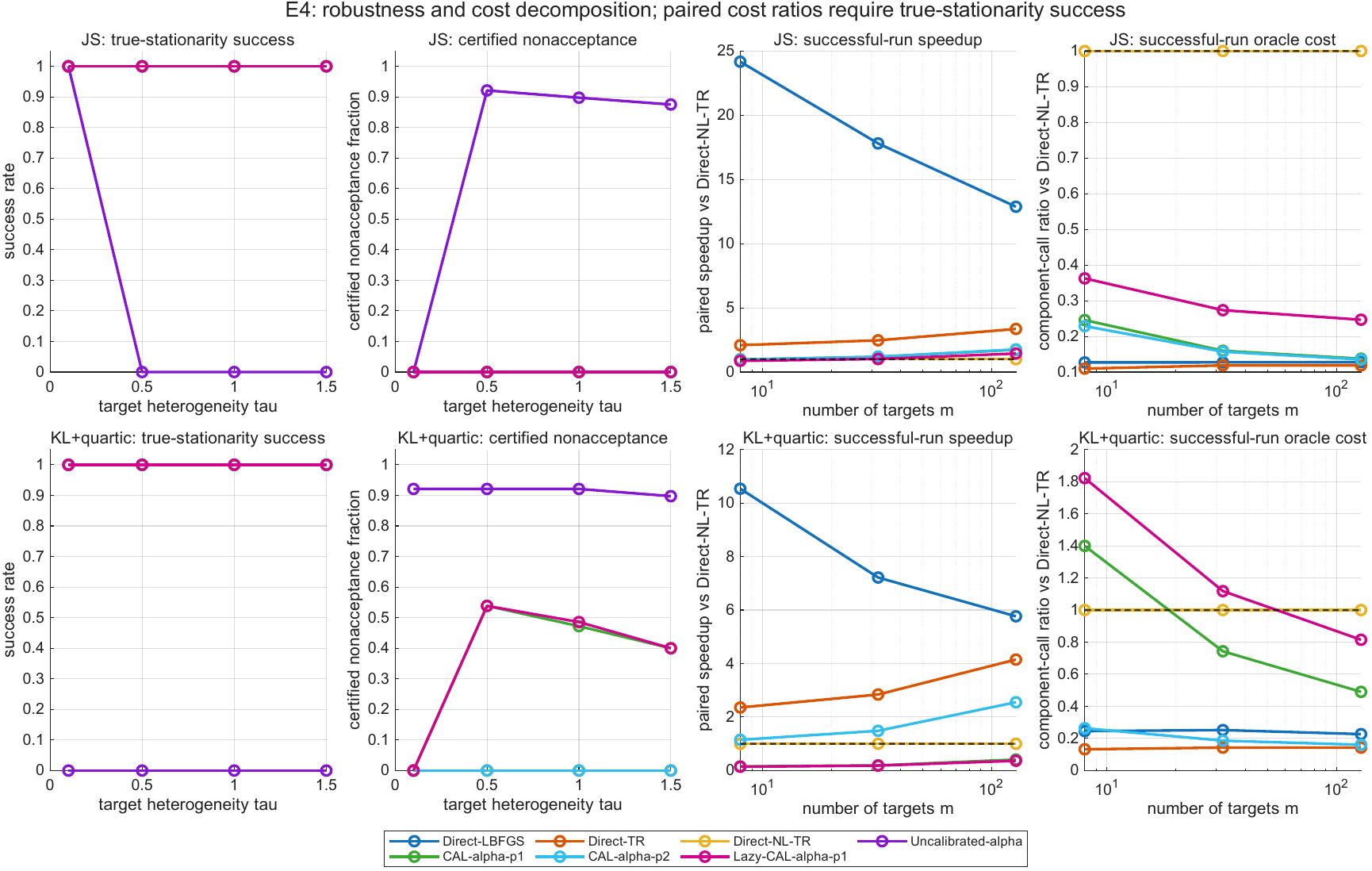}
  \caption{E4 robustness and cost, 30 paired seeds. Rows: JS and KL+quartic. Left: success rates and median certified nonacceptance fractions, including failed subproblems and rejected trials (not applicable to L-BFGS). Right: median paired runtime speedups and component-call ratios against Direct-NL-TR in the target-count sweep. Both runs must succeed.}
  \label{fig:e4-robustness-cost}
\end{figure}

\paragraph{Local order under controlled inner accuracy (E5).}
E5 uses a positive affine probability chart, ten seeds, two generators,
both orientations, and 120-digit arithmetic. Independent one-step
maps recover orders two and three for $p=2$ and $p=3$, respectively
(Figure~\ref{fig:e5-orders}). All 160 primary trajectories reach the
true-gradient tolerance. Condition ablations show why inner accuracy,
eventually inactive truncation, and nondegeneracy matter. The full
protocol and actual-iterate orders are in Appendix~\ref{app:e5-details}.

\begin{figure}[!htbp]
  \centering
  \includegraphics[width=0.798\linewidth]{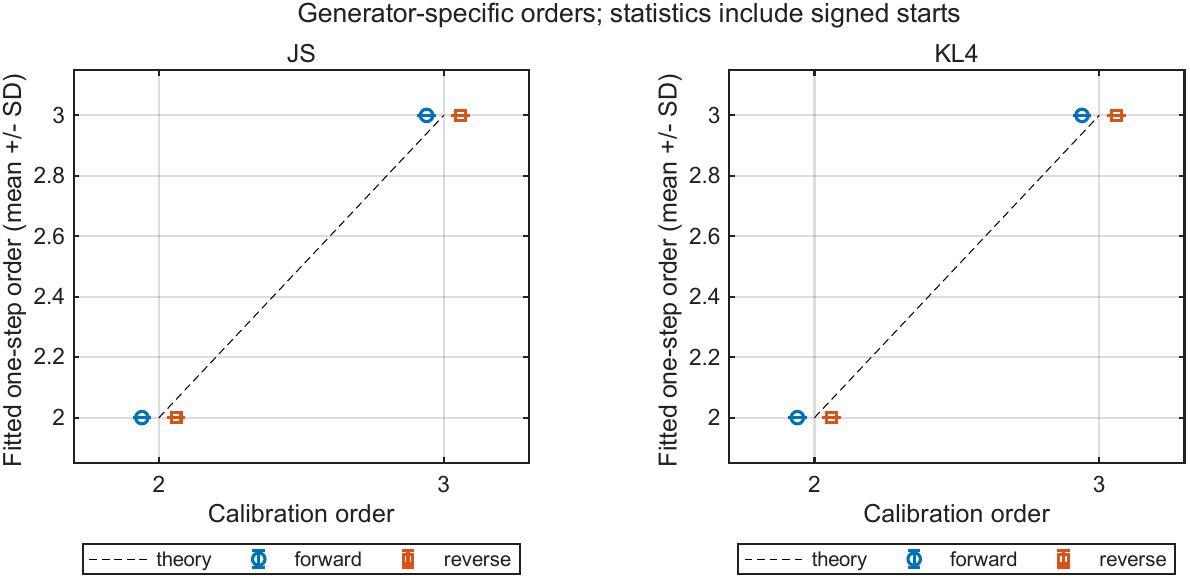}
  \caption{E5 one-step orders: means and one standard deviation over ten seeds and two paired signs (bars smaller than markers, not confidence intervals). Dashed lines mark order $p$. Actual-iterate orders appear in Table~\ref{tab:e5-iterate-orders}.}
  \label{fig:e5-orders}
\end{figure}

\paragraph{Template choice at fixed calibration order (E7).}
E7 holds the matched quadratic jet fixed and varies the nonlinear
template. Canonical geometry improves the displayed finite-step
model errors in several regimes (Figure~\ref{fig:e7-gradient}), but
does not guarantee the smallest error for every target configuration.
All calibrated models and Taylor-p2 attain true-gradient stationarity
in 480/480 cases. Canonical CAL's group-median paired speedups are
$0.470$--$0.816$ relative to Taylor-p2, so these accuracy improvements
do not imply lower runtime in this experiment.

\begin{figure}[!htbp]
  \centering
  \includegraphics[width=0.874\linewidth]{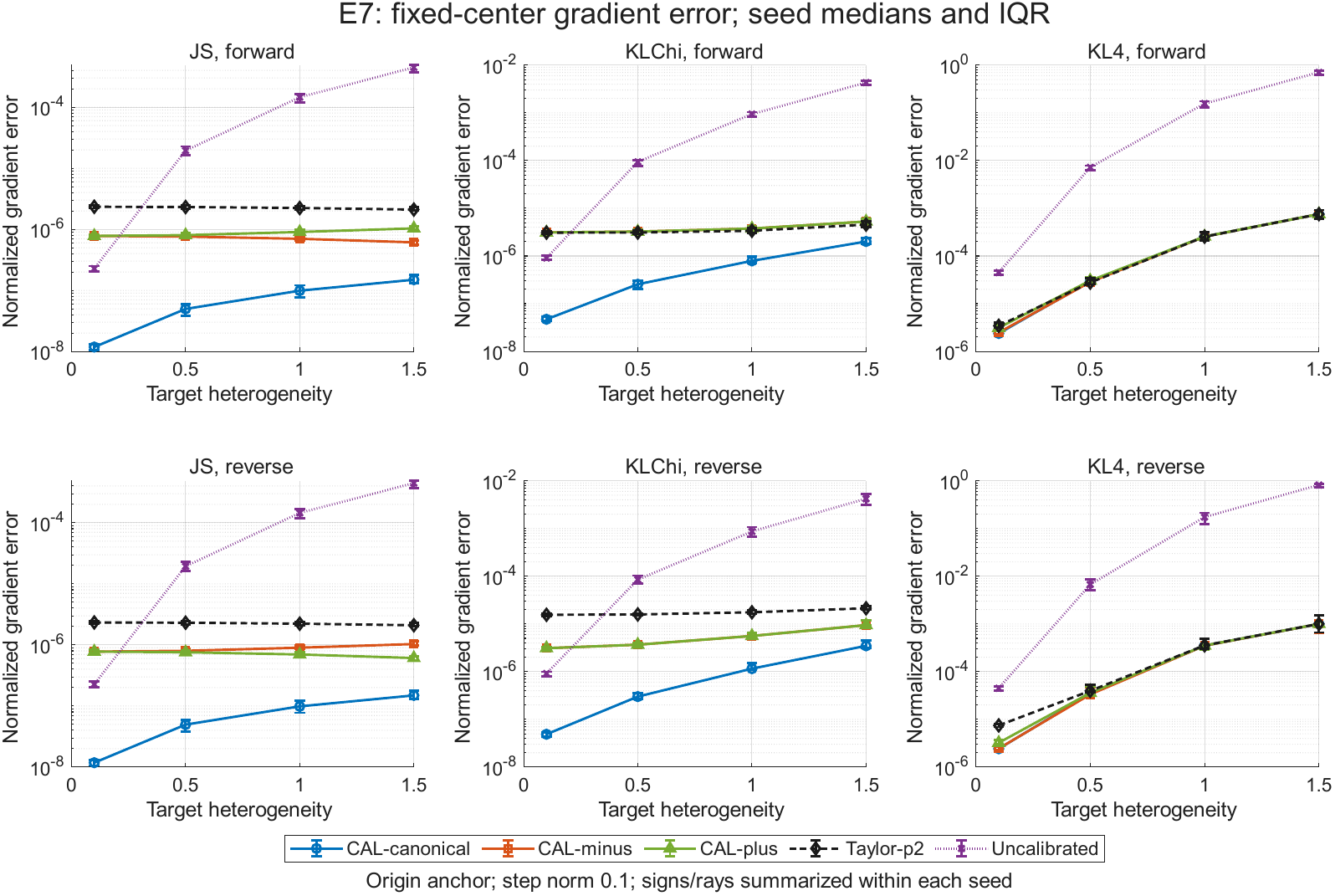}
  \caption{E7 normalized gradient errors~\eqref{eq:e7-errors} at $\theta=0$, $\norm s=0.1$. Directions and signs are summarized within seeds, then medians/IQRs across 20 seeds. CAL-minus/plus use $\alpha_f\mp1$. All calibrated models and Taylor-p2 match through order two. JS orientations coincide.}
  \label{fig:e7-gradient}
\end{figure}

\subsection{Canonical geometry and local iteration error (E10)}
\label{sec:e10}
E10 compares canonical CAL-p2, shifts $\alpha_f\pm1$, and Taylor-p2,
which share the true value, gradient, and Hessian at each anchor.
Five seeds, forward JS and KL+quartic, and four heterogeneity levels
$\eta\in\{0,.01,.05,.2\}$ yield 1,920 accepted, audited local maps.
All methods use the same cubic inner forcing and 120-digit arithmetic.
Third derivatives are evaluated only offline. At common matching
($\eta=0$), canonical CAL-p2 has median fitted order $3.000$ for both
generators, versus $2.000$ for all controls. Subtracting the predicted
quadratic error in~\eqref{eq:leading-local-error} leaves a cubic
remainder for every method, directly testing the coefficient formula.

For $\eta>0$, fitted orders return to approximately two. Canonical CAL-p2
has smaller one-step errors than Taylor-p2 in all 360 paired probes.
These tests measure local accuracy at fixed calibration order, not runtime
or a universal ranking of templates (Appendix~\ref{app:e10-details}).

\begin{figure}[!htbp]
 \centering
 \begin{tabular}{@{}cc@{}}
 \includegraphics[width=.38\linewidth,viewport=0 190 190 381,clip]{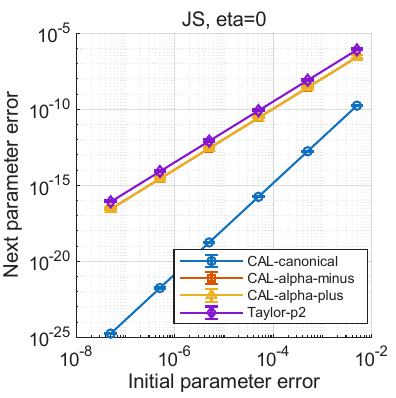}&
 \includegraphics[width=.38\linewidth,viewport=381 190 571 381,clip]{E10_one_step_scaling.pdf}\\[3pt]
 \includegraphics[width=.38\linewidth,viewport=0 0 190 189,clip]{E10_one_step_scaling.pdf}&
 \includegraphics[width=.38\linewidth,viewport=381 0 571 189,clip]{E10_one_step_scaling.pdf}
 \end{tabular}
 \caption{E10 one-step errors, seed medians/IQRs, with identical $p=2$ calibration and cubic inner forcing. Canonical CAL-p2 is cubic at $\eta=0$ and quadratic for positive heterogeneity. Panels show $\eta=0,0.05$ and all six radii. Fits and audits use all four heterogeneities.}
 \label{fig:e10-maps}
\end{figure}

\par\medskip\noindent
\begin{minipage}[t]{0.52\linewidth}
\vspace{0pt}
\textbf{Error coefficients.}
Table~\ref{tab:e10-coefficients} compares the leading directional
errors in~\eqref{eq:leading-local-error}. At $\eta=0.05$, canonical
CAL-p2 reduces the seed-median coefficient by a factor of $2.37\times10^3$
for JS and $31.8$ for KL4. The models have identical center values,
gradients, and Hessians. Their different nonlinear remainders account for
the change in local error. E7 tests finite-step model accuracy, and E9
examines whether aggregation lowers runtime.
\end{minipage}\hfill
\begin{minipage}[t]{0.44\linewidth}
\vspace{0pt}
\captionsetup{font=small,hypcap=false,justification=raggedright,singlelinecheck=false}
\captionof{table}{E10 directional error ratio
$\norm{b_{\rm Taylor}(d_s)}/\norm{b_{\rm CAL}(d_s)}$.
Medians over five seeds. Values above one favor canonical CAL-p2.
The denominator vanishes at $\eta=0$.}
\label{tab:e10-coefficients}
\centering\small
\setlength{\tabcolsep}{7pt}
\begin{tabular}{@{}lrr@{}}
\toprule
$\eta$ & JS & KL4\\
\midrule
$0.01$ & $12{,}793.3$ & $177.6$\\
$0.05$ & $2{,}369.0$ & $31.8$\\
$0.2$ & $514.5$ & $2.77$\\
\bottomrule
\end{tabular}
\end{minipage}
\par\medskip

\begin{table}[htbp]
 \centering\small
 \caption{Controlled evidence. Local rates require the stated neighborhood and inner-accuracy conditions. The E10 cubic result is at common matching.}
 \label{tab:mechanism-main}
 \begin{tabularx}{\linewidth}{@{}l X l@{}}
 \toprule
 Claim & Observation & Experiment\\
 \midrule
 Exact aggregation & Max. scaled identity error $3.77\times10^{-15}$ & E2\\
 Step-based accuracy & Value/directional-derivative exponents $(p+1,p)$ & E3\\
 Calibration fidelity & Higher true-gradient success rates than uncalibrated models & E4, E7\\
 Canonical cancellation & CAL-p2 order $3.000$; Taylor-p2 order $2.000$ & E10\\
 Quadratic local iteration & Mean fitted orders $1.993$--$2.014$ & E5, $p=2$\\
 Cubic local iteration & Mean fitted orders $2.994$--$2.996$ & E5, $p=3$\\
 \bottomrule
 \end{tabularx}
\end{table}

\subsection{When compression reduces total cost (E9)}
\label{sec:e9}

\paragraph{Design and cost accounting.}
E9 forms nested banks of $m\in\{3,10,30,100,300\}$ Gaussian-mixture
targets by perturbing three fitted latent targets. Five seeds, five
positive preferences, and forward JS or KL are used. The model has two
Gaussian components and five free parameters. CAL-p1/p2 are compared
with direct L-BFGS and a quadratic true-objective trust-region solver
(Direct-TR). All methods share target caches, vectorized evaluations,
initial points, and stationarity thresholds. The direct KL baselines
use exact aggregation, as does the predictor.

Total time includes scalarization setup, calibration, inner solves, and
true-objective acceptance. Common data preparation and independent
audits are excluded. Three unprofiled repeats are summarized within
each case. Timing pairs must be stationary and have objective values
within $2\times10^{-6}$. Ratios are summarized over preferences within
seeds, then by medians and IQRs across five seeds. All 1,000 primary
runs pass the stationarity audit. Different local solutions account
for incomplete timing-pair coverage.

\begin{figure}[t]
 \centering
 \includegraphics[width=.48\linewidth,viewport=0 37 275 240,clip]{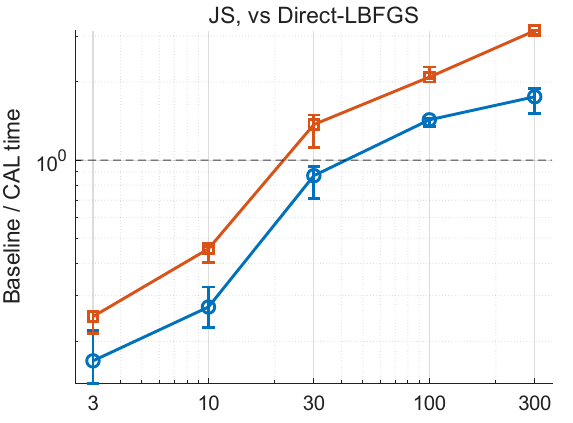}\hfill
 \includegraphics[width=.48\linewidth,viewport=0 280 275 483,clip]{E9_paired_speedup.pdf}
 \par\smallskip
 {\small Number of targets $m$ \quad
 \textcolor[rgb]{0,0.447,0.741}{\(\circ\) CAL-p1}\quad
 \textcolor[rgb]{0.850,0.325,0.098}{\(\square\) CAL-p2}}
 \caption{E9: measured JS acceleration relative to direct L-BFGS (left) and Direct-TR (right). Values above one favor CAL. Points and bars are seed medians and IQRs of objective-comparable successful pairs. These are the JS panels of the complete comparison in Figure~\ref{fig:e9-all-speedups}. Table~\ref{tab:e9-main} reports CAL-p2 pair coverage.}
 \label{fig:e9-main}
\end{figure}

\par\medskip\noindent
\begin{minipage}[t]{0.51\linewidth}
\vspace{0pt}
\textbf{Cost crossover.}
CAL-p2 becomes faster than L-BFGS at $m=30$.
At $m=300$, median paired speedups are $3.15\times$ over L-BFGS and
$1.10\times$ over Direct-TR (Figure~\ref{fig:e9-main}).
CAL-p2 uses median counts of five true value/gradient/Hessian calls
and four additional value calls, compared with six and five for
Direct-TR. The aggregate handles the repeated inner queries.
Pairs reaching different local solutions are excluded from timing
ratios but retained in the endpoint audits and reference-gap plots.
\end{minipage}\hfill
\begin{minipage}[t]{0.45\linewidth}
\vspace{0pt}
\captionsetup{hypcap=false,justification=raggedright,singlelinecheck=false}
\captionof{table}{E9 CAL-p2 speedups on JS.
Parentheses give eligible pairs out of 25. All runs pass the stationarity
audit. Values below one favor the baseline.}
\label{tab:e9-main}
\centering\small
\setlength{\tabcolsep}{5pt}
\begin{tabular}{@{}rcc@{}}
\toprule
$m$ & vs L-BFGS & vs Direct-TR\\
\midrule
3 & $0.25$ (22) & $0.48$ (24)\\
10 & $0.45$ (22) & $0.50$ (21)\\
30 & $1.37$ (25) & $0.72$ (24)\\
100 & $2.09$ (24) & $0.90$ (24)\\
300 & $3.15$ (23) & $1.10$ (24)\\
\bottomrule
\end{tabular}
\end{minipage}
\par\medskip

\paragraph{Why the gain appears.}
After aggregation, JS predictor cost is essentially independent of $m$, while the implemented true JS oracle processes all targets. At $m=300$, measured true-to-predictor cost ratios are $188.8$ for values, $68.1$ for value/gradient pairs, and $8.79$ for value/gradient/Hessian evaluations (Figure~\ref{fig:e9-cost-ratio}). This separation offsets calibration and inner-solver overhead. An implementation ablation keeps the optimization rules fixed and verifies agreement of acceptance decisions and endpoints while requesting only the needed derivative order, confirming the cost of unnecessary template Hessians (Figure~\ref{fig:e9-implementation}).

\begin{figure}[t]
 \centering
 \includegraphics[width=.86\linewidth]{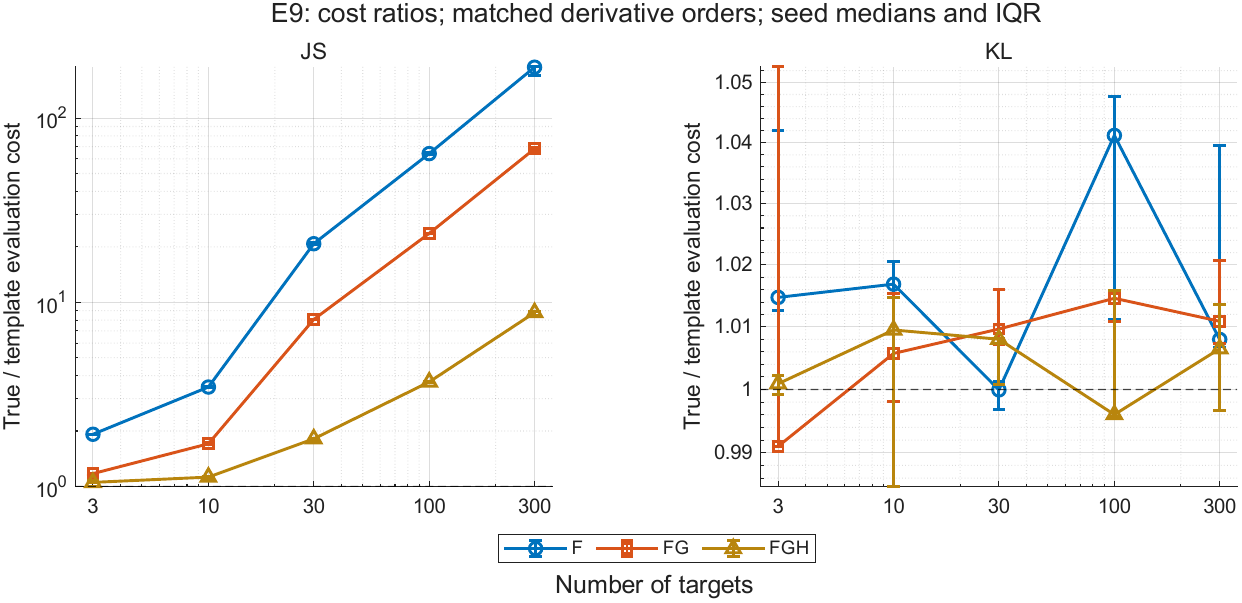}
 \caption{E9: true-objective cost divided by aggregated-predictor cost, at matched derivative orders. F, FG, and FGH include values, values/gradients, and values/gradients/Hessians. The JS cost separation grows with the target count. KL stays near one because both implementations aggregate exactly. Points and bars are seed medians and IQRs across five seeds.}
 \label{fig:e9-cost-ratio}
\end{figure}

\paragraph{Derivative-demand control.}
For the original three-target problem, we compare eager template evaluation (computing a Hessian at every template query) with demand-driven evaluation that computes only the requested derivative order. Optimization rules, tolerances, and initialization are held fixed. The ablation uses five seeds, two representative preferences, and CAL-p1/p2.  Outer acceptance decisions agree, and endpoint parameter/objective differences pass tolerances $10^{-8}$/$10^{-10}$. Figure~\ref{fig:e9-implementation} shows the timing and template-Hessian savings.

\begin{figure}[!htbp]
 \centering
 \includegraphics[width=0.779\linewidth]{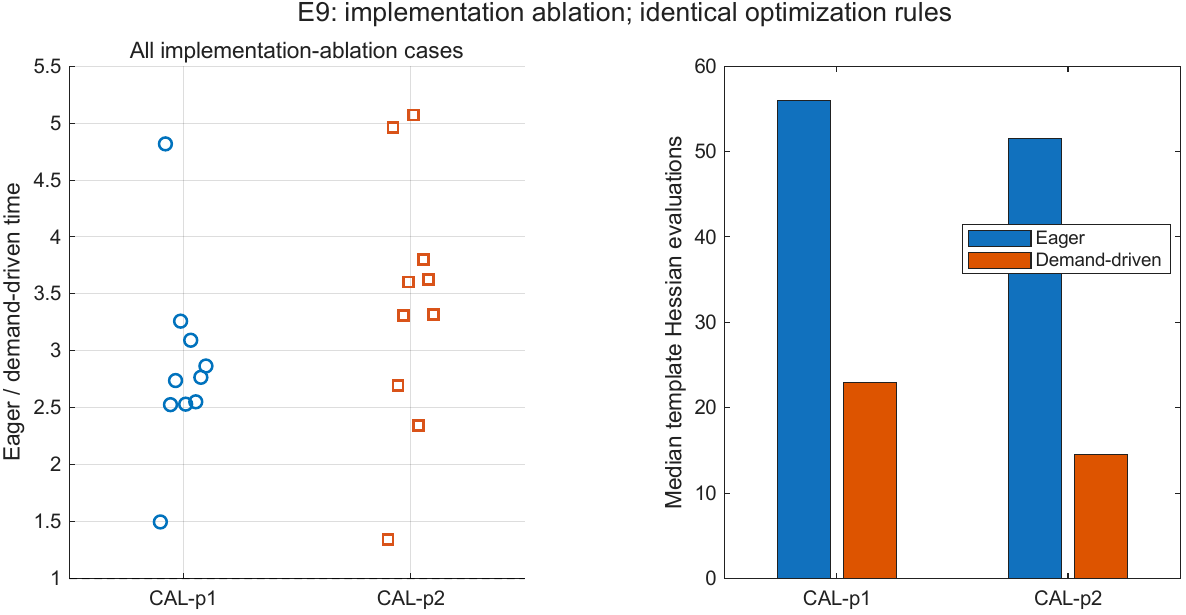}
 \caption{E9 implementation ablation with identical optimization rules. Left: eager/demand-driven time ratios for CAL-p1/p2. Right: median template-Hessian counts. This compares two implementations of the same method, not CAL against a direct optimizer.}
 \label{fig:e9-implementation}
\end{figure}

Direct-TR matches CAL-p2's true derivative order. The larger speedup
over L-BFGS also reflects different use of curvature. CAL-p2 is slower
than Direct-TR through $m=100$. At $m=300$, the JS speedup IQRs are
$[3.06,3.15]$ against L-BFGS and $[1.10,1.15]$ against Direct-TR.
The solve time includes median counts of 56 predictor-value calls and
14 predictor-Hessian calls. Target count alone does not ensure savings:
E4's direct solvers remain faster at $m=128$ (Table~\ref{tab:e4-cost}).
The evidence supports a cost crossover at comparable attained objectives,
without implying better minima or faster exact-KL optimization.

\subsection{Learned targets and decoded trade-offs (E8)}
\label{sec:e8}

E8 uses MNIST digits 0, 1, and 2 with a trained autoencoder whose
two-dimensional encoder and decoder are then frozen. Classwise Gaussian
mixtures fitted to training encodings define $R_0,R_1,R_2$. Three targets
mix them with weights $(.7,.2,.1)$, $(.1,.7,.2)$, and $(.2,.1,.7)$.
A five-parameter, two-component latent mixture $Q_\theta$ matches these
targets under forward JS or KL, and its samples pass through the fixed
decoder. We distinguish optimization fidelity from the observable
consequences of changing preferences.

\begin{figure}[!htbp]
 \centering\small
 \setlength{\tabcolsep}{5pt}
 \begin{tabular}{@{}cccc@{}}
 $(.33,.33,.33)$ & $(.90,.05,.05)$ & $(.05,.90,.05)$ & $(.05,.05,.90)$\\[4pt]
 \includegraphics[width=.21\linewidth,viewport=145 344 188 386,clip]{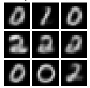}&
 \includegraphics[width=.21\linewidth,viewport=313 344 357 386,clip]{\EeightSamplesPDF}&
 \includegraphics[width=.21\linewidth,viewport=482 344 526 386,clip]{\EeightSamplesPDF}&
 \includegraphics[width=.21\linewidth,viewport=650 344 694 386,clip]{\EeightSamplesPDF}
 \end{tabular}
 \caption{E8: saved decoded mean images for CAL-p2, JS, seed 1, with a fixed decoder. Each preference $\lambda$ has nine latent draws, without reranking. All-method samples and complete held-out image diagnostics are in Appendix~\ref{app:e8-details}.}
 \label{fig:e8-main}
\end{figure}

\paragraph{Optimization fidelity.}
Success requires the refined-grid true-gradient norm to be below
$5\times10^{-6}$ and passing objective, gradient, and mass checks.
Across five seeds and ten preferences, all 200 calibrated runs and all
200 direct-baseline runs pass. The uncalibrated predictor passes 0/50
JS cases, despite plausible decoded images, and 50/50 exact-KL controls.
Maximum objective and gradient discrepancies under grid refinement are
below $1.86\times10^{-8}$ and $1.14\times10^{-7}$, respectively
(Figure~\ref{fig:e8-audits}).

\begin{figure}[!htbp]
 \centering
 \includegraphics[width=0.9\linewidth]{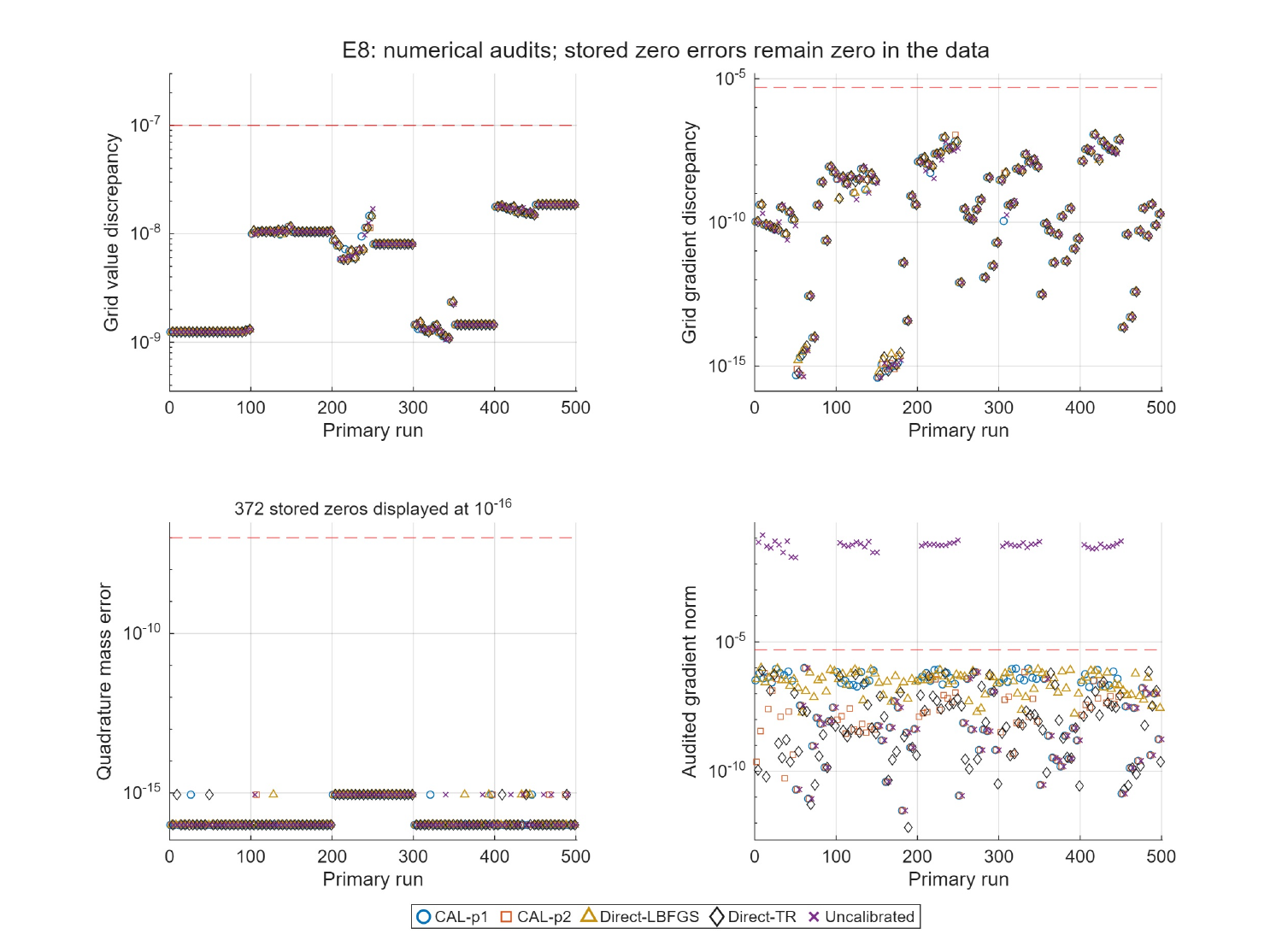}
 \caption{E8 quadrature and endpoint audits. Dashed lines mark the configured thresholds. The mass panel reports normalization error of the discrete masses. Stored zeros are displayed at $10^{-16}$ without changing the data. Large true gradients of uncalibrated JS endpoints explain their failed audits.}
 \label{fig:e8-audits}
\end{figure}

\paragraph{Preference response.}
Changing $\lambda$ changes which learned target is prioritized
(Figure~\ref{fig:e8-tradeoffs}). These weights specify divergence
objectives, not generated class proportions.

\begin{figure}[!htbp]
 \centering
 \includegraphics[width=0.864\linewidth]{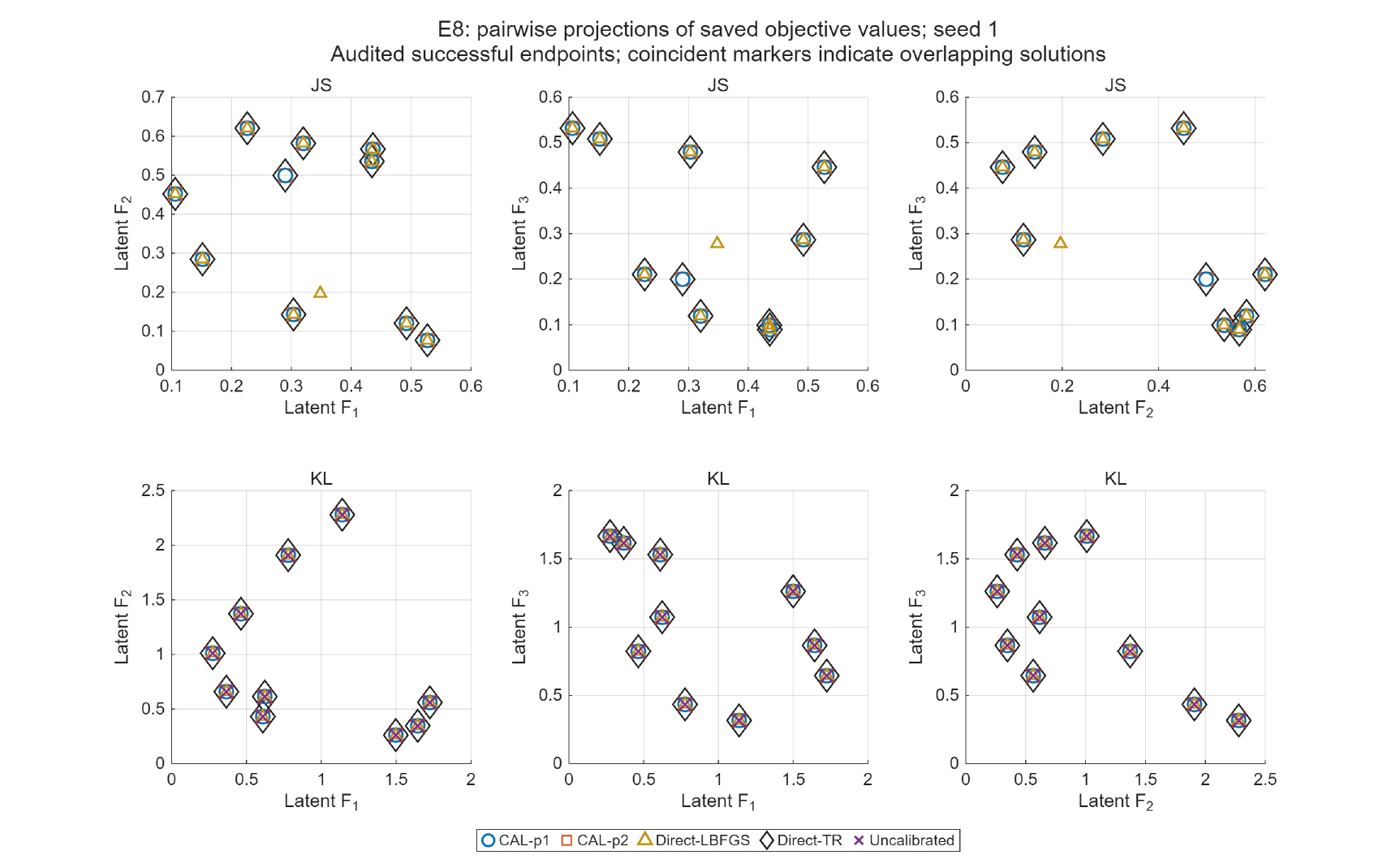}
 \caption{E8 pairwise projections of latent objectives at audited successful endpoints, seed 1. Coincident markers indicate overlapping solutions. The points are sampled stationary trade-offs, not a certified full Pareto front.}
 \label{fig:e8-tradeoffs}
\end{figure}

\paragraph{Decoded consequences.}
Figure~\ref{fig:e8-main} shows the changes in decoded digits.
Independent class proportions and image-space MMD$^2$
(Figure~\ref{fig:e8-images}) describe the resulting distributions.
They do not enter the optimized objective or certify stationarity.
The targets have six Gaussian components each, while the model has two
with a fixed shared covariance. Thus the aggregate need not be representable
by the model. E8 validates a restricted low-dimensional fit, without
establishing better image quality or faster three-target optimization.

\begin{figure}[!htbp]
 \centering
 \includegraphics[width=0.9\linewidth]{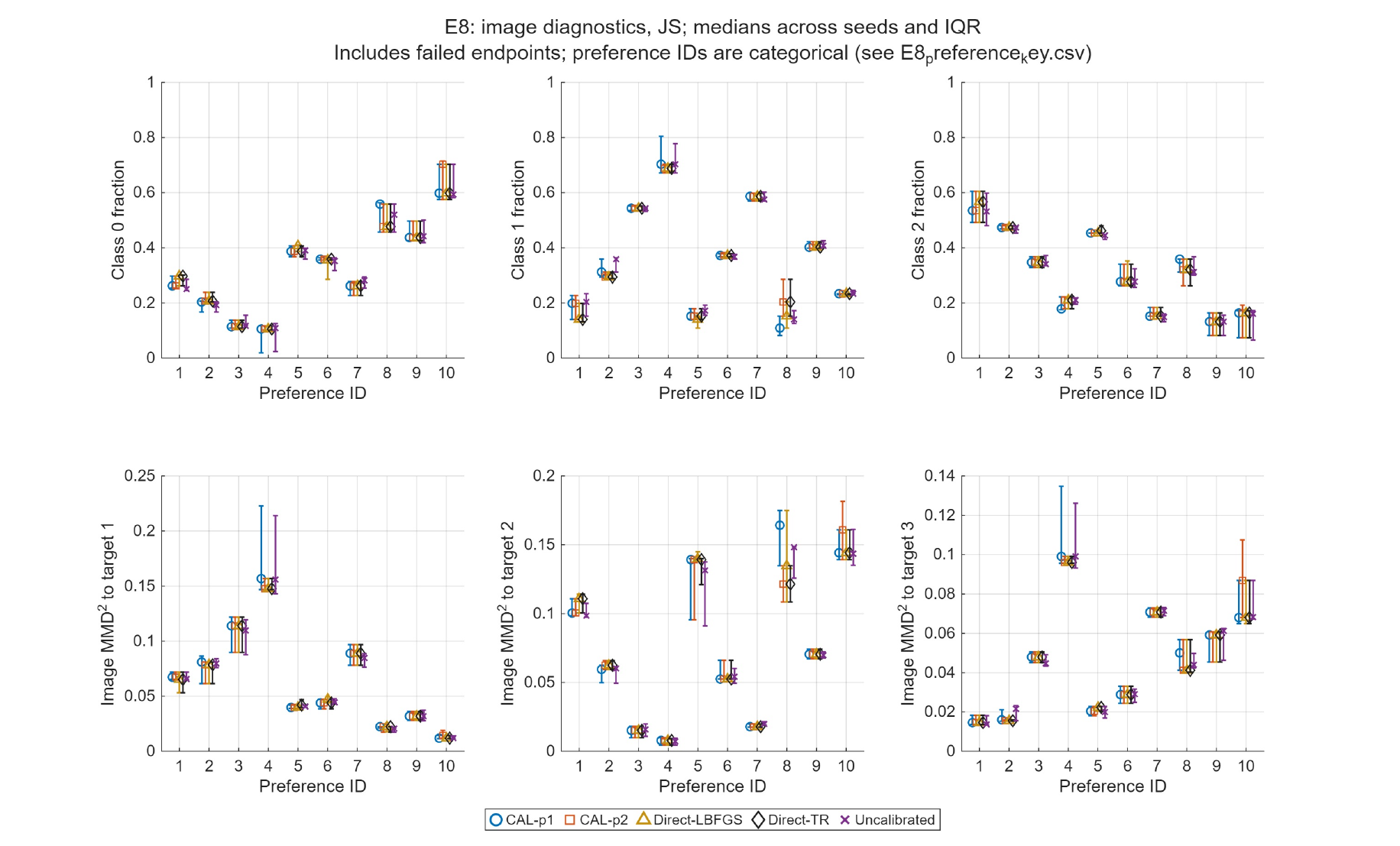}
 \caption{E8 held-out image diagnostics: class fractions and image MMD$^2$ under JS, summarized by seed medians and IQRs. Preference IDs are categorical and follow the saved preference key. Failed endpoints are included. These quantities describe decoded distributions rather than certify optimality.}
 \label{fig:e8-images}
\end{figure}

\subsection{Complementary acceleration across preferences (E6)}
\label{sec:e6}

Standard implicit differentiation provides a complementary initializer
for repeated preferences, independently of $\alpha$-aggregation
(Section~\ref{app:continuation}). E6 compares Cold, Warm, and Tangent
initializations with the same CAL-p2 corrector on categorical JS and
KL+quartic problems, both orientations, and ten seeds. All 20,880 point
solves meet the true-gradient tolerance $10^{-8}$, with maximum reference
objective discrepancy $9.65\times10^{-15}$. At 49 preferences, Tangent
has median paired speedups $1.26$--$1.59\times$ over Warm and
$2.03$--$2.10\times$ over Cold, including aggregation, tangent solves,
and correction. Reusing regular stationary branches thus reduces the
work of the calibrated corrector (Table~\ref{tab:e6-cost}).

\begin{table}[!htbp]
\centering\small
 \setlength{\abovecaptionskip}{0pt}
 \caption{E6: Tangent costs at 49 preferences; paired seed medians. Timings include aggregation, initialization, and correction. F/R: forward/reverse; KL4: KL+quartic.}
 \label{tab:e6-cost}
 \setlength{\tabcolsep}{4pt}
 \begin{tabular}{@{}lccc@{}}
 \toprule
 & \multicolumn{2}{c}{Time speedup} & True-jet ratio\\
 \cmidrule(lr){2-3}
 Case & vs Cold & vs Warm & vs Cold\\
 \midrule
 JS, F  & $2.099$ & $1.590$ & $0.581$\\
 JS, R  & $2.044$ & $1.474$ & $0.581$\\
 KL4, F & $2.087$ & $1.376$ & $0.562$\\
 KL4, R & $2.026$ & $1.263$ & $0.516$\\
 \bottomrule
 \end{tabular}
\end{table}

\par\medskip
Along a regular branch, Tangent has $O(\norm{\delta\lambda}^2)$
initialization error, while Warm generally has first-order error.
Independent E6 probes recover these orders
(Figure~\ref{fig:e6-prediction}). The same corrector and
stopping test are used throughout (Appendix~\ref{app:e6-details}).

\begin{figure}[!htbp]
  \centering
  \includegraphics[width=0.817\linewidth]{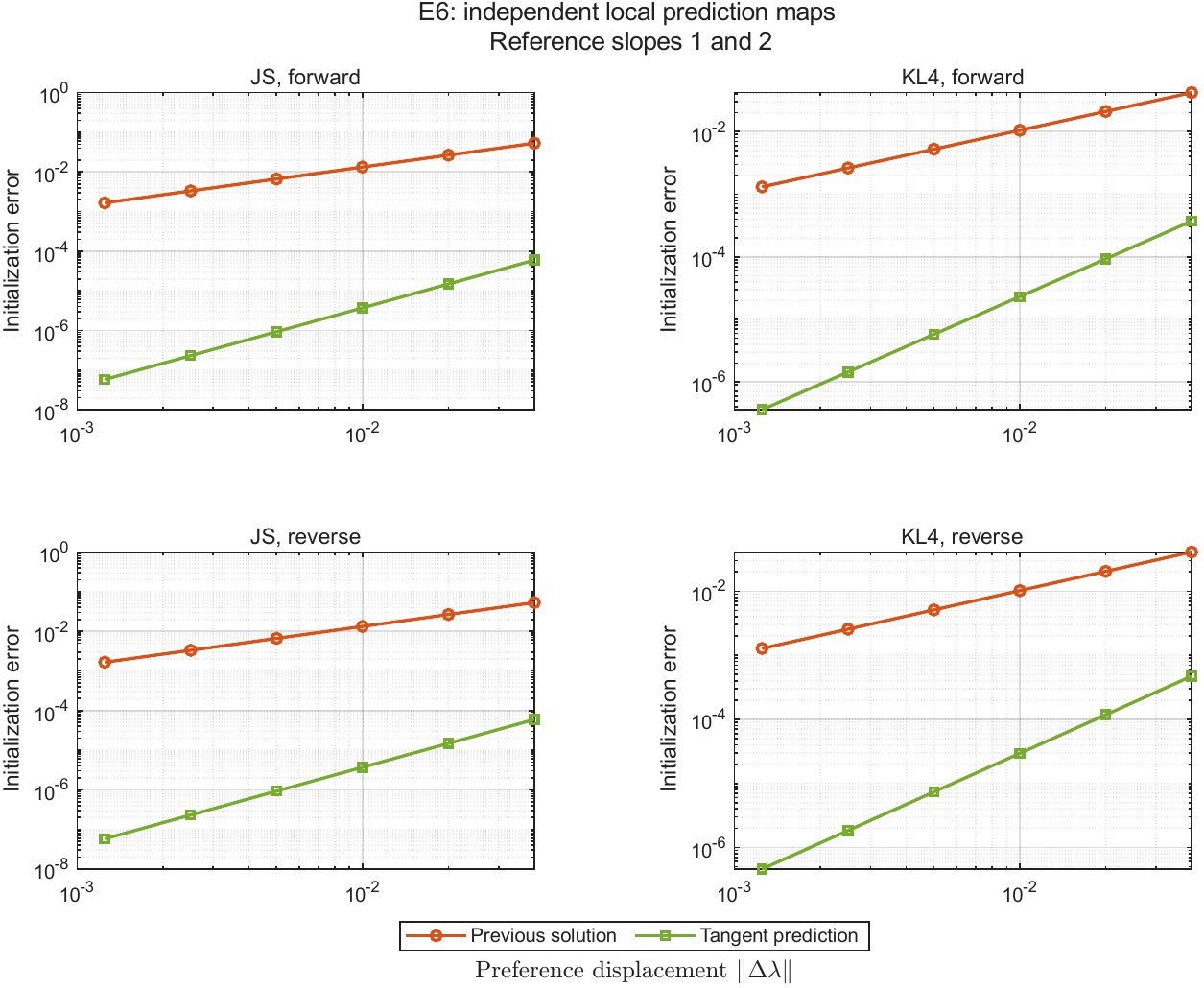}
  \caption{E6 independent prediction errors: medians over signs within each seed, then over ten seeds. Warm uses the base reference. Tangent adds its sensitivity prediction. Dashed slopes one and two nearly overlap the curves. The horizontal axis measures preference displacement, not corrector iteration.}
  \label{fig:e6-prediction}
  \label{fig:e6-main-prediction}
\end{figure}

\FloatBarrier
\section{Discussion and conclusion}
\label{sec:conclusion}

CAL-$\alpha$-MODM retains a compressible nonlinear remainder beyond a
matched Taylor jet. Calibration controls parameter-step error, and
true-objective tests certify progress. The stationarity and order-$p$
local guarantees hold under the stated assumptions without target
proximity. Canonical cancellation further improves the order-two model
at common matching under the stronger inner-accuracy condition.

\paragraph{Computational scope.}
E9 identifies a regime where cheaper predictor queries offset calibration
overhead. E6 adds initialization savings for repeated preferences.
When the true objective aggregates exactly, as in the KL control,
residual calibration is unnecessary.

\paragraph{Why latent matching is meaningful.}
Let $K$ be a fixed decoding map or a common stochastic decoder, and
write $P_i^Z,Q_\theta^Z$ for latent distributions, with $KP$ denoting
the decoded law of $P$. Data processing for
$f$-divergences \citep{csiszar1967} gives, in either orientation,
\[
 D_f^\sigma(KP_i^Z,KQ_\theta^Z)
 \leq D_f^\sigma(P_i^Z,Q_\theta^Z).
\]
Thus the exact latent scalarization upper-bounds the weighted decoded
divergences. This bound does not transfer stationarity to the decoded
objective. In E8, the decoded target laws come from fitted latent
mixtures, not empirical MNIST distributions. Independent image
diagnostics assess agreement with held-out image targets.

The demonstrated scaling is in target count $m$, not parameter dimension
$d$ or latent dimension $D$: E8/E9 use $d=5,D=2$. Dense factorization and
tensor-grid integration can cost $O(d^3)$ and $O(k^D)$, respectively,
with $k$ nodes per coordinate. Matrix-free variants and stochastic
oracles remain untested extensions (Section~\ref{app:scalability}).
The guarantees concern stationarity, not full Pareto coverage.
\section*{Reproducibility statement}
\addcontentsline{toc}{section}{Reproducibility statement}
Section~\ref{sec:models} proves the model identities and error bounds.
Section~\ref{sec:theory} states the exact-oracle assumptions, and
Appendices~\ref{app:geometry}--\ref{app:trust} give the remaining proofs.
Appendix~\ref{app:numerics} records all ten MATLAB experimental protocols,
including seeds, tolerances, timing boundaries, and audits. E5 and E10
use 120-digit arithmetic. Histories, paired timings, and validation
records were retained as described in the protocols.

\section*{AI use statement}
\addcontentsline{toc}{section}{AI use statement}
Generative AI tools were used to assist with polishing the language, improving readability and organization, and formatting the manuscript in \LaTeX{}. They were also used to assist in generating and debugging MATLAB code for the numerical experiments. All AI-generated content was carefully reviewed, revised, and verified by the authors, who assume full responsibility for the final manuscript.

\phantomsection

\clearpage
\appendix
\section{Proof of the residual derivative bounds}
\label{app:geometry}
\label{app:calibration-size}

\paragraph{Proof of Proposition~\ref{prop:residual-tensors}.}
Write $h=h_f$, $u=p/q_\theta$, $z=\nabla\log q_\theta$, and $J=\nabla^2\log q_\theta$.  For the forward orientation,
\begin{align}
 \nabla R^-(\theta)
 &=\int q_\theta\bigl(h(u)-uh'(u)\bigr)z\dd\mu,\\
 \nabla^2R^-(\theta)
 &=\int q_\theta\left[\bigl(h-uh'+u^2h''\bigr)zz^\top
 +\bigl(h-uh'\bigr)J\right]\dd\mu.
\end{align}
Because $h^{(j)}(1)=0$ for $j=0,1,2,3$, we have
$h(u)=O(|u-1|^4)$, $h'(u)=O(|u-1|^3)$, and $h''(u)=O(|u-1|^2)$.  The claimed bounds follow.  For the reverse orientation, with $u=q_\theta/p$,
\begin{align}
 \nabla R^+(\theta)&=\int q_\theta h'(u)z\dd\mu,\\
 \nabla^2R^+(\theta)&=\int q_\theta[(h'(u)+uh''(u))zz^\top+h'(u)J]\dd\mu,
\end{align}
and the same Taylor bounds apply. For the third derivative, put
$L=\nabla^3\log q_\theta$ and let $\operatorname{Sym}_3(z\otimes J)$
denote the sum over its three distinct placements. Differentiating
once more gives, in the forward orientation,
\[
 \nabla^3R^-=
 \int q_\theta\bigl[(h-uh'-u^3h''')z^{\otimes3}
 +(h-uh'+u^2h'')\operatorname{Sym}_3(z\otimes J)
 +(h-uh')L\bigr]\dd\mu.
\]
For the reverse orientation the three coefficients are respectively
$h'+3uh''+u^2h'''$, $h'+uh''$, and $h'$.
Using $h'''(u)=O(\abs{u-1})$, the preceding Taylor bounds, and the
bounded tensors $z,J,L$ proves~\eqref{eq:residual-third-bound}.
Weighted sums give the corresponding scalarized residual bound.
\hfill$\square$

\section{Proofs of the convergence guarantees}
\label{app:trust}

\subsection{Proof of Theorem~\ref{thm:global}}

Since $m_t^{(p)}(0)=F(\theta_t)$,
\begin{equation}
 \abs{\operatorname{ared}_t-\operatorname{pred}_t}\leq\kappa_f\Delta_t^2.
 \label{eq:ratio-error}
\end{equation}
Define
\begin{equation}
 c_\rho:=\min\left\{\frac{1}{\kappa_H},
 \frac{(1-\eta_2)\kappa_c}{\kappa_f}\right\}.
 \label{eq:c-rho}
\end{equation}
If $g_t\neq0$ and
\begin{equation}
 \Delta_t\leq c_\rho\norm{g_t},
 \label{eq:small-radius}
\end{equation}
then~\eqref{eq:cauchy}--\eqref{eq:ratio-error} give
$|1-\rho_t|\leq1-\eta_2$.  Thus a sufficiently small radius relative to the gradient produces a very successful step and cannot be repeatedly contracted.

Fix $0<\eps\leq\eps_0:=\Delta_0/(\gamma_1c_\rho)$.  Before the first $\eps$-stationary point,
\begin{equation}
 \Delta_t\geq\underline\Delta_\eps:=\gamma_1c_\rho\eps.
 \label{eq:radius-lower}
\end{equation}
Indeed, an unsuccessful iteration must have $\Delta_t>c_\rho\norm{g_t}>c_\rho\eps$, so its contracted successor exceeds~\eqref{eq:radius-lower}. A radius below $c_\rho\eps$ instead gives $\rho_t\geq\eta_2$ and is not contracted. The claim follows by induction from the definition of $\eps_0$.

Every successful iteration before termination therefore decreases the true objective by at least
\begin{equation}
 \eta_1\kappa_c\eps\min\{\gamma_1c_\rho\eps,\eps/\kappa_H\}=c_s\eps^2.
\end{equation}
Boundedness below limits the successful iterations to
$(F(\theta_0)-F_{\inf})/(c_s\eps^2)$.  If $N_u$ and $N_v$ denote unsuccessful and very successful iterations, respectively, the radius recursion and~\eqref{eq:radius-lower} imply
\begin{equation}
 N_u|\log\gamma_1|
 \leq N_v\log\gamma_2+\log\!\frac{\Delta_{\max}}{\gamma_1c_\rho\eps}.
\end{equation}
Since $N_v$ is no larger than the successful count and $|\log\eps|=O(\eps^{-2})$ on $(0,1)$, this proves~\eqref{eq:complexity}. Larger $\eps$ are absorbed into $C_0$. The constants need not be uniform in $m$ or target heterogeneity.

For the zero-tolerance run, either an exact stationary point is reached or the preceding argument gives $\liminf_t\norm{g_t}=0$. To prove convergence of the entire gradient sequence, fix $\eps>0$ and set
\begin{equation}
 b_\eps=\eta_1\kappa_c\eps
 \min\{1,\eps/(\kappa_H\Delta_{\max})\}>0.
 \label{eq:gradient-excursion-constant}
\end{equation}
At each successful iteration with $\norm{g_t}\geq\eps$, Cauchy decrease implies
\begin{equation}
 F(\theta_t)-F(\theta_{t+1})
 \geq b_\eps\Delta_t\geq b_\eps\norm{s_t}.
 \label{eq:decrease-length}
\end{equation}
If $\limsup_t\norm{g_t}>0$, choose $\eps$ with $2\eps<\limsup_t\norm{g_t}$. Since the liminf is zero, there are infinitely many disjoint intervals starting with $\norm{g_t}\geq2\eps$ and ending at its first subsequent value below $\eps$. Within each interval, Lipschitz continuity requires total accepted-step length at least $\eps/L_g$. Rejected steps do not move the iterate. Equation~\eqref{eq:decrease-length} therefore gives objective decrease at least $b_\eps\eps/L_g$ per interval, contradicting the lower bound on $F$. Thus $\norm{g_t}\to0$. \hfill$\square$

\paragraph{Proof of Corollary~\ref{cor:pareto}.}
Continuity at any accumulation point $\bar\theta$ gives
$\sum_i\lambda_i\nabla F_i^\sigma(\bar\theta)=0$. A direction that strictly decreases every component to first order would have negative inner product with this weighted sum, a contradiction. If the component objectives are convex, the first-order convexity inequality shows that $\bar\theta$ globally minimizes their weighted sum. Because every $\lambda_i>0$, any point dominating $\bar\theta$ would strictly lower that sum. Hence $\bar\theta$ is globally Pareto optimal. \hfill$\square$

\subsection{Local stationary branch and proof of Theorem~\ref{thm:local}}

For a generic anchor $\theta$, write $m_\theta^{(p)}$ for~\eqref{eq:model} and define
$G_p(\theta,s)=\nabla_sm_\theta^{(p)}(s)$.

\begin{lemma}[Local model stationary branch]
\label{lem:local-branch}
Under the hypotheses of Theorem~\ref{thm:local}, there are open balls around $\theta^\star$ and $0$ and a unique continuously differentiable map $s_p(\theta)$ such that
$G_p(\theta,s_p(\theta))=0$, $s_p(\theta^\star)=0$, and
\begin{equation}
 \norm{s_p(\theta)}\leq C_s\norm{\theta-\theta^\star},
 \qquad
 \nabla_s^2m_\theta^{(p)}(s)\succeq\frac{\mu}{2}I
 \label{eq:branch-bounds}
\end{equation}
throughout a smaller product neighborhood.  If $s$ lies in this neighborhood, then
$\norm{s-s_p(\theta)}\leq(2/\mu)\norm{G_p(\theta,s)}$.
\end{lemma}

\paragraph{Proof.}
Derivative matching gives $G_p(\theta^\star,0)=\nabla F(\theta^\star)=0$ and
$\nabla_sG_p(\theta^\star,0)=\nabla^2F(\theta^\star)\succ0$.  The implicit function theorem gives the unique branch.  Its derivative is locally bounded, which yields the first bound in~\eqref{eq:branch-bounds}. Continuity gives the Hessian bound.  Strong monotonicity of $G_p(\theta,\cdot)$ then gives the final inequality. \hfill$\square$

\paragraph{Proof of Theorem~\ref{thm:local}.}

Let $e_t=\norm{\theta_t-\theta^\star}$ and $s_p^t=s_p(\theta_t)$. Strong regularity gives constants $c_1,c_2>0$ such that locally
\begin{equation}
 c_1e_t\leq\norm{\nabla F(\theta_t)}\leq c_2e_t.
 \label{eq:gradient-equivalence}
\end{equation}
Lemma~\ref{lem:local-branch} and~\eqref{eq:inner} give
$\norm{s_t-s_p^t}=O(e_t^p)$ and hence $\norm{s_t}\leq c_3e_t$. At the trial point,~\eqref{eq:model-grad} and~\eqref{eq:inner} give
\begin{align}
 \norm{\nabla F(\theta_t+s_t)}
 &\leq\frac{L_{p+1}}{p!}\norm{s_t}^p
 +\kappa_s\norm{\nabla F(\theta_t)}^p
 \leq c_4e_t^p.
\end{align}
Equation~\eqref{eq:gradient-equivalence} at the trial point yields the order-$p$ error bound whenever the step is accepted. To verify acceptance, let $M$ bound the local model Hessians. Since
$g_t=-\int_0^1\nabla_s^2m_t^{(p)}(u s_p^t)s_p^t\dd u$,
we have $\norm{s_p^t}\geq\norm{g_t}/M$. Strong convexity gives
$m_t^{(p)}(0)-m_t^{(p)}(s_p^t)\geq\mu\norm{s_p^t}^2/4$,
while
$m_t^{(p)}(s_t)-m_t^{(p)}(s_p^t)=O(e_t^{2p})$.
Together with the local upper bounds, this proves
$\operatorname{pred}_t=\Theta(e_t^2)$ for nonstationary anchors. The value-model error is $O(e_t^{p+1})$, so
$|1-\rho_t|=O(e_t^{p-1})\to0$. Fix $\Delta_0\in(0,\Delta_{\max}]$.
Choose the initial neighborhood sufficiently small, depending on $\Delta_0$, so that the returned steps satisfy $\norm{s_t}<\Delta_0$, the local error bound is contractive, and $\rho_t\geq\eta_2$.
The accepted iterates then remain in this neighborhood, while the radius update never decreases the radius. Thus the local steps remain untruncated and accepted, and~\eqref{eq:qorder} applies. \hfill$\square$

\subsection{Proof of Corollary~\ref{cor:lazy}}
\label{app:lazy-proof}

Lazy steps only decrease $F$. Thus each successful certified iteration
still contributes its original true decrease, and the radius recursion
is unchanged. The radius lower bound and successful-iteration count
in the proof of Theorem~\ref{thm:global} therefore apply before the
first $\eps$-stationary certified anchor.

For the full gradient limit, the path from $\theta_t$ to
$\theta_{t+1}$ has length at most
$(1+\kappa_{\rm move})\norm{s_t}$ at a successful iteration and zero
at a rejected one. Equation~\eqref{eq:decrease-length}, with its
constant divided by $1+\kappa_{\rm move}$, controls this entire path
length whenever $\norm{g_t}\geq\eps$. The same disjoint-excursion
argument proves $\norm{g_t}\to0$ at certified anchors. Every intermediate
lazy point is within $\kappa_{\rm path}\norm{g_t}$ of the block endpoint.
Lipschitz continuity transfers the limit to those points as well.
At most $K_{\rm lazy}$ extra true-value calls accompany each certified
iteration, proving the stated evaluation count. These safeguards
establish global stationarity, not the local order of the augmented
sequence. \hfill$\square$

\section{Experimental protocols and additional results}
\label{app:numerics}

The generators and their normalization are defined in Section~\ref{sec:generators}.

\subsection{Local normal-form scaling}
\label{app:e1-details}

For E1, each of 20 seeds uses $K=128$, draws $z_k\sim\mathcal N(0,0.55^2)$ and $v_k\sim\mathcal N(0,1)$, sets $P=\softmax(z)$, centers $v$ under $P$, and rescales it to $\norm v_\infty=1$.  Along
\begin{equation}
 Q_t(k)=\frac{P_k\exp(tv_k)}{\sum_\ell P_\ell\exp(tv_\ell)},
 \qquad Q_0=P,
\end{equation}
we evaluate $R_f^-(t)=D_f(P\|Q_t)-a_fD_{\alpha_f}(P\|Q_t)$ and $R_f^+(t)=D_f(Q_t\|P)-a_fD_{\alpha_f}(Q_t\|P)$ and their first two analytic derivatives.  The 21 step sizes span $[10^{-3.2},10^{-0.7}]$ and slopes are fitted on $[2\times10^{-3},3\times10^{-2}]$. We test JS, KLChi, and KL4 under the conventions above, together with the exact canonical control.

The fitted value/first-/second-derivative orders are $(4,3,2)$.  Across the 18 aggregate cells, the largest mean-exponent deviation is $7.23\times10^{-4}$, the largest across-seed standard deviation is $5.55\times10^{-3}$, and all exact-control residuals are below $10^{-12}$.

\begin{figure}[!htbp]
  \centering
  \includegraphics[width=0.902\linewidth]
  {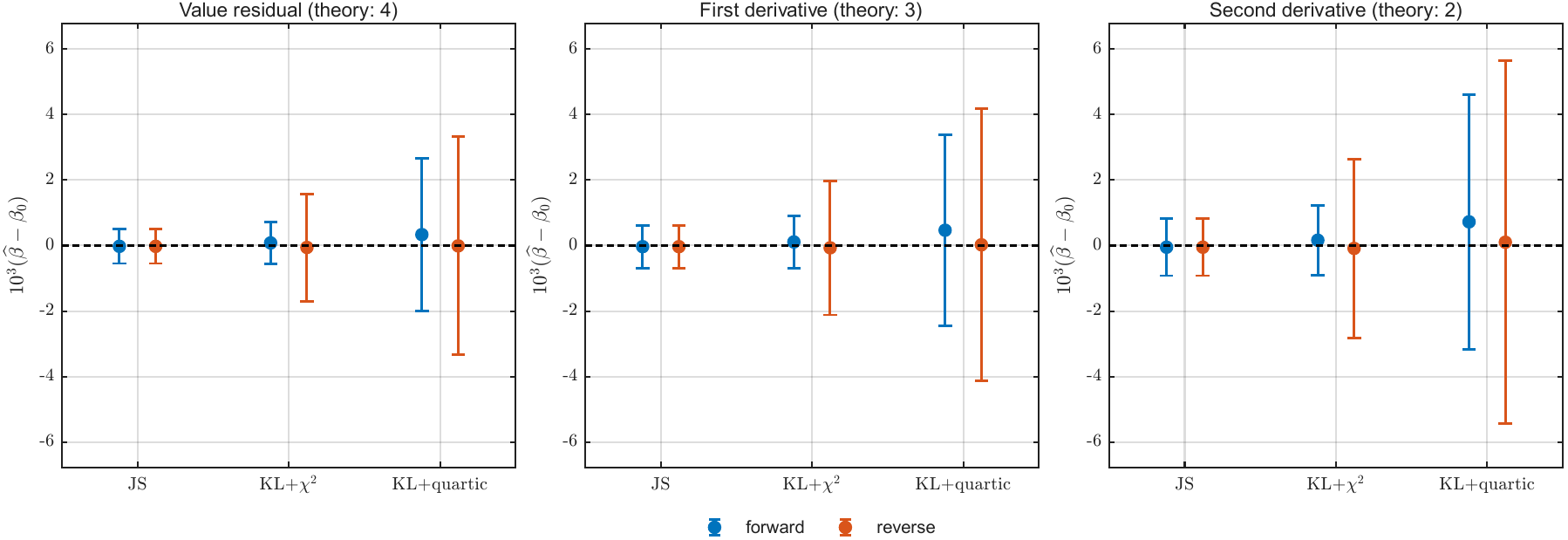}
  \caption{Fitted-exponent deviations for E1, scaled by $10^3$.  Markers show means over 20 seeds, error bars one standard deviation, and the dashed line exact agreement.}
  \label{fig:e1-slopes}
\end{figure}

\subsection{Exact aggregation and amortized cost}
\label{app:e2-details}

For the identity check, E2 uses $K=128$, $m\in\{2,4,8,16,32,64,128\}$, $\alpha\in\{-3,-2,-1,0,1,2,3\}$, both orientations, 20 seeds, and eight query distributions per cell.  All distributions are softmaxes of Gaussian logits with scale $0.65$ and all preferences are positive.  The direct scalarization is compared with~\eqref{eq:agg-minus}--\eqref{eq:agg-plus}, including $\omega^\sigma_{\lambda,\alpha}$.  All 15,680 query comparisons pass the $10^{-10}$ tolerance. The maximum scaled identity and normalization errors are $3.77\times10^{-15}$ and $6.66\times10^{-16}$.

The timing benchmark uses $K=2048$, $m\in\{2,4,8,16,32,64,128,256\}$, 256 sequential queries, and 30 paired repetitions with alternating order.  A break-even count is reported only when the central empirical 95\% interval of paired per-query savings is positive.  At $m=256$, median target-kernel speedups are $12.00\times$, $10.39\times$, and $12.08\times$ for forward KL, reverse KL, and forward $\alpha=0$. End-to-end speedups are $2.79\times$, $2.46\times$, and $3.14\times$, with break-even counts 45, 5, and 100.

\begin{figure}[!htbp]
  \centering
  \includegraphics[width=0.902\linewidth]{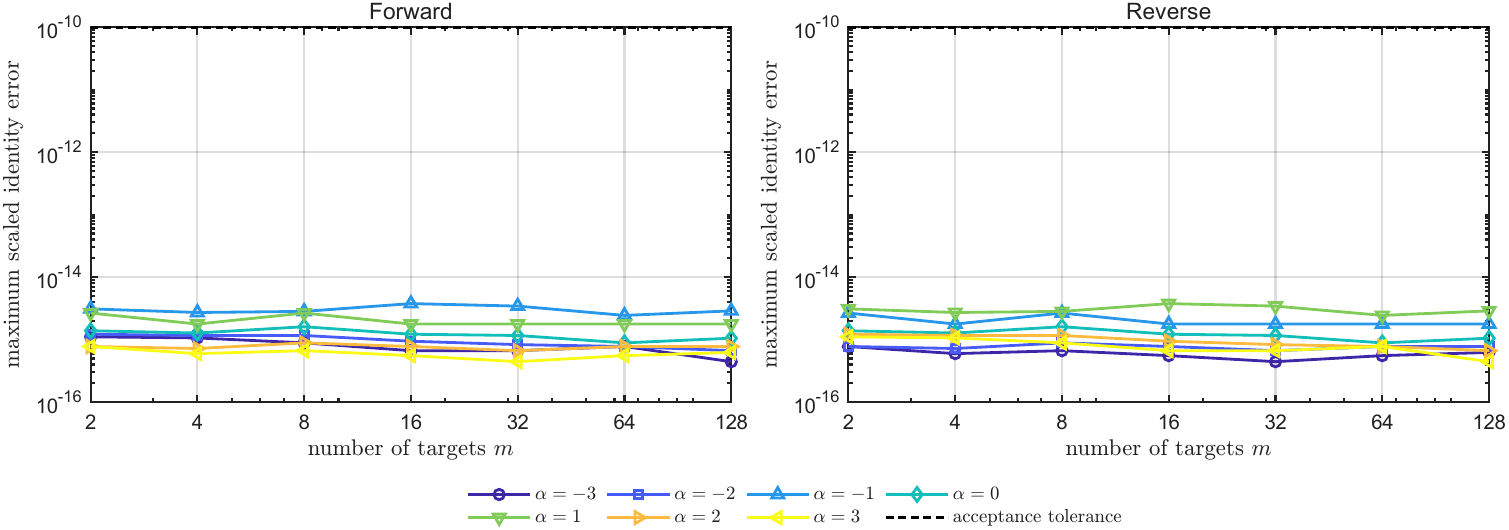}
  \caption{E2 aggregation identity errors in both orientations: maximum scaled discrepancy over 20 seeds and eight queries per seed, including the aggregate's constant and scale. The dashed line marks tolerance $10^{-10}$. Errors remain at double-precision roundoff.}
  \label{fig:e2-identity}
\end{figure}

\subsection{Calibrated-model order diagnostics}
\label{app:e3-details}

E3 uses $K=64$ categories, 16 targets, 20 seeds, and the three generators
in Section~\ref{sec:generators}, each in both orientations.
For each seed, draw $z_c\sim\mathcal N(0,0.45^2I_K)$ and
$z_i\sim\mathcal N(0,I_K)$ independently. Set
$Q_0=\softmax(z_c)$ and $P_i=\softmax(z_i)$.
Normalize weights $\lambda_i\propto0.25+U_i$, with independent
$U_i\sim\operatorname{Unif}(0,1)$. Center a Gaussian direction $d$
under $Q_0$ and rescale it to $\norm d_\infty=1$. Both signed rays are used.

The 25 step magnitudes are logarithmically spaced in $[10^{-3},0.2]$.
Log--log fits use $[2\times10^{-3},3\times10^{-2}]$ and exclude errors
below $10^{-13}$. If fewer than six points remain, the upper endpoint
is extended to $6\times10^{-2}$. The predictor includes the aggregation
constant and scale. Residual coefficients through degree three are
computed with fourth-order truncated-series arithmetic.

For $p=1,2,3$, value and directional-derivative fits recover orders
$(p+1,p)$ (Figure~\ref{fig:e3-model-scaling}). The median maximum
absolute log-likelihood ratio is $3.94$, and the median target-to-center
KL is $0.540$. Each of the 18 generator--orientation--order cells has
40 trials (20 seeds, two signs). All 720 trials yield valid value and derivative fits. The largest
mean-order deviation is $1.67\times10^{-2}$ and the largest standard
deviation is $0.148$ (Figure~\ref{fig:e3-slope-summary}). At $p=2$, one
value fit and two directional-derivative fits have $R^2<0.95$. The two
derivative residuals change sign within the window. All fits are retained.

All 80 exact-$\alpha$ controls with $\alpha=0.4$ pass. Their maximum
value, directional-derivative, and residual-jet discrepancies are
$1.94\times10^{-15}$, $1.25\times10^{-16}$, and $1.78\times10^{-15}$.

\begin{figure}[!htbp]
\centering
\includegraphics[width=0.864\linewidth]{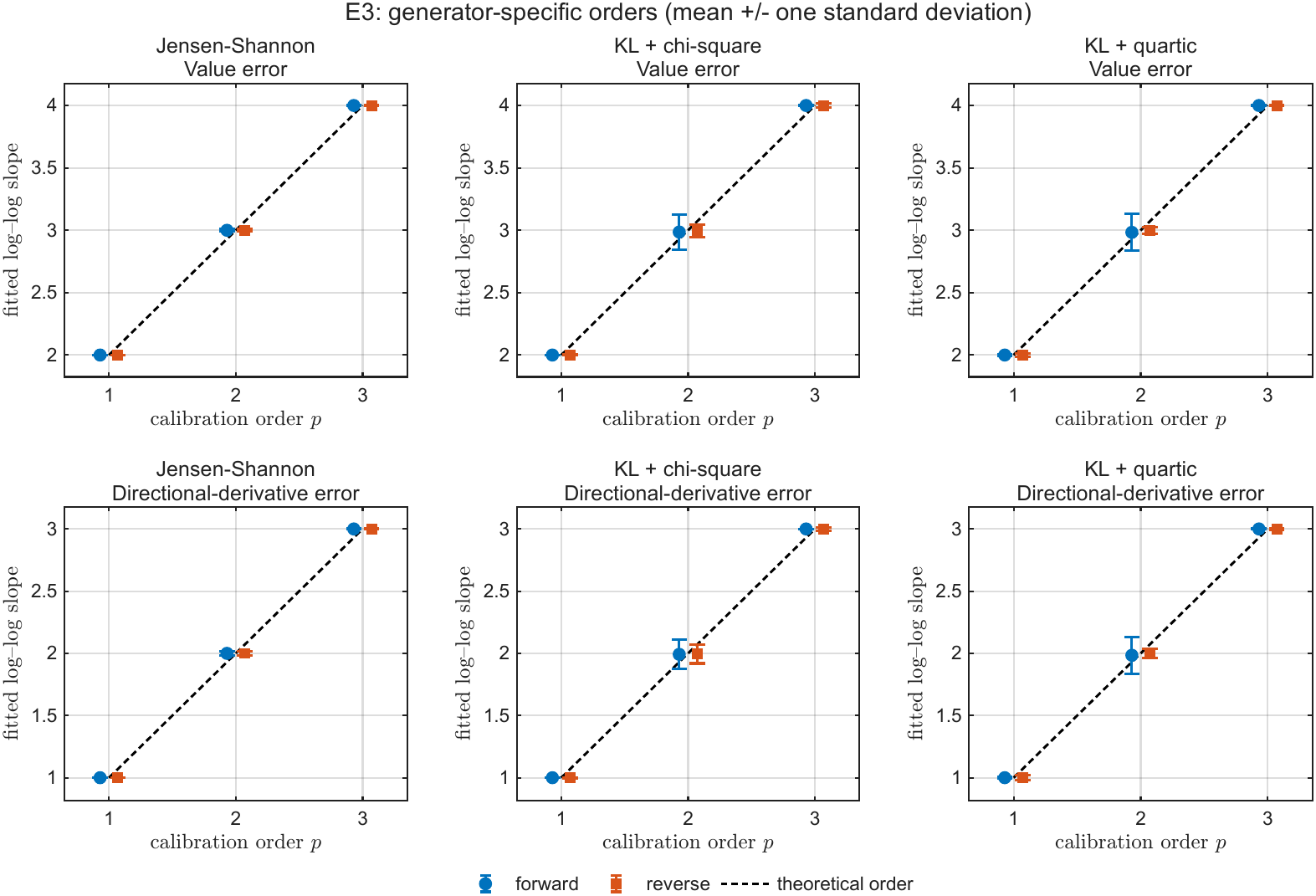}
\caption{E3 fitted model-error orders by generator (columns): value errors (top, order $p+1$) and directional-derivative errors (bottom, order $p$). Markers and bars show means and one standard deviation over 20 seeds and two signed rays, not confidence intervals. Blue circles/orange squares denote forward/reverse orientations, offset for visibility. Dashed lines mark theoretical orders.}
\label{fig:e3-slope-summary}
\end{figure}

\subsection{Full optimization: robustness and cost}
\label{app:e4-details}

\paragraph{Problem construction.}
We use $Q_\theta=\softmax(b+B\theta)$ with $K$ categories and
$d<K$ parameters.  Let
$\Pi=I_K-\mathbf{1}\mathbf{1}^{\mathsf T}/K$.
We obtain $U\in\R^{K\times d}$ by QR factorization of a
column-centered Gaussian matrix and set $B=1.2U$.
Independently draw
$b\sim\mathcal N(0,0.25^2I_K)$,
$\xi_i\sim\mathcal N(0,I_d)$,
$z_i\sim\mathcal N(0,I_K)$, and
$\zeta\sim\mathcal N(0,I_d)$.
The targets and initialization are
\begin{align}
 P_i
 &=\softmax\!\left(b+\tau B\xi_i+0.1\tau\Pi z_i\right),
 \label{eq:e4-targets}\\
 \theta_0
 &=-\frac{0.75\tau}{m}\sum_{i=1}^m\xi_i
   +0.8\max\{\tau,0.25\}\zeta .
 \label{eq:e4-initialization}
\end{align}
The perturbation $\Pi z_i$ allows targets outside the model
family.  A deterministic seed mapping fixes each problem,
and all methods use the same targets, preference, and
initialization for a paired comparison.

We use JS and KL4 with the generator conventions stated in Section~\ref{sec:generators}, and test both divergence orientations.

\paragraph{Experimental configurations.}
The robustness experiment fixes $(K,d,m)=(128,16,32)$ and
varies $\tau\in\{0.1,0.5,1.0,1.5\}$.
Preferences are uniform, drawn from
$\operatorname{Dirichlet}(1,\ldots,1)$, or biased with
$\lambda_1=0.6$ and $\lambda_i=0.4/(m-1)$ for $i>1$.
The scaling experiment uses $(K,d)=(512,64)$,
$m\in\{8,32,128\}$, $\tau=1$, and Dirichlet preferences.
Both experiments use seeds $1,\ldots,30$, two generators,
two orientations, and seven methods.
They contain 10,080 and 2,520 method runs, respectively,
giving 1,800 main-experiment runs per method.

\paragraph{Methods and numerical settings.}
The seven methods are Direct-LBFGS, Direct-TR, Direct-NL-TR,
Uncalibrated-$\alpha$, CAL-$\alpha$-$p1$,
CAL-$\alpha$-$p2$, and Lazy-CAL-$\alpha$-$p1$.
The uncalibrated method uses the aggregated predictor without
residual-derivative calibration, but its success is assessed
using the true objective gradient.

The stopping tolerance is $10^{-6}$, with a budget of 120
certified iterations and 360 trust-region trials.
Trust-region parameters are
$\Delta_0=0.5$, $\Delta_{\max}=2$,
$\Delta_{\min}=10^{-10}$,
$(\eta_1,\eta_2)=(0.1,0.75)$, and
$(\gamma_1,\gamma_2)=(0.5,2)$.
The Cauchy-decrease check uses $\kappa_c=0.1$ and the
conservative curvature bound $10^8$.
The common nonlinear subproblem routine permits at most
12 inner iterations and uses gradient tolerance $10^{-6}$.
Direct-LBFGS uses memory 10 and Armijo backtracking.

The lazy variant permits at most three trials per block and
uses $c_{\rm lazy}=10^{-4}$.
Its path coefficient is selected as $0.25/L_{\rm est}$, where
\[
 L_{\rm est}
 =2\max\{1,\norm{\nabla^2F(\theta^c)}_{\op}\}.
\]
This is a numerical local-curvature estimate, not a verified
global Lipschitz bound.  The implementation retains the stale
model's certified coordinate origin, enforces the path budget
and true-objective sufficient decrease, and leaves the certified
radius unchanged during lazy trials. The additional step-length
budget in Section~\ref{app:lazy} was not enforced here, so these
measurements evaluate the implemented heuristic, rather than verify
every hypothesis of Corollary~\ref{cor:lazy}.

\paragraph{Metrics and paired comparisons.}
Success is defined exclusively by
$\norm{\nabla F_\lambda^\sigma(\theta_{\rm final})}\leq10^{-6}$.
End-to-end time includes per-method pseudo-target construction, residual calibration, inner solves, and true-objective acceptance checks. Data generation, derivative audits, plotting, and export are excluded.

The component-equivalent ledger charges $m$ units per full scalarization query or pseudo-target construction and one per predictor query. Derivative orders are recorded but equally weighted: the ledger measures query structure, not floating-point work.

For method $a$ and a paired Direct-NL-TR run, define
\[
 S_a=\frac{T_{\rm Direct\text{-}NL\text{-}TR}}{T_a},
 \qquad
 R_a=\frac{C_a}{C_{\rm Direct\text{-}NL\text{-}TR}},
\]
where $T$ is recorded end-to-end time and $C$ is the
component-equivalent count.
We report medians of per-problem ratios only when both runs pass the true-gradient test.

\paragraph{Robustness and cost results.}
All six direct and calibrated methods succeed in every
main-experiment run, whereas the uncalibrated method succeeds
in 180/1,800 runs.
Table~\ref{tab:e4-cost} summarizes the scaling results at
$m=128$, with 30 seeds and two orientations per generator.

\begin{table}[!htbp]
  \centering
  \small
\setlength{\tabcolsep}{5pt}
\renewcommand{\arraystretch}{1.04}
\setlength{\abovecaptionskip}{0pt}
\setlength{\belowcaptionskip}{4pt}
  \caption{E4 at $m=128$: median paired ratios against Direct-NL-TR. Each displayed method succeeds in all 60 runs per generator. Uncalibrated fails in all 60, so its paired costs are unavailable. Higher speedups and lower call ratios are better.}
  \label{tab:e4-cost}
  \begin{tabular}{@{}lrrrr@{}}
    \toprule
    & \multicolumn{2}{c}{Time speedup}
    & \multicolumn{2}{c}{Call ratio} \\
    \cmidrule(lr){2-3}\cmidrule(l){4-5}
    Method & JS & KL+quartic & JS & KL+quartic \\
    \midrule
    Direct-LBFGS          & 12.87 & 5.77 & 0.128 & 0.227 \\
    Direct-TR             &  3.36 & 4.15 & 0.119 & 0.143 \\
    Direct-NL-TR          &  1.00 & 1.00 & 1.000 & 1.000 \\
    CAL-$\alpha$-$p1$     &  1.74 & 0.41 & 0.137 & 0.490 \\
    CAL-$\alpha$-$p2$     &  1.76 & 2.55 & 0.135 & 0.160 \\
    Lazy-CAL-$\alpha$-$p1$&  1.44 & 0.36 & 0.247 & 0.815 \\
    \bottomrule
  \end{tabular}
\end{table}

For CAL-$\alpha$-$p2$, percentile bootstrap 95\% intervals are $[1.72,1.79]$ (JS) and $[2.45,2.66]$ (KL+quartic), using 10,000 resamples of 30 seed clusters with both orientations retained. They measure between-instance variation, not repeated-timing noise.

Figure~\ref{fig:e4-convergence} contrasts calibrated convergence with uncalibrated stalling. Direct-LBFGS and Direct-TR remain competitive or faster. The reported gain is against Direct-NL-TR's repeated full-objective queries.

\begin{figure}[!htbp]
  \centering
  \includegraphics[width=0.902\linewidth]{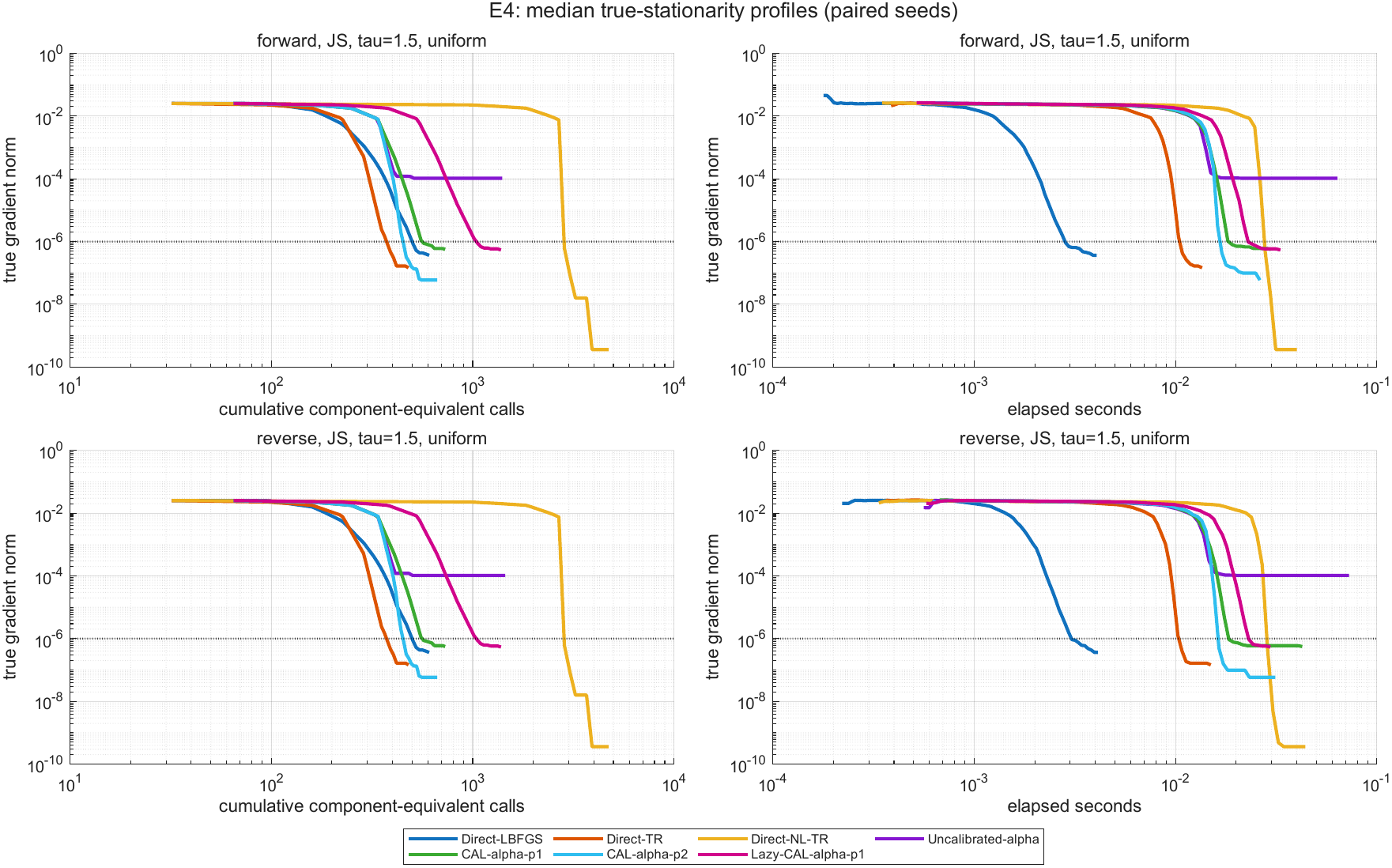}
  \caption{E4 true-gradient profiles: JS, $(K,d,m)=(128,16,32)$, $\tau=1.5$, uniform preferences. Rows separate orientations. Budgets are component-equivalent calls (left) and time (right). Curves are pointwise medians over 30 seeds after log-coordinate interpolation. Final gradients are carried forward, and budgets before the first record are unavailable. The dotted line marks $10^{-6}$.}
  \label{fig:e4-convergence}
\end{figure}

\paragraph{Implementation controls.}
Four generator--orientation derivative checks pass, with
maximum relative gradient and Hessian discrepancies
$3.06\times10^{-12}$ and $3.71\times10^{-12}$.
An additional exact-$\alpha$ experiment uses $\alpha=0.4$,
$(K,d,m)=(64,8,16)$, ten seeds, both orientations, and five
methods, giving 100 runs.
The 80 comparisons against Direct-NL-TR agree in their
success decisions, with maximum final-objective discrepancy
$2.12\times10^{-11}$.

A separate small-dimensional control uses
$(K,d,m)=(64,4,16)$, ten seeds, two generators, both
orientations, and five methods, giving 200 runs.
All 40 CAL-$\alpha$-$p3$ runs reach the true-gradient tolerance.
In this control, the third residual tensor is approximated by
centered differences of residual Hessians with step
$2\times10^{-4}$ and then symmetrized.
This tests attainable optimization accuracy, not the
Q-cubic convergence rate.
Together with the main experiments, E4 contains 12,900 runs.

\subsection{Local high-order convergence and condition ablations}
\label{app:e5-details}

\paragraph{Problem construction.}
E5 tests Theorem~\ref{thm:local} on the affine chart $q_\theta=q_0+B\theta>0$, with $(K,d,m)=(6,2,4)$.
For seed $s\in\{1,\ldots,10\}$, the MATLAB random stream is
\texttt{rng(27100+s,'twister')}.
Draw $z\sim\mathcal N(0,0.35^2I_K)$ and a standard Gaussian
matrix $G\in\R^{K\times d}$, and set
\begin{equation}
 q_0=\softmax(z),\qquad
 \widetilde G=G-\mathbf 1(q_0^\top G),\qquad
 B=\frac{0.4\,\operatorname{diag}(q_0)\widetilde G}
 {\max_k\norm{\widetilde G_{k,:}}_2}.
\end{equation}
Thus $\mathbf 1^\top B=0$.
Targets are $P_i=\softmax(\log q_0+\xi_i)$ with independent
$\xi_i\sim\mathcal N(0,0.7^2I_K)$, and preferences are
$\lambda_i=w_i/\sum_jw_j$, where $w_i$ are independent
$\operatorname{Unif}(0.25,1.25)$ draws.
We use the JS and KL+quartic generators specified in Section~\ref{sec:generators}, in both orientations.
The same primitive random draws are used across generators,
orientations, and calibration orders within each seed.

\paragraph{Precision and local solves.}
E5 uses 120-digit arithmetic. Independent damped Newton solves obtain
40 references with $\norm{\nabla F(\theta^\star)}\leq10^{-90}$ and
minimum Hessian eigenvalues in $[2.206\times10^{-3},0.665]$.
References are used only for initialization and error measurement.
Acceptance and stopping use the true objective and gradient.
Recomputing the first seed's four references at 160 digits gives a
maximum parameter discrepancy of $4.88\times10^{-108}$.
The diagnostic error floor is $10^{-80}$, without an interval certificate.

Primary runs use analytic residual derivatives through $p=2$ or $3$.
At most 35 damped Newton iterations, initialized at zero and maintaining
positive model Hessians, seek the local stationary branch.
The forcing condition~\eqref{eq:inner} uses $\kappa_s=0.01$.
An arithmetic floor of $10^{-105}$ is recorded separately. Untruncated
primary steps are accepted only if the original forcing condition holds.
The initial and maximum radii are $1$ and $4$, with
$(\eta_1,\eta_2)=(0.1,0.75)$. Rejection halves the radius, and a very
successful step doubles it. Each applied ball is limited to $0.45$
times the distance to the positive-chart boundary. The Cauchy test uses
an analytic ballwise curvature bound and $\kappa_c=0.1$.
Runs stop at true-gradient norm $10^{-30}$ or after 60 trials.

\paragraph{Independent one-step maps.}
Let $v$ be a normalized Gaussian direction and define
\begin{equation}
 r_\star=\min_{k:\norm{B_{k,:}}_2>0}
 \frac{q_{\theta^\star}(k)}{\norm{B_{k,:}}_2},\qquad
 c_\star=\min\{1,0.2r_\star\},\qquad
 \theta_{\epsilon,h}=\theta^\star+\epsilon h c_\star v,
 \quad \epsilon\in\{-1,1\}.
\end{equation}
For each problem, order, and sign, independent starts at
$h\in\{10^{-2},\ldots,10^{-8}\}$ give 1,120 probes.
Log next-error versus log initial-error fits use the four fixed radii
$h\in\{10^{-3},\ldots,10^{-6}\}$. All four points must be accepted,
untruncated, forcing-compliant, and above the error floor.
All 1,120 probes and 160 fits qualify, without slope-based filtering or
adaptive windows. The largest slope deviation from $p$ is
$5.79\times10^{-4}$, the largest within-group standard deviation is
$2.21\times10^{-4}$, and the minimum $R^2$ is $0.999999960$.
Figures~\ref{fig:e5-one-step} and~\ref{fig:e5-orders} report the results.
Signed starts share a seed. JS orientations coincide and are not
independent replications.

\begin{figure}[!htbp]
  \centering
  \includegraphics[width=0.817\linewidth]{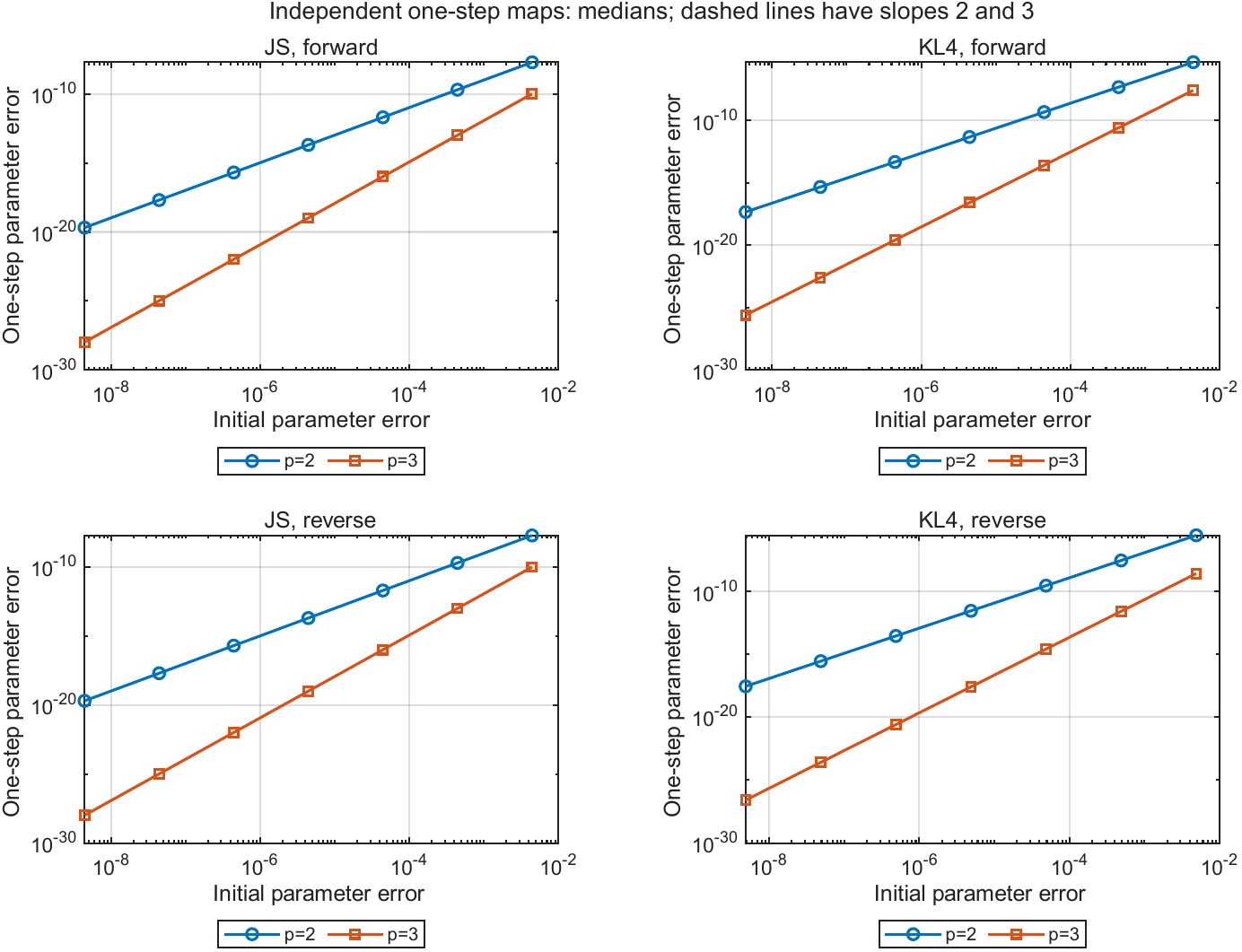}
  \caption{E5 independent one-step errors by generator (columns) and orientation (rows). Coordinates are medians over ten seeds and two signs. Dashed slopes two and three largely overlap the curves. Fits use each seed/sign's prespecified window, not the median curves.}
  \label{fig:e5-one-step}
\end{figure}

\paragraph{Actual iterates.}
Primary trajectories start at $\theta_{\epsilon,0.1}$,
giving $10\times2\times2\times2\times2=160$ runs.
All reach the true-gradient tolerance. The largest final
gradient norm is $3.43\times10^{-31}$.
All 542 primary steps are accepted and untruncated and pass
the Cauchy test. The largest ratio
$\norm{\nabla m_t^{(p)}(s_t)}/\norm{\nabla F(\theta_t)}^p$
is $0.009790<0.01$.
For $e_t=\norm{\theta_t-\theta^\star}$, define
\begin{equation}
 \widehat q_t=
 \frac{\log(e_{t+1}/e_t)}{\log(e_t/e_{t-1})}.
\end{equation}
An eligible observation requires three decreasing errors
above the resolution threshold and two consecutive accepted,
untruncated steps satisfying the local checks.
Table~\ref{tab:e5-iterate-orders} reports the last eligible observation from all 160 runs as finite-trajectory evidence.

\begin{table}[!htbp]
  \centering
  \small
\setlength{\tabcolsep}{5pt}
\renewcommand{\arraystretch}{1.04}
\setlength{\abovecaptionskip}{0pt}
\setlength{\belowcaptionskip}{4pt}
  \caption{E5 actual-iterate orders: mean $\pm$ one standard
  deviation of the last eligible $\widehat q_t$ from each run.
  Each cell contains ten seeds and two paired signs.
  JS orientations coincide mathematically.}
  \label{tab:e5-iterate-orders}
  \begin{tabular}{@{}llcc@{}}
    \toprule
    Generator & Orientation & $p=2$ & $p=3$ \\
    \midrule
    JS & Forward & $1.993\pm0.045$ & $2.996\pm0.012$ \\
    JS & Reverse & $1.993\pm0.045$ & $2.996\pm0.012$ \\
    KL+quartic & Forward & $2.014\pm0.029$ & $2.996\pm0.012$ \\
    KL+quartic & Reverse & $1.995\pm0.015$ & $2.994\pm0.013$ \\
    \bottomrule
  \end{tabular}
\end{table}

\begin{figure}[!htbp]
  \centering
  \includegraphics[width=0.817\linewidth]{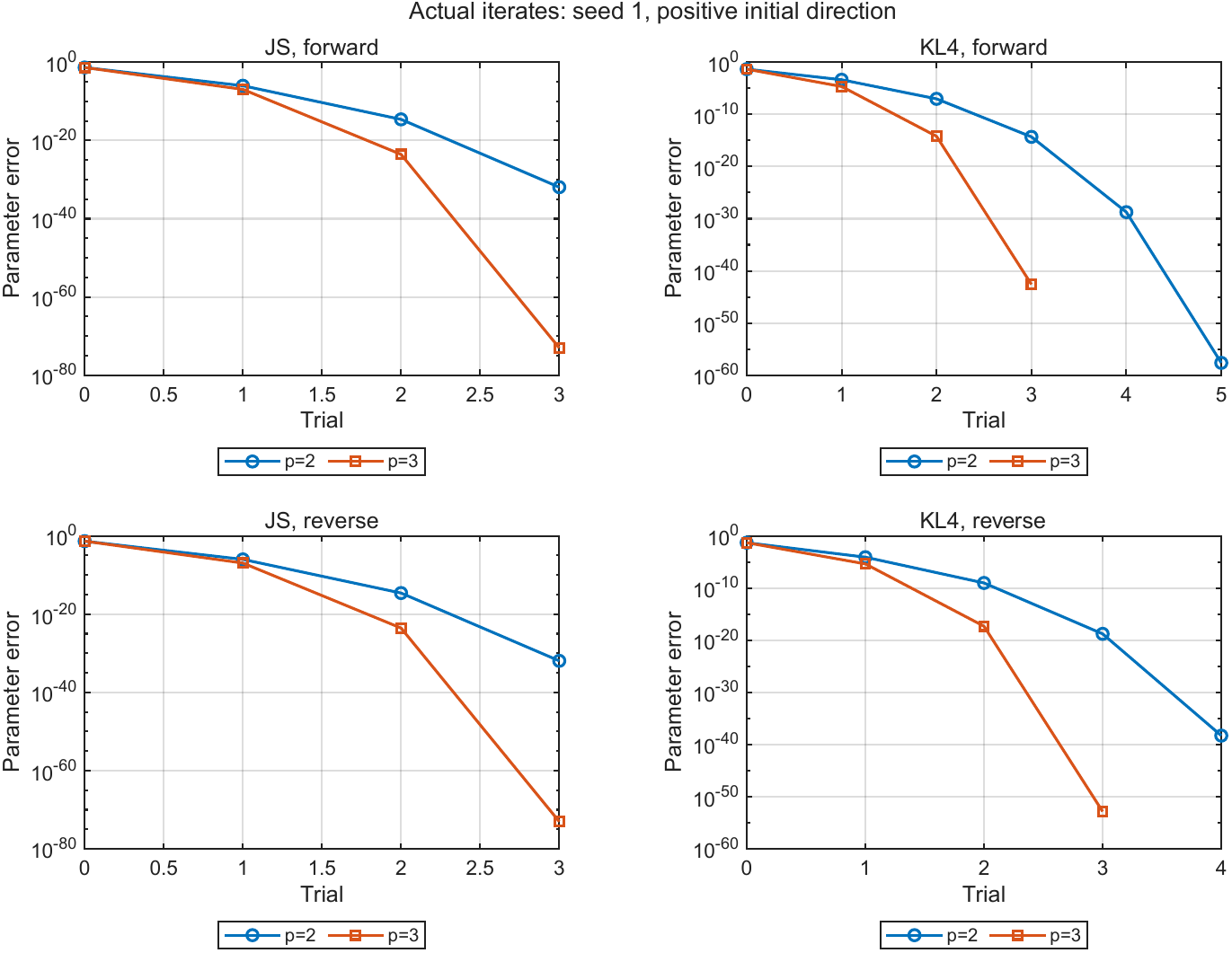}
  \caption{E5 parameter-error trajectories, seed 1, positive initial direction, 120-digit arithmetic. Panels separate generators and orientations. Table~\ref{tab:e5-iterate-orders} summarizes all runs.}
  \label{fig:e5-local-trajectories}
\end{figure}

\paragraph{Condition ablations and singular control.}
Four $p=3$ controls modify one condition on the first forward KL+quartic
problem with positive initial direction (Figure~\ref{fig:e5-ablations}).
A fixed inner tolerance of $10^{-8}$ stalls at true-gradient norm
$6.23\times10^{-11}$. Capping each step at one quarter of the branch
step gives late error ratio $0.75$ and order $1.00$, exhausting the
60-trial budget. Residual Hessians differenced with fixed step
$2\times10^{-4}$ still yield stationarity, but the last observed order
is $2.257$. This short run shows third-jet sensitivity rather than a
finite-difference asymptotic rate. With initial radius $10^{-3}$,
five truncated steps are followed by recovery to observed order $2.982$.

Setting the second column of $B$ to zero creates a flat direction.
From $\theta^\star+(0.02,0.1)^\top$, this control reaches gradient norm
$3.64\times10^{-58}$ while remaining $0.1$ from the selected reference.
Stationarity therefore need not identify a unique parameter without
nondegeneracy. The five controls bring the total to 165 trajectories.

\begin{figure}[!htbp]
  \centering
  \includegraphics[width=0.836\linewidth]{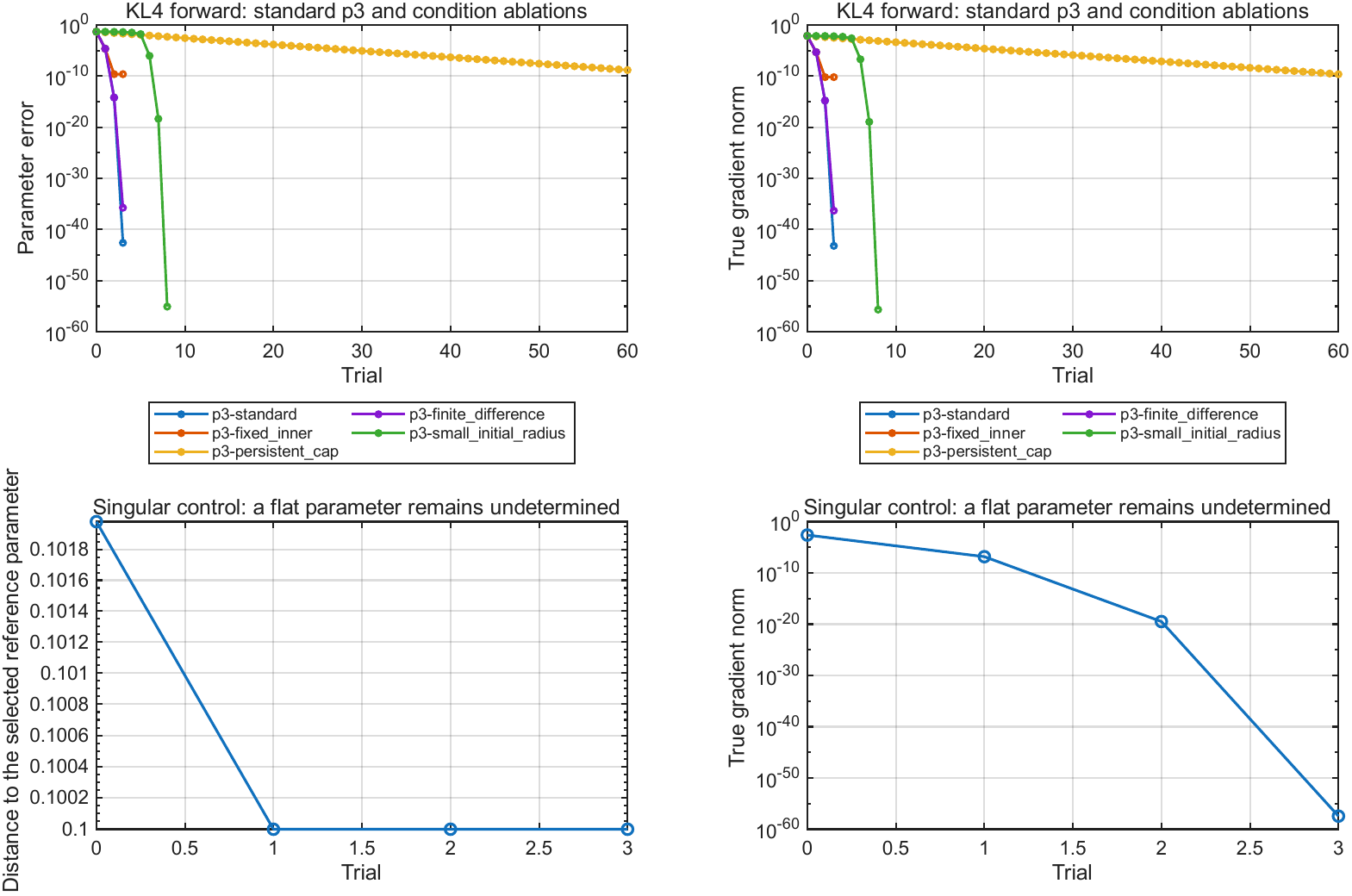}
  \caption{E5 individual condition ablations on the first forward KL+quartic problem. Top: parameter errors and true gradients for standard $p=3$ and four controls. Bottom: a flat parameter direction permits vanishing gradient without convergence to the selected reference. The text reports each control's outcome.}
  \label{fig:e5-ablations}
\end{figure}

\subsection{Preference continuation and sensitivity prediction}
\label{app:e6-details}

\paragraph{Problem construction.}
E6 evaluates the nonsingular preference sensitivity formula
in Section~\ref{app:continuation}.
We use the affine probability chart
$q_\theta=q_0+B\theta>0$, with $(K,d,m)=(24,3,3)$.
For seed $s\in\{1,\ldots,10\}$, the MATLAB random stream is
\texttt{rng(61000+s,'twister')}.
Draw $z\sim\mathcal N(0,0.35^2I_K)$ and a standard Gaussian
matrix $A\in\R^{K\times d}$, and define
\begin{equation}
  q_0=\softmax(z),\qquad
  \widetilde A=A-\mathbf 1(q_0^\top A),\qquad
  B=
  \frac{0.4\,\operatorname{diag}(q_0)\widetilde A}
       {\max_k\norm{\widetilde A_{k,:}}_2}.
  \label{eq:e6-chart}
\end{equation}
Then $\mathbf 1^\top B=0$.
For independent $u_i\sim\mathcal N(0,I_d)$, set
\begin{equation}
  \vartheta_i=1.2\,\frac{u_i}{\norm{u_i}_2},
  \qquad
  P_i=q_0+B\vartheta_i,\qquad i=1,2,3.
\end{equation}
The construction gives $P_i(k)\geq0.52q_0(k)>0$.
The implementation normalizes the target columns to remove
floating-point normalization error.
The same primitive random draws are used across generators
and divergence orientations within each seed.
We use the JS and KL+quartic generators from Section~\ref{sec:generators}. KL4 in the figures denotes
KL+quartic.

\paragraph{Preference loops and initializers.}
The preference path is the closed triangle with vertices
\[
  (0.90,0.05,0.05)^\top,\qquad
  (0.05,0.90,0.05)^\top,\qquad
  (0.05,0.05,0.90)^\top.
\]
Each edge is subdivided into $L\in\{4,8,16\}$ equal
intervals, giving $3L+1\in\{13,25,49\}$ sampled points
including the repeated starting vertex.
Both traversal directions are tested independently of the divergence orientation.

Fix a divergence orientation $\sigma$.
Let $\theta_{j-1}$ denote the preceding corrected solution,
$\delta\lambda_j=\lambda_j-\lambda_{j-1}$, and
\[
  G_{j-1}
  =
  \bigl[
    \nabla F_1^\sigma(\theta_{j-1}),\,
    \nabla F_2^\sigma(\theta_{j-1}),\,
    \nabla F_3^\sigma(\theta_{j-1})
  \bigr],
  \qquad
  H_{j-1}
  =
  \nabla^2F_{\lambda_{j-1}}^\sigma(\theta_{j-1}),
\]
where the component objectives are defined in~\eqref{eq:moo}.

At subsequent preferences, the three initializers are
\begin{equation}
  \begin{aligned}
    \theta_{j,0}^{\mathrm{Cold}}&=0,\\
    \theta_{j,0}^{\mathrm{Warm}}&=\theta_{j-1},\\
    \theta_{j,0}^{\mathrm{Tangent}}
      &=\theta_{j-1}-H_{j-1}^{-1}G_{j-1}\delta\lambda_j.
  \end{aligned}
  \label{eq:e6-initializers}
\end{equation}
Since $\mathbf1^\top\delta\lambda_j=0$, Tangent equals~\eqref{eq:preference-path}. Each method uses its own preceding corrected endpoint, with a common first initialization $\theta=0$. Tangent reuses retained component jets. Independent references never initialize timed paths.

\paragraph{Common corrector and safeguards.}
All three strategies use the same order-two nonlinear
calibrated model, Cauchy decrease safeguard, and true-objective
acceptance test.
The initial trust-region radius is reset to $1$ at every
preference, with maximum radius $8$.
We use $(\eta_1,\eta_2)=(0.1,0.75)$ and Cauchy constant
$\kappa=0.1$.
The radius is halved after rejection and doubled after a
very successful step, subject to the maximum radius.
Each applied ball is restricted to $0.45$ times the
distance to the affine-chart boundary.

The nonlinear model solve allows at most 25 inner Newton
iterations and uses the relative gradient target
\[
  \max\{10^{-12},\,0.01\norm{\nabla F_\lambda^\sigma(\theta)}\}.
\]
The solver also evaluates the model's center Newton candidate, retains the larger predicted reduction, and falls back to a Cauchy step if needed.
Each preference solve permits at most 200 trials and stops
when $\norm{\nabla F_\lambda^\sigma}\leq10^{-8}$.
These settings assess continuation efficiency. They are
not the strict high-order inner forcing protocol of E5.

Tangent falls back to Warm after a failed Hessian definiteness/conditioning check and may shorten predictions to preserve positivity. Neither occurs here. All methods freshly evaluate true component jets at each initializer.

\paragraph{Reference accuracy and validation.}
Independent damped Newton solves from zero target gradient norm
$10^{-13}$. The maximum reference gradient is $9.88\times10^{-14}$,
and the minimum Hessian eigenvalue is $1.96\times10^{-3}$.
Derivative, aggregation, center-calibration, and stable-difference
checks all pass. The maximum normalized finite-difference discrepancy
for the preference derivative is $8.24\times10^{-9}$.

The primary design has
$10\times2\times2\times2\times3\times(13+25+49)=20{,}880$
point solves, indexed by seed, generator, divergence orientation,
traversal, initializer, and grid. Three timing repetitions give 62,640
solves. All pass stationarity and reference-objective checks, with
identical work counts across repetitions. The largest primary final
gradient is $9.96\times10^{-9}$, and the largest absolute reference
objective gap is $9.65\times10^{-15}$.

\paragraph{Independent local prediction maps.}
To isolate initializer accuracy, use
$\lambda_0=(0.2,0.3,0.5)^\top$ and a random unit vector
$v\in\R^3$ satisfying $\mathbf 1^\top v=0$.
For both signs and $h=0.04\,2^{-r}$, $r=0,\ldots,5$,
compute independent reference solutions at
$\lambda_0$ and $\lambda_0+\delta\lambda$,
where $\delta\lambda=\pm hv$.
The two errors are
\begin{align}
  e_{\mathrm W}(h)
  &=
  \norm{
    \theta^\star(\lambda_0)
    -\theta^\star(\lambda_0+\delta\lambda)
  },\\
  e_{\mathrm T}(h)
  &=
  \norm{
    \theta^\star(\lambda_0)
    -H_0^{-1}G_0\delta\lambda
    -\theta^\star(\lambda_0+\delta\lambda)
  }.
\end{align}
Here $H_0$ and $G_0$ are evaluated at the base reference.
On a smooth nonsingular solution branch, the expected
orders are $e_{\mathrm W}(h)=O(h)$ and
$e_{\mathrm T}(h)=O(h^2)$.

Fits exclude errors below the diagnostic resolution floor
\[
  \max\left\{
    10^{-12},
    20\left(
      \frac{\norm{g_{\mathrm{ref},0}}}{\mu_{\mathrm{ref},0}}
      +
      \frac{\norm{g_{\mathrm{ref},h}}}{\mu_{\mathrm{ref},h}}
    \right)
  \right\},
\]
where $\mu_{\mathrm{ref}}$ is the smallest reference
Hessian eigenvalue.
Tangent fits additionally require a feasible prediction.
All 160 signed fits are valid and use all six step sizes.
We average the two signed slopes within each seed before
computing the mean and standard deviation across ten seeds.

\begin{table}[!htbp]
  \centering
  \small
\setlength{\tabcolsep}{5pt}
\renewcommand{\arraystretch}{1.04}
\setlength{\abovecaptionskip}{0pt}
\setlength{\belowcaptionskip}{4pt}
  \caption{
  E6 local prediction orders.
  Entries are mean $\pm$ one standard deviation across
  ten seeds, after averaging the two signed slopes
  within each seed.
  The reference orders are one for Warm and two for Tangent.
  }
  \label{tab:e6-orders}
  \begin{tabular}{@{}llcc@{}}
    \toprule
    Generator & Orientation & Warm & Tangent\\
    \midrule
    JS & Forward
      & $0.999989\pm0.000005$
      & $2.000089\pm0.000283$\\
    JS & Reverse
      & $0.999989\pm0.000005$
      & $2.000089\pm0.000283$\\
    KL+quartic & Forward
      & $1.000184\pm0.000122$
      & $2.000039\pm0.000690$\\
    KL+quartic & Reverse
      & $1.000238\pm0.000082$
      & $2.000341\pm0.000808$\\
    \bottomrule
  \end{tabular}
\end{table}

\paragraph{Timing and work accounting.}
After untimed warm-up, each method/traversal is timed three times in randomized method order. For each seed, generator, divergence orientation, and grid, paired speedups use median whole-path times over repetitions, then the median ratio over traversals, followed by the median and IQR across seeds. Repetitions and traversals are not independent instances. A failed traversal invalidates its seed summary.

Whole-path times include the common first solve, aggregation, initialization, tangent algebra, model construction, inner solves, and acceptance checks. References, diagnostics, tables, plots, and export are excluded. A true-jet pass evaluates all component values, gradients, and Hessians. Aggregation, predictor queries, objective differences, and linear solves are separately counted but included in runtime. Jet counts are not total operation counts.

\begin{figure}[!htbp]
  \centering
  \includegraphics[width=0.836\linewidth]{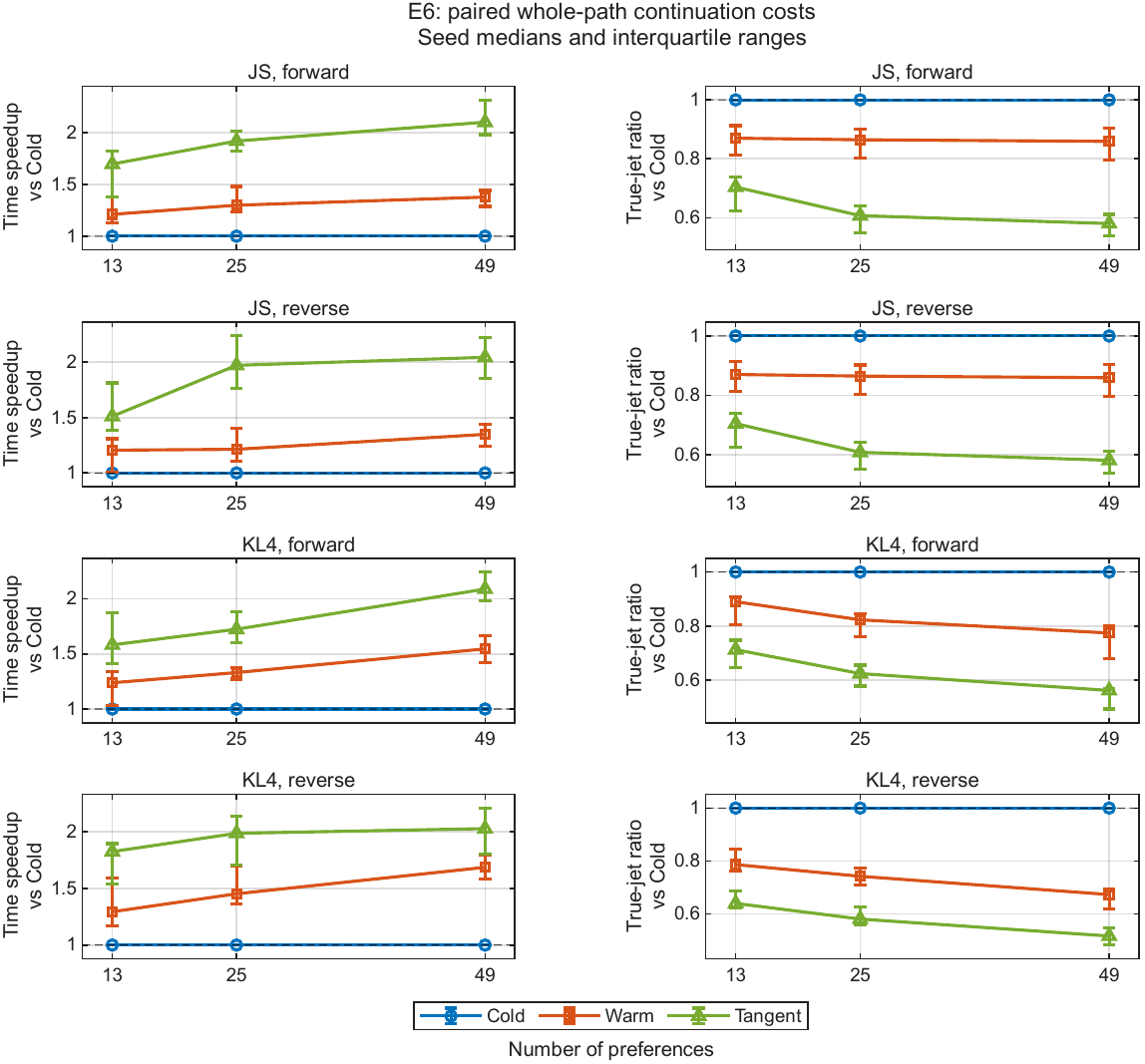}
  \caption{E6 whole-path costs by generator/orientation: runtime speedup against Cold (left) and true-jet pass ratio (right). Paired ratios use median times over three repetitions, then medians over traversals within seeds and medians/IQRs over ten seeds. Initialization and the common first solve are included.}
  \label{fig:e6-cost}
\end{figure}

Table~\ref{tab:e6-cost} in the main text reports the finest-grid costs.

\paragraph{Pathwise initializer quality and objective tradeoffs.}
On the finest grid, Tangent reduces pre-correction gradients after the common first point. Changes at triangle vertices reflect the new preference direction (Figure~\ref{fig:e6-profiles}).

All corrected paths agree numerically with independent references (Figure~\ref{fig:e6-tradeoffs}). This loop samples supported tradeoffs, not the full Pareto surface. Singular-Hessian crossings, branch switching, and real-data performance are untested. JS orientations are algebraically identical, not independent replications.

\begin{figure}[!htbp]
  \centering
  \includegraphics[width=0.750\linewidth]{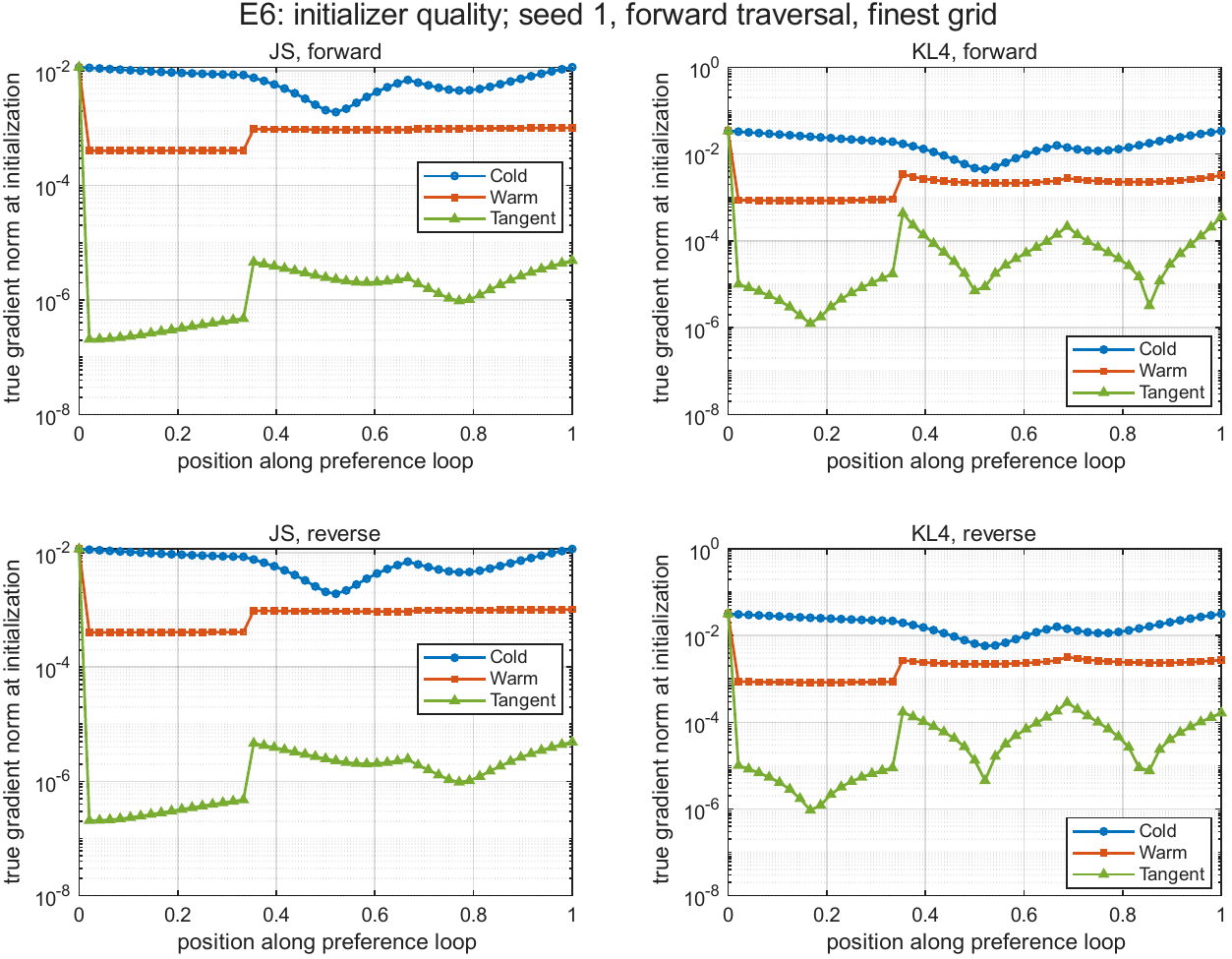}
  \caption{E6 pre-correction gradient norms, seed 1, forward traversal of the 49-point loop, both divergence orientations. All methods share the first cold initialization. The reference line is the corrector tolerance $10^{-8}$.}
  \label{fig:e6-profiles}
\end{figure}

\begin{figure}[!htbp]

  \centering
  \includegraphics[width=0.750\linewidth]{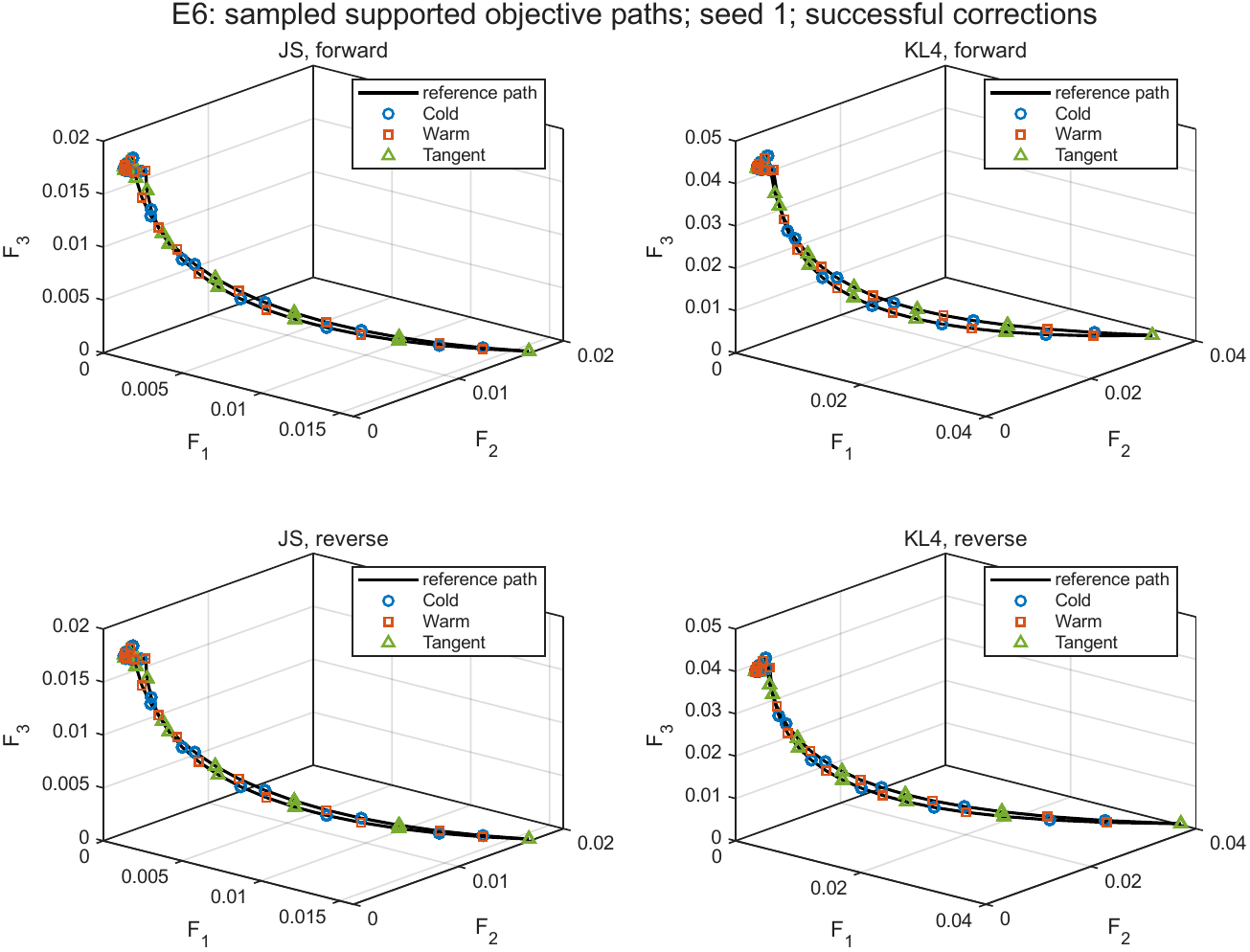}
  \caption{E6 corrected objective paths, seed 1, finest grid, forward traversal. The black line connects independent references. Staggered markers show Cold, Warm, and Tangent solutions. All sampled points enter validation. The loop does not cover the full Pareto surface.}
  \label{fig:e6-tradeoffs}
\end{figure}

\subsection{Template choice after second-order calibration}
\label{app:e7-details}

\paragraph{Compared models.}
E7 varies the nonlinear template while fixing $p=2$.
For $\Acal_\alpha=a_f\sum_i\lambda_iD_\alpha^\sigma(P_i,Q_\theta)$,
the decomposition~\eqref{eq:taylor-remainder} becomes
\begin{equation}
 m_{\alpha,\theta}^{(2)}(s)=T_{2,\theta}F(s)
 +\Acal_\alpha(\theta+s)-T_{2,\theta}\Acal_\alpha(s),
 \label{eq:e7-template-decomposition}
\end{equation}
where $T_{2,\theta}$ is the quadratic Taylor polynomial at $\theta$.
The templates share the true quadratic expansion and differ in their
higher-order terms.

We compare CAL-canonical with $\alpha=\alpha_f$,
CAL-minus with $\alpha=\alpha_f-1$,
CAL-plus with $\alpha=\alpha_f+1$, and
Taylor-p2 with model $T_{2,\theta}F$.
The fifth method uses the value-aligned uncalibrated model
\[
  m_{\theta}^{\mathrm{uncal}}(s)
  =F(\theta)+\Acal_{\alpha_f}(\theta+s)
   -\Acal_{\alpha_f}(\theta).
\]
This control omits residual derivative corrections. All $\alpha$ predictors use exact aggregation with the required scale.

\paragraph{Problem construction.}
E7 uses $Q_\theta=\softmax(b+B\theta)$ with $(K,d,m)=(128,16,32)$.
As in E4, $b_k\sim\mathcal N(0,0.25^2)$ and $B=1.2U$, where $U$
is the thin QR factor of a column-centered Gaussian matrix.
Targets and initialization follow~\eqref{eq:e4-targets}--\eqref{eq:e4-initialization},
with $\tau\in\{0.1,0.5,1,1.5\}$ and Dirichlet$(1,\ldots,1)$ preferences.
Twenty seeds use MATLAB's Twister stream initialized by $71000+\text{seed}$.
Primitive draws are shared across heterogeneities, generators,
orientations, and methods. JS, KLChi, and KL4 in both orientations give
480 cases and 2,400 primary method runs. JS orientations coincide and
are not independent replications.

\paragraph{Common optimization protocol.}
All runs use double precision, a true-gradient tolerance of $10^{-6}$,
and limits of 360 trials and 120 accepted steps. The initial, maximum,
and minimum radii are $0.5$, $2$, and $10^{-12}$.
The thresholds are $(\eta_1,\eta_2)=(0.1,0.75)$.
Rejection halves the radius, while $\rho\geq\eta_2$ doubles it up to
the maximum. The Cauchy-decrease check uses coefficient $0.1$.

All methods use at most 15 regularized Newton iterations with ball
projection, Armijo backtracking, a Cauchy candidate, and model-gradient
threshold $\max\{10^{-10},0.01\norm{\nabla m(0)}\}$.
Acceptance uses the true objective, and success uses its gradient.
An uncalibrated run that stops at surrogate stationarity but fails the
true-gradient test is unsuccessful. Taylor-p2 uses the same inner routine
to isolate template effects, rather than a separate quadratic solver.

Independent Taylor-p2 references start from zero, use tolerance
$10^{-10}$, and have larger budgets. All 480 converge, with maximum
gradient norm $9.71\times10^{-11}$ and minimum observed Hessian
eigenvalue $3.95\times10^{-3}$. They are local numerical references,
not global-optimality certificates.

\paragraph{Fixed-center model diagnostics.}
Each case is probed at three centers:
$\theta=0$, the common initialization, and the
independent reference parameter.
At every center, we use two random unit directions,
both signs, and step norms
\[
  h\in\{0.2,0.1,0.05,0.02,0.01,0.005,0.002,0.001\}.
\]
The normalized errors are
\begin{align}
  e_v(\theta,s)
  &=
  \frac{\abs{F(\theta+s)-m_\theta(s)}}
       {1+\abs{F(\theta)}},
  \notag\\
  e_g(\theta,s)
  &=
  \frac{\norm{\nabla F(\theta+s)-\nabla_s m_\theta(s)}}
       {1+\norm{\nabla F(\theta+s)}}.
  \label{eq:e7-errors}
\end{align}
Stable objective differences are used when evaluating
small value errors.
All 230,400 probe records are finite.

Figures~\ref{fig:e7-gradient} and~\ref{fig:e7-value} use the preselected origin and $h=0.1$: medians over directions/signs within seeds, then medians and IQRs across seeds. IQRs measure dispersion, not confidence. Fixed-step errors do not estimate convergence order.

Diagnostic fits use radii $\{0.1,0.05,0.02,0.01\}$ and require all four errors above recorded roundoff thresholds. Of 57,600 value/gradient fits, 57,353 qualify. The remaining 247 are unavailable in the records. Neither slope agreement nor $R^2$ determines eligibility.

\begin{figure}[!htbp]
  \centering
  \includegraphics[width=0.874\linewidth]{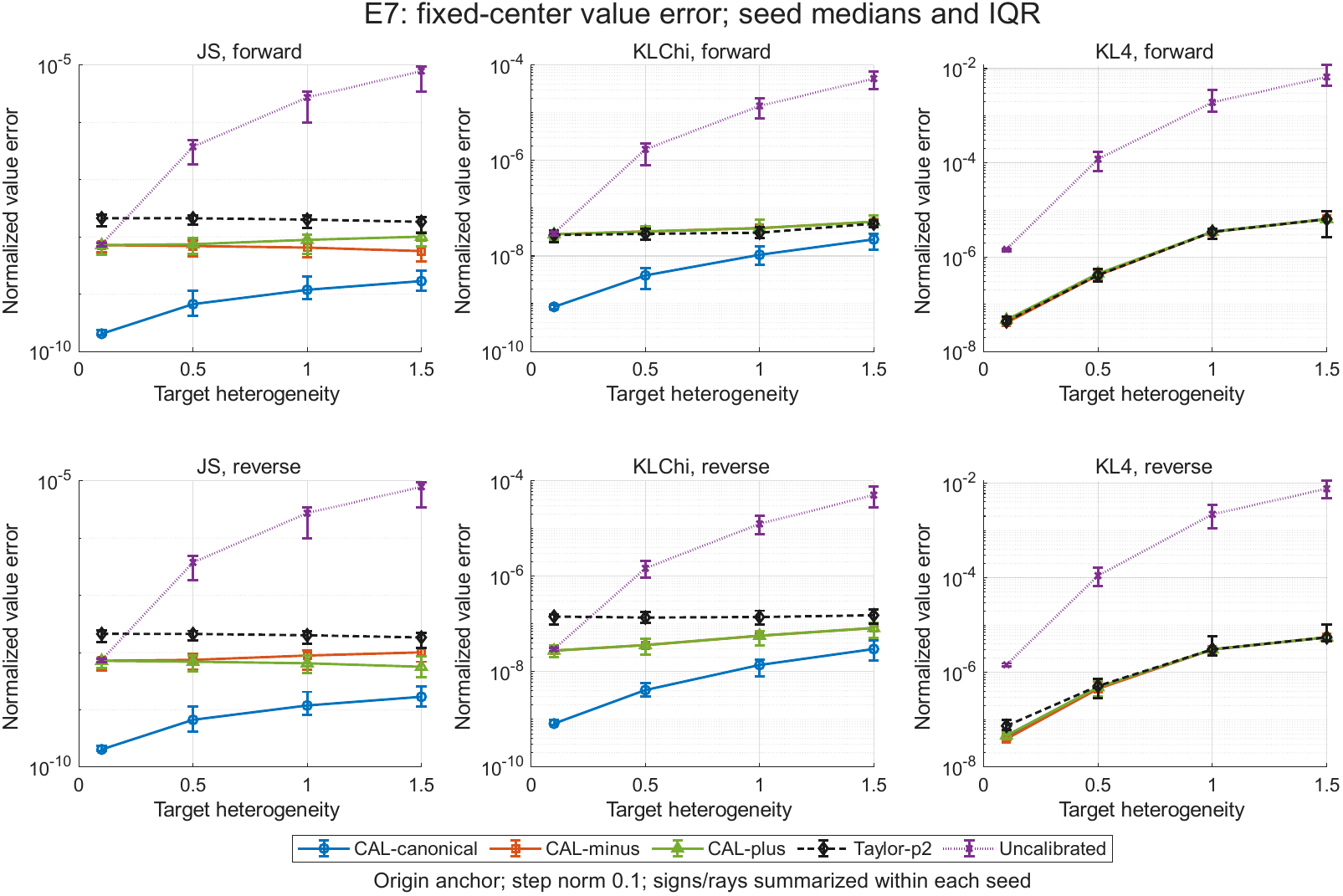}
  \caption{E7 normalized value errors with the probes and aggregation of Figure~\ref{fig:e7-gradient}. Canonical median errors improve on JS and KL+chi-square. Calibrated and Taylor curves nearly coincide for KL+quartic at larger heterogeneities. The horizontal axis varies targets, not step size.}
  \label{fig:e7-value}
\end{figure}

\paragraph{Convergence and timing results.}
Each of CAL-canonical, CAL-minus, CAL-plus, and
Taylor-p2 succeeds in all 480 cases.
Their largest final gradient norm is
$9.68\times10^{-7}$, and their largest absolute
objective discrepancy from the numerical reference
is $7.28\times10^{-11}$.
None of these 1,920 runs fails the Cauchy-decrease check.
Uncalibrated succeeds in 40/160 JS, 40/160 KL+chi-square, and 0/160 KL+quartic cases, only at $\tau=0.1$. Denominators include every attempt.

After untimed warm-up, three repetitions per case/method in randomized order give 7,200 timed runs with identical numerical outcomes and work counts. For method $M$, the paired speedup uses median times:
\[
  S_{\mathrm{time}}(M)
  =
  \frac{\operatorname{median}_r t_{\mathrm{Taylor},r}}
       {\operatorname{median}_r t_{M,r}}.
\]
A timing pair is eligible only if both methods succeed
and their final objectives satisfy
\[
  \abs{F_M-F_{\mathrm{Taylor}}}
  \leq
  10^{-8}\!
  \left(1+\max\{\abs{F_M},\abs{F_{\mathrm{Taylor}}}\}\right).
\]
Eligible ratios are summarized by medians/IQRs across 20 seeds. Repetitions are not independent instances, and groups without pairs are unplotted.

Timings include aggregation, true-jet evaluation,
model construction and calibration, inner solves,
and true-objective acceptance checks.
Independent references and offline model diagnostics
are excluded.
The uncalibrated control retains the common true-jet
evaluation pipeline, so its timing is not a benchmark
of the cheapest possible standalone predictor solver.

CAL-canonical's group-median speedups are $0.470$--$0.816$ (Figure~\ref{fig:e7-runtime}): better model accuracy does not yield lower runtime here. Figure~\ref{fig:e7-convergence} gives representative gradient histories. Success statistics use all seeds.

\begin{figure}[!htbp]
  \centering
  \includegraphics[width=0.864\linewidth]{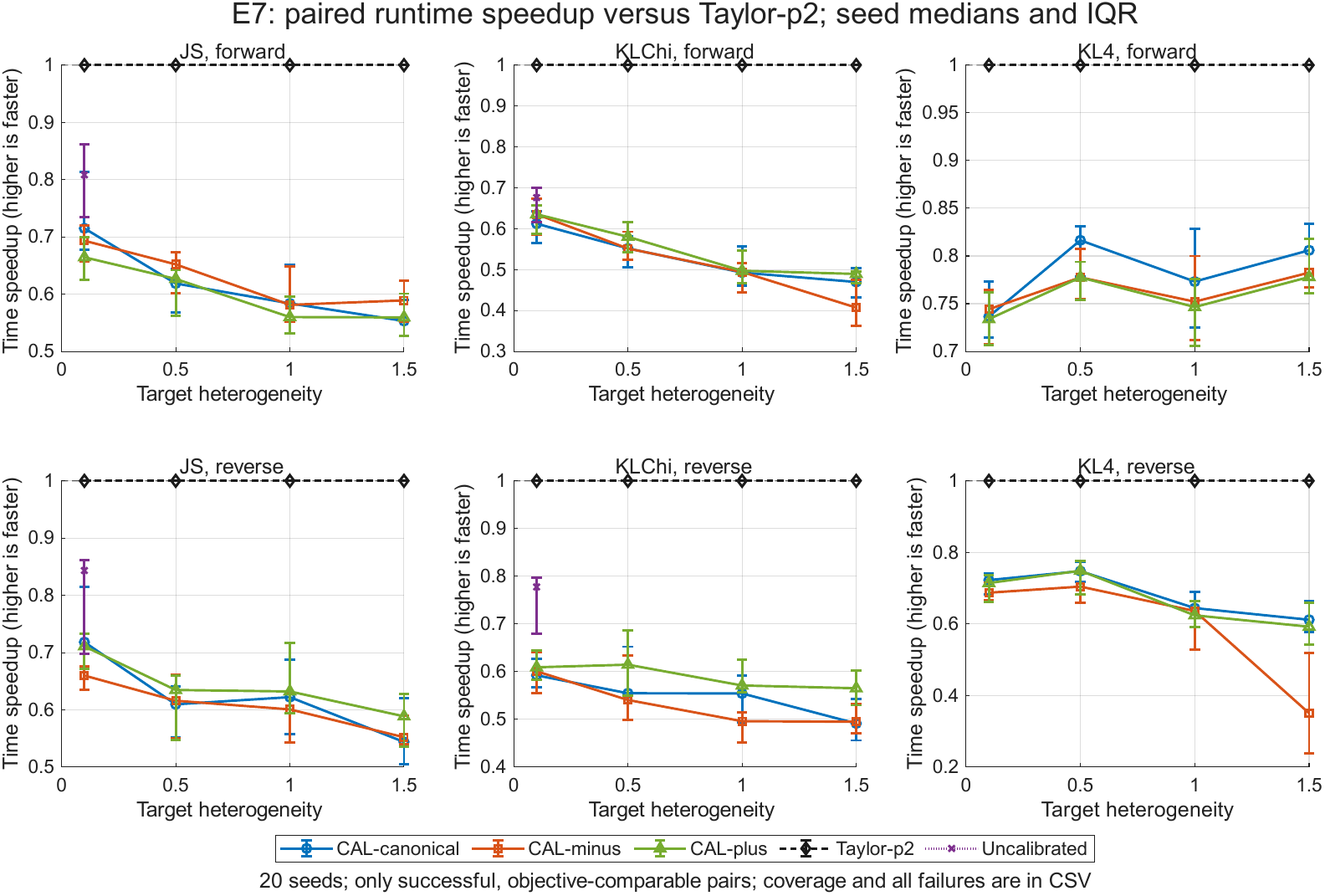}
  \caption{E7 paired speedups against Taylor-p2 (above one is faster), using median times over three repetitions and seed medians/IQRs. Each calibrated group has 20 successful, objective-comparable pairs. Uncalibrated pairs exist only for JS and KL+chi-square at $\tau=0.1$. Missing points indicate no eligible pairs.}
  \label{fig:e7-runtime}
\end{figure}

\begin{figure}[!htbp]
  \centering
  \includegraphics[width=0.874\linewidth]{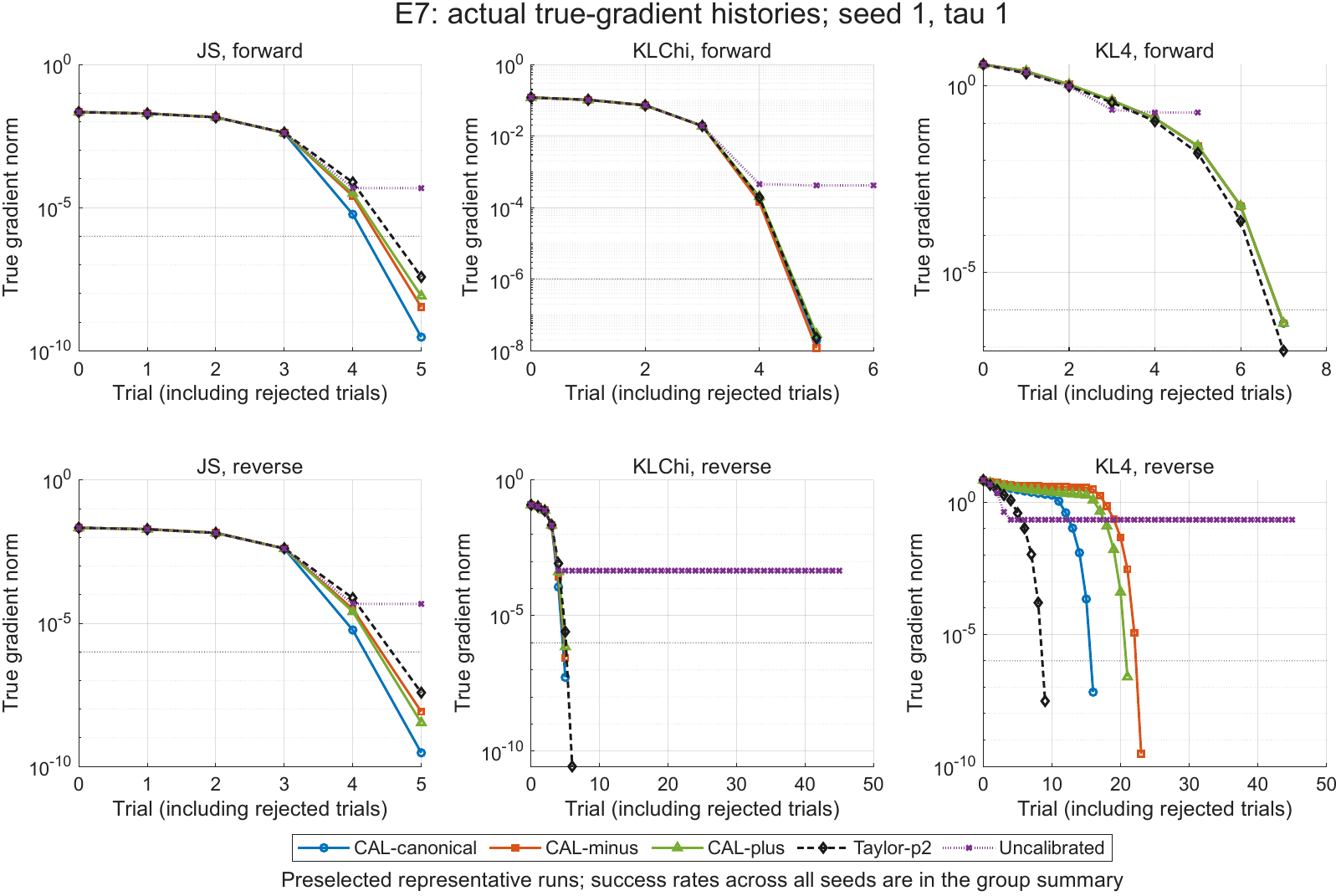}
  \caption{E7 true-gradient histories, preselected seed 1 and $\tau=1$, including rejected trials. These representative histories do not measure asymptotic local order or replace all-seed success statistics.}
  \label{fig:e7-convergence}
\end{figure}

\subsection{E8: learned latent distribution matching and generation}
\label{app:e8-details}

\paragraph{Restricted model family.}
Each target has six Gaussian components, while $Q_\theta$ has two with a
fixed shared covariance. Even the exact forward-KL aggregate may lie
outside this family. For example, preference $(.90,.05,.05)$ gives
$.645R_0+.220R_1+.135R_2$. These are weights of fitted latent laws,
not prescribed image-class frequencies.

\paragraph{Data and representation.}
We use the MNIST digits 0, 1, and 2, downsampled from $28\times28$ to $14\times14$ by averaging each $2\times2$ block and scaled to $[0,1]$. Each of five seeds selects 1,200 training and 180 validation images per class from the original training split, and 500 held-out images per class from the original test split. An autoencoder with widths $196$--$64$--$2$--$64$--$196$ uses tanh hidden activations, a linear bottleneck, and a sigmoid output. We minimize pixelwise binary cross-entropy plus a $2\times10^{-4}$ latent squared-norm penalty with Adam, learning rate $10^{-3}$ and batch size 128, for at most 120 epochs. Validation early stopping has patience 15 and restores the best checkpoint. Training-encoding statistics standardize the latent coordinates. The encoder, decoder, and standardization are then frozen for all optimization methods and preferences within that seed.

\begin{figure}[!htbp]
 \centering
 \includegraphics[width=0.807\linewidth]{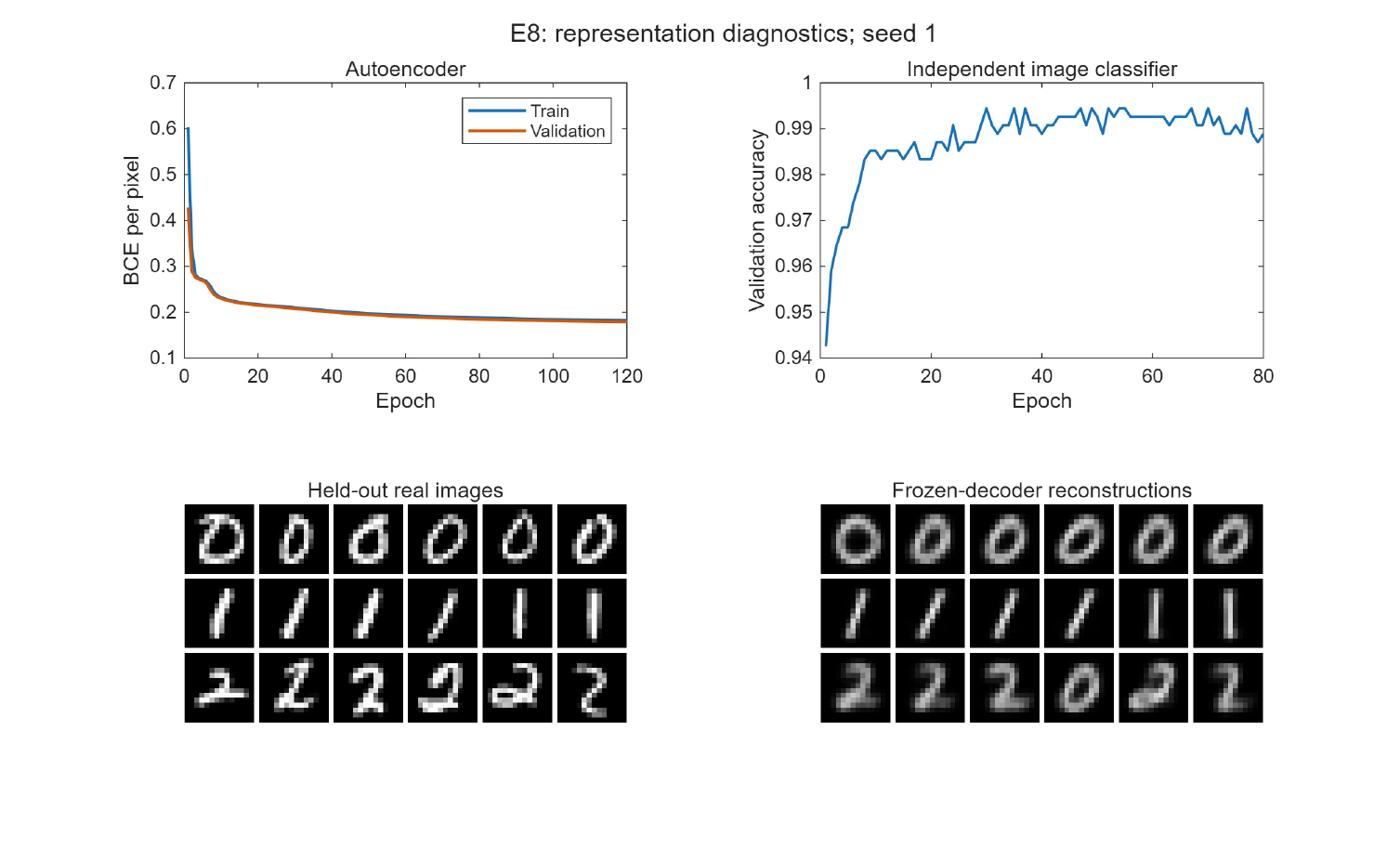}
 \caption{E8 representation diagnostics for seed 1: training and validation reconstruction loss, validation accuracy of an independent image classifier, and held-out inputs with frozen-decoder reconstructions.}
 \label{fig:e8-representation}
\end{figure}

\paragraph{Targets and optimized model.}
For each class, a two-component diagonal-covariance Gaussian mixture $R_k$ is fitted to training encodings with three EM starts, at most 80 iterations, and standard-deviation floor 0.35. The targets are
\begin{equation}
 \begin{aligned}
 P_1&=.7R_0+.2R_1+.1R_2,\\
 P_2&=.1R_0+.7R_1+.2R_2,\\
 P_3&=.2R_0+.1R_1+.7R_2.
 \end{aligned}
 \label{eq:e8-targets}
\end{equation}
 The optimized latent distribution $Q_\theta$ has two Gaussian components with a common fixed isotropic standard deviation equal to the larger of 0.5 and the median fitted teacher standard deviation. Its five free parameters are one mixture-weight logit and four mean coordinates. Means use a bounded tanh parameterization. All methods use the same initial parameter vector. There is no continuation warm start between preferences.

The ten preferences are the order-three simplex lattice transformed by $\lambda=.05\mathbf1+.85r$, where $r$ ranges over nonnegative triples with sum one and coordinates in $\{0,1/3,2/3,1\}$. This includes the uniform preference and the three $(.90,.05,.05)$ permutations. For JS and forward KL, we minimize $\sum_i\lambda_iD_f(P_i\|Q_\theta)$. With $a_f=1/2$ and $\alpha_f=0$, the normalized-grid JS template can equivalently be evaluated as
\begin{equation}
 \Acal(\theta)=2\left(1-\sum_x H_\lambda(x)\sqrt{q_\theta(x)}\right),
 \qquad H_\lambda(x)=\sum_i\lambda_i\sqrt{p_i(x)},
 \label{eq:e8-js-template}
\end{equation}
where $p_i$ and $q_\theta$ denote probability masses on the quadrature grid. This expression retains the scale and constant in Theorem~\ref{thm:aggregation}.

\paragraph{Solvers, budgets, and timing.}
CAL-p1 and CAL-p2 use residual derivatives through orders one and two, respectively. The p1 implementation can still use the predictor Hessian inside its subproblem. Direct-LBFGS uses memory eight and Armijo backtracking. Direct-TR uses the true quadratic model and an eigendecomposition-based trust-region step. The uncalibrated control aligns the predictor value but omits residual derivative corrections. All methods use the same true-objective stopping tolerance $10^{-6}$, a budget of 220 trial iterations and 100 accepted steps, initial radius 0.5, maximum radius 2, and minimum radius $10^{-11}$. CAL inner solves use at most 12 iterations, relative gradient tolerance 0.01, and absolute floor $10^{-9}$. The design comprises 500 primary runs (five seeds, ten preferences, two generators, five methods), with three timing repetitions in rotated method order. Four independent reference starts per seed/preference/generator yield 400 reference solves, excluded from optimization timing and never used for initialization. References are local best-found solutions, not certified global optima.

\paragraph{Quadrature and endpoint audits.}
Optimization uses a fixed $97\times97$ midpoint grid and normalized probability masses. The integration box is chosen using a $10^{-10}$ Gaussian tail-mass threshold. Every endpoint is re-evaluated on independently refined $193\times193$ and $385\times385$ grids. Audit thresholds are $5\times10^{-6}$ for the refined gradient norm, $10^{-7}$ for objective discrepancies, $5\times10^{-6}$ for gradient discrepancies, and $10^{-7}$ for discrete-mass normalization error. The maximum observed objective and gradient discrepancies are $1.85594\times10^{-8}$ and $1.13101\times10^{-7}$. Maximum normalization error is $8.88\times10^{-16}$. Grid refinement checks numerical consistency, not rigorous continuous-objective error bounds. Mass normalization alone does not certify quadrature accuracy.

\paragraph{Generation and independent image diagnostics.}
We sample $z\sim Q_\theta$ and display the frozen decoder's mean image, without adding Bernoulli pixel noise. Image-space diagnostics do not enter the optimized latent divergence. An independent linear softmax classifier is trained on training images for at most 80 epochs with learning rate 0.005. Its minimum held-out accuracy across seeds is $98.27\%$. The minimum classification accuracy of reconstructed held-out images is $97.2\%$. Unbiased image MMD$^2$ is computed in a 24-dimensional PCA representation learned from training images, using a Gaussian kernel with training-data bandwidth. Each image diagnostic uses 256 latent draws. Class proportions and MMD$^2$ are summarized across seeds separately for each categorical preference ID. Failed endpoints remain in these descriptive diagnostics and are identified in the saved samples.

\paragraph{Results and interpretation.}
All 200 calibrated runs and all 200 direct runs attain audited stationarity. The uncalibrated predictor succeeds on all 50 KL cases but no JS case. Comparable-pair coverage against Direct-TR is 43/50 for JS CAL-p1 and 46/50 for JS CAL-p2, and 50/50 for both KL variants. Median paired speedups against Direct-TR are 0.108 and 0.131 for JS CAL-p1/p2, and 0.178 and 0.155 for KL: the three-target application establishes reliable matching and preference-controlled generation, not acceleration. Signed reference gaps retain both improvements and worse local solutions.

\begin{figure}[!htbp]
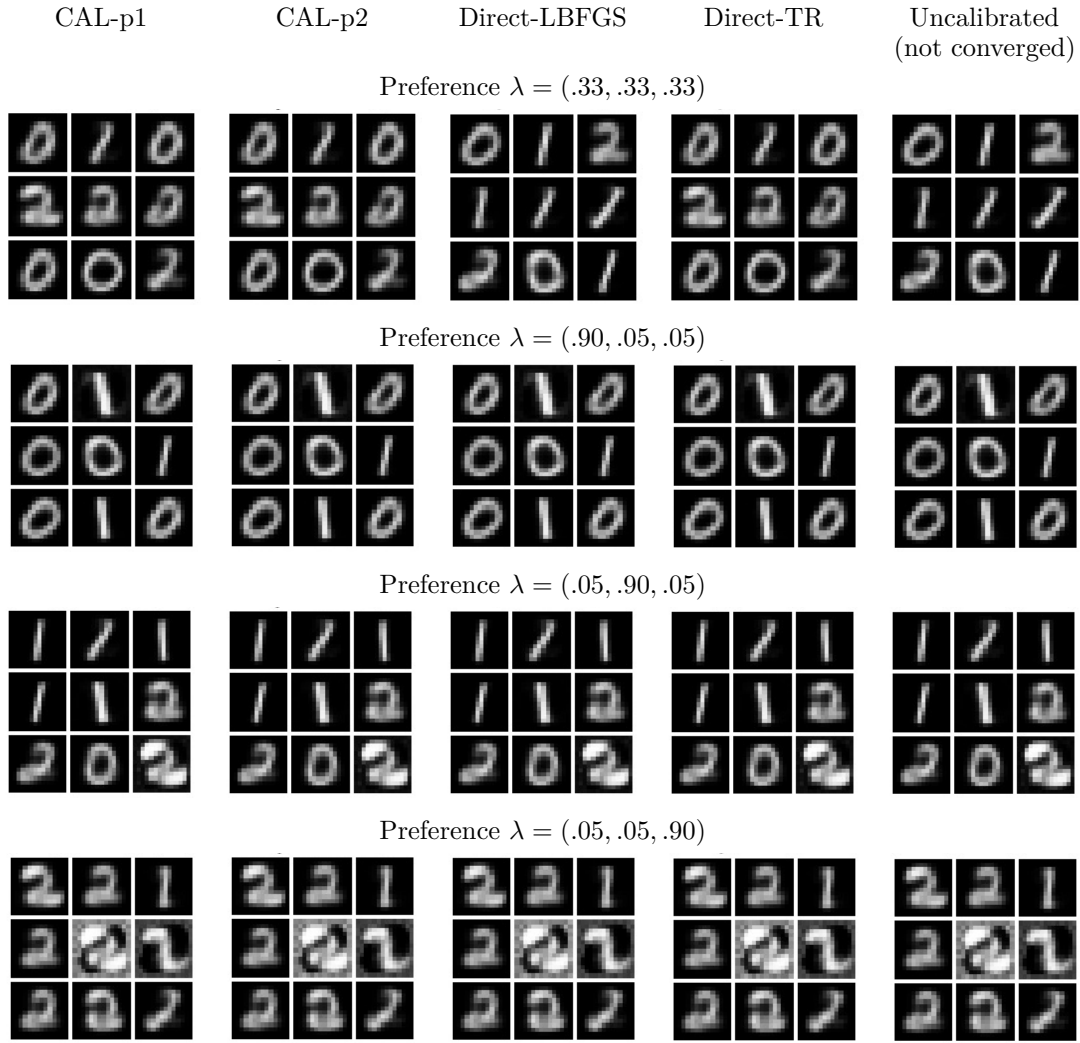

 \centering
 \small
 \setlength{\tabcolsep}{4pt}
\begin{tabular}{@{}ccccc@{}}
CAL-p1 & CAL-p2 & Direct-LBFGS & Direct-TR & Uncalibrated\\
 & & & & (not converged)\\[3pt]
\multicolumn{5}{c}{Preference $\lambda=(.33,.33,.33)$}\\[2pt]
\includegraphics[width=.160\linewidth,viewport=145 439 188 481,clip]{\EeightSamplesPDF}&
\includegraphics[width=.160\linewidth,viewport=145 344 188 386,clip]{\EeightSamplesPDF}&
\includegraphics[width=.160\linewidth,viewport=145 249 188 291,clip]{\EeightSamplesPDF}&
\includegraphics[width=.160\linewidth,viewport=145 154 188 196,clip]{\EeightSamplesPDF}&
\includegraphics[width=.160\linewidth,viewport=145 59 188 101,clip]{\EeightSamplesPDF}\\[4pt]
\multicolumn{5}{c}{Preference $\lambda=(.90,.05,.05)$}\\[2pt]
\includegraphics[width=.160\linewidth,viewport=313 439 357 481,clip]{\EeightSamplesPDF}&
\includegraphics[width=.160\linewidth,viewport=313 344 357 386,clip]{\EeightSamplesPDF}&
\includegraphics[width=.160\linewidth,viewport=313 249 357 291,clip]{\EeightSamplesPDF}&
\includegraphics[width=.160\linewidth,viewport=313 154 357 196,clip]{\EeightSamplesPDF}&
\includegraphics[width=.160\linewidth,viewport=313 59 357 101,clip]{\EeightSamplesPDF}\\[4pt]
\multicolumn{5}{c}{Preference $\lambda=(.05,.90,.05)$}\\[2pt]
\includegraphics[width=.160\linewidth,viewport=482 439 526 481,clip]{\EeightSamplesPDF}&
\includegraphics[width=.160\linewidth,viewport=482 344 526 386,clip]{\EeightSamplesPDF}&
\includegraphics[width=.160\linewidth,viewport=482 249 526 291,clip]{\EeightSamplesPDF}&
\includegraphics[width=.160\linewidth,viewport=482 154 526 196,clip]{\EeightSamplesPDF}&
\includegraphics[width=.160\linewidth,viewport=482 59 526 101,clip]{\EeightSamplesPDF}\\[4pt]
\multicolumn{5}{c}{Preference $\lambda=(.05,.05,.90)$}\\[2pt]
\includegraphics[width=.160\linewidth,viewport=650 439 694 481,clip]{\EeightSamplesPDF}&
\includegraphics[width=.160\linewidth,viewport=650 344 694 386,clip]{\EeightSamplesPDF}&
\includegraphics[width=.160\linewidth,viewport=650 249 694 291,clip]{\EeightSamplesPDF}&
\includegraphics[width=.160\linewidth,viewport=650 154 694 196,clip]{\EeightSamplesPDF}&
\includegraphics[width=.160\linewidth,viewport=650 59 694 101,clip]{\EeightSamplesPDF}\\[4pt]
\end{tabular}
 \caption{E8 saved decoded samples, four JS preferences, seed 1. Columns compare methods. All displayed uncalibrated endpoints failed to converge. Mosaics retain the saved samples, cropping only surrounding whitespace. CAL-p2 also appears in Figure~\ref{fig:e8-main}.}
 \label{fig:e8-all-samples}
\end{figure}

\begin{figure}[!htbp]
 \centering
 \includegraphics[width=0.779\linewidth]{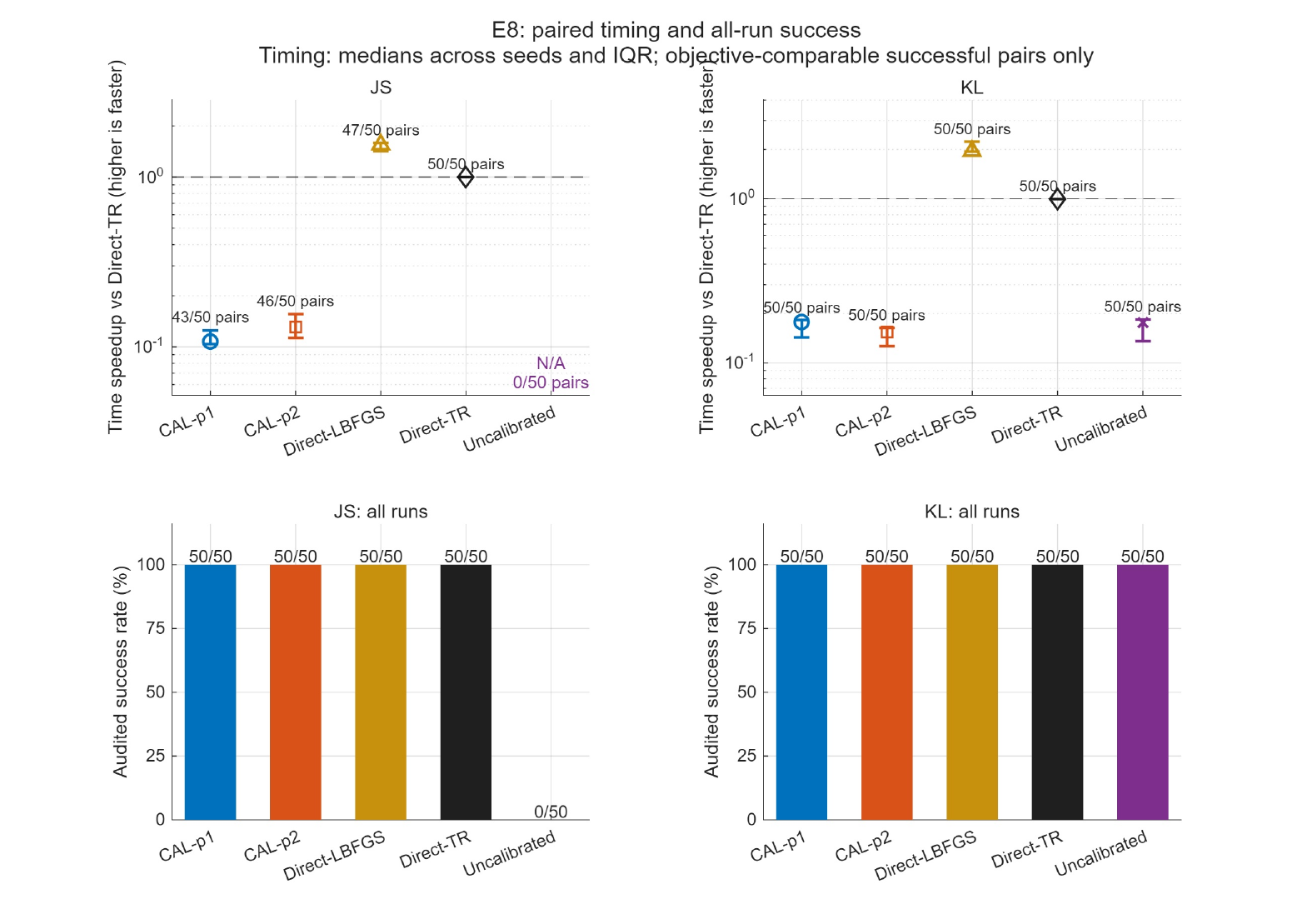}
 \caption{E8 paired runtime and all-run audited success. Timing uses only successful, objective-comparable pairs and reports pair counts. Success uses all attempted cases. N/A denotes absence of eligible timing pairs.}
 \label{fig:e8-cost}
\end{figure}
\begin{figure}[!htbp]
 \centering
 \includegraphics[width=0.817\linewidth]{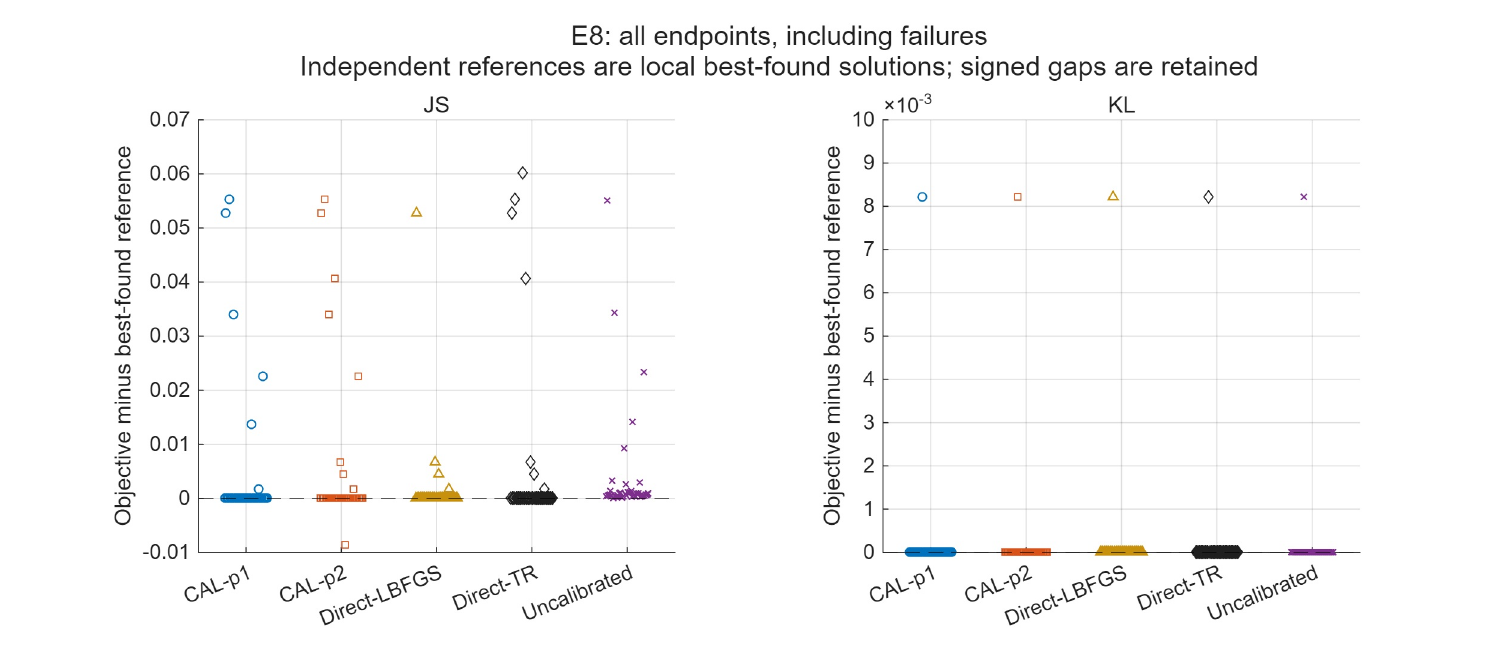}
 \caption{E8 signed objective gaps to independent local best-found references, including failed endpoints. A negative value means that the tested run improved on its reference.}
 \label{fig:e8-gaps}
\end{figure}

\subsection{E9: many-target cost crossover and implementation controls}
\label{app:e9-details}

\paragraph{Nested target banks.}
E9 reuses each E8 seed's three fitted latent targets and common initial model, then creates a bank of 300 six-component Gaussian mixtures. The first three are the original targets. Additional targets cycle through the three families, perturb component means by Gaussian noise of scale 0.15 times the student standard deviation, multiply component standard deviations by $\exp(.075\tanh Z)$, and multiply mixture weights by $\exp(.15\tanh Z)$ followed by renormalization, with independent standard-normal $Z$. The bank seed is $91000+$ the E8 seed. Target counts $3,10,30,100,300$ use nested prefixes of the same bank. The common integration box is fixed using the full 300-target bank, extending to nine component standard deviations beyond the extreme means when necessary. These are fitted and perturbed latent densities, not 300 separately trained neural networks.

The correlated bank isolates target-count effects within three families. Its speedups do not establish performance for independent, highly heterogeneous teachers. E7 tests heterogeneity effects on model accuracy, not this timing comparison.

For family $g(j)=1+\operatorname{mod}(j-1,3)$ and its count $n_g$ in the chosen prefix, target weights are $\lambda_j=\pi_{g(j)}/n_{g(j)}$. The five family preferences are uniform, the three permutations of $(.8,.1,.1)$, and $(.5,.35,.15)$. This keeps family-level preferences comparable as $m$ grows. The student remains the same five-parameter, two-component mixture used in E8. Both forward JS and forward KL are tested.

\paragraph{Shared implementations and cost boundaries.}
All solvers share cached teacher masses, log masses, and entropies on a $97\times97$ grid, block-vectorized evaluation with block size 32, and a per-solve cache for the current student density. A cached higher-order evaluation can satisfy a lower-order query. Direct KL objectives use exact aggregation, just as the KL template does.  Experiments run in double-precision MATLAB R2025b Update 5 on 64-bit Windows. Comparisons are paired on the same execution environment.

Total solve time includes preference-dependent setup, predictor aggregation, calibration and refreshes, inner solves, true-objective trial tests, and stopping checks. Shared bank/grid preparation, representation training, references, audits, profiling, and microbenchmarks are excluded equally. E9 measures optimization of fitted densities.

\paragraph{Baseline scope.}
CAL-p2 and Direct-TR use true values, gradients, and Hessians.
L-BFGS uses values and gradients. Adaptive weighting or constrained
preferences change the problem (Section~\ref{sec:related}).
Full-exponential-family Pareto constructions \citep{zhang2025} do not
directly solve restricted-mixture JS matching. The KL baselines already
aggregate exactly. Target-subsampled solvers are untested here.

The derivative-demand control (Stage A) is described in
Section~\ref{sec:e9} and Figure~\ref{fig:e9-implementation}.

\paragraph{Stage B: matched-order microbenchmarks.}
We separately time the true objective and aggregated predictor at three parameter points, with three repetitions. F returns a value, FG a value and gradient, and FGH a value, gradient, and Hessian. A timing batch targets at least 0.04 seconds, with a 500-call cap. Evaluation batches reach the requested duration. Aggregation preparation is timed separately, and is already charged to full solves. At $m=300$, JS true/predictor cost ratios are 188.788, 68.120, and 8.786 for F, FG, and FGH. KL ratios remain approximately one. See Figures~\ref{fig:e9-evaluation-cost} and~\ref{fig:e9-cost-ratio}.

\begin{figure}[!htbp]
 \centering
 \includegraphics[width=0.817\linewidth]{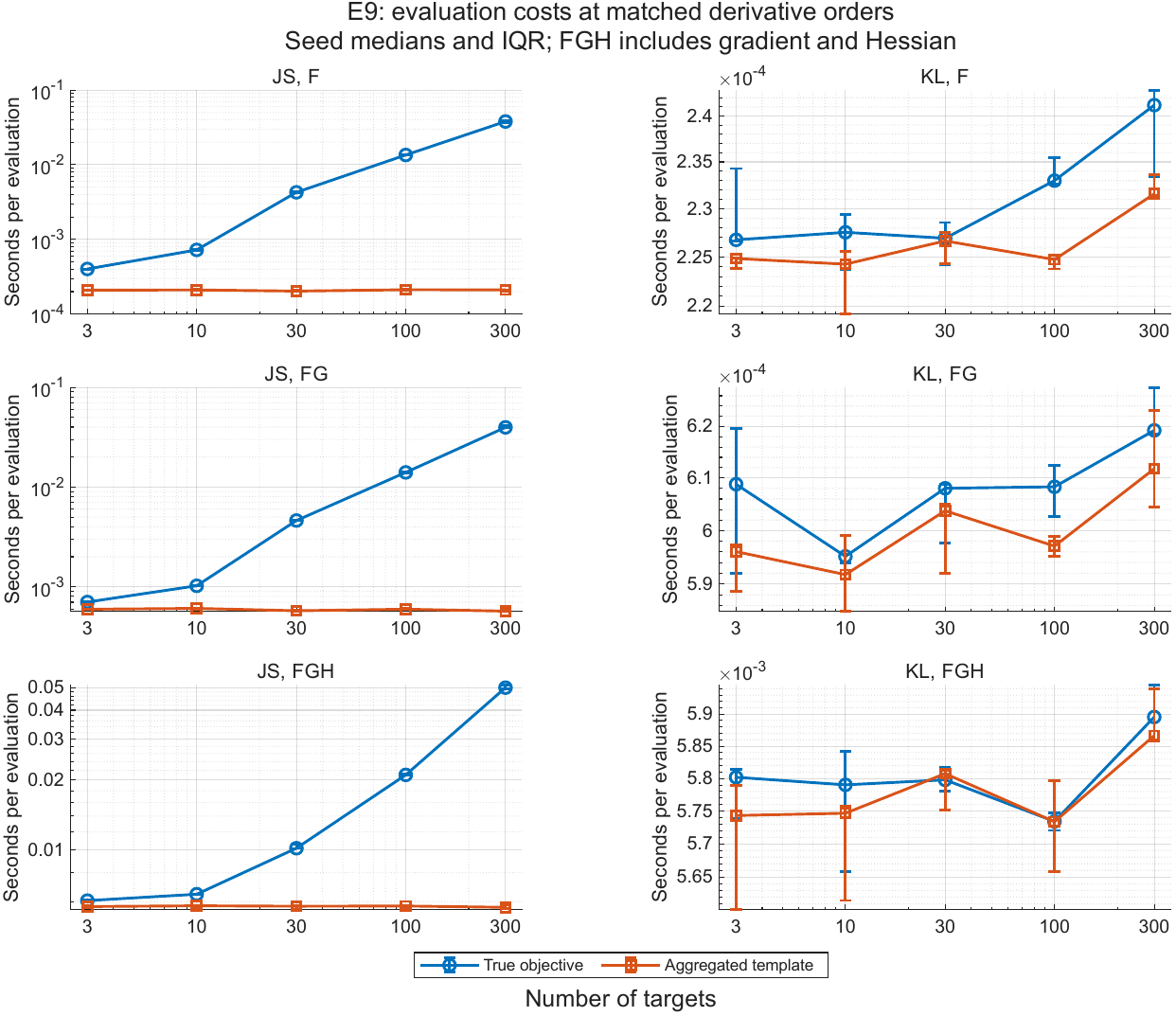}
 \caption{E9 absolute evaluation costs at matched derivative orders, summarized by seed medians and IQRs. FGH includes both gradients and Hessians. Predictor preprocessing is accounted for separately and included in total solves.}
 \label{fig:e9-evaluation-cost}
\end{figure}

\paragraph{Stage C: optimization, references, and audits.}
Five seeds, five target counts, five preferences, two generators, and four methods yield 1,000 primary runs and 3,000 timed repetitions. Before measured repeats, each method performs a warm-up solve limited to two trial iterations. The measured method order is rotated. The methods are CAL-p1, CAL-p2, Direct-LBFGS, and Direct-TR. CAL-p1 refreshes true FG information, CAL-p2 and Direct-TR use true FGH information, and L-BFGS uses true FG information. The CAL inner solver may request predictor Hessians for either calibration order. Independent local references use three Direct-TR starts and true-gradient tolerance $10^{-7}$, without initializing any compared run.

Primary solves use the E8 limits of 220 trials and 100 accepted steps, radii $(\Delta_0,\Delta_{\max},\Delta_{\min})=(0.5,2,10^{-11})$, and at most 12 inner iterations with relative tolerance 0.01 and absolute floor $10^{-9}$. They stop at true-gradient norm $10^{-6}$. Endpoint audits require at most $2\times10^{-6}$. All 1,000 pass, with maximum audited norm $9.948\times10^{-7}$. All 140 formula and derivative checks pass. Forty selected endpoints (seed 1, first preference, all target counts, both generators, all methods) are independently re-evaluated on a $193\times193$ grid. All satisfy objective-discrepancy tolerance $10^{-7}$ and gradient-discrepancy tolerance $5\times10^{-6}$. Saved validation checks also cover case uniqueness, repetition counts, independent references, and replay consistency.

\begin{figure}[!htbp]
 \centering
 \includegraphics[width=0.817\linewidth]{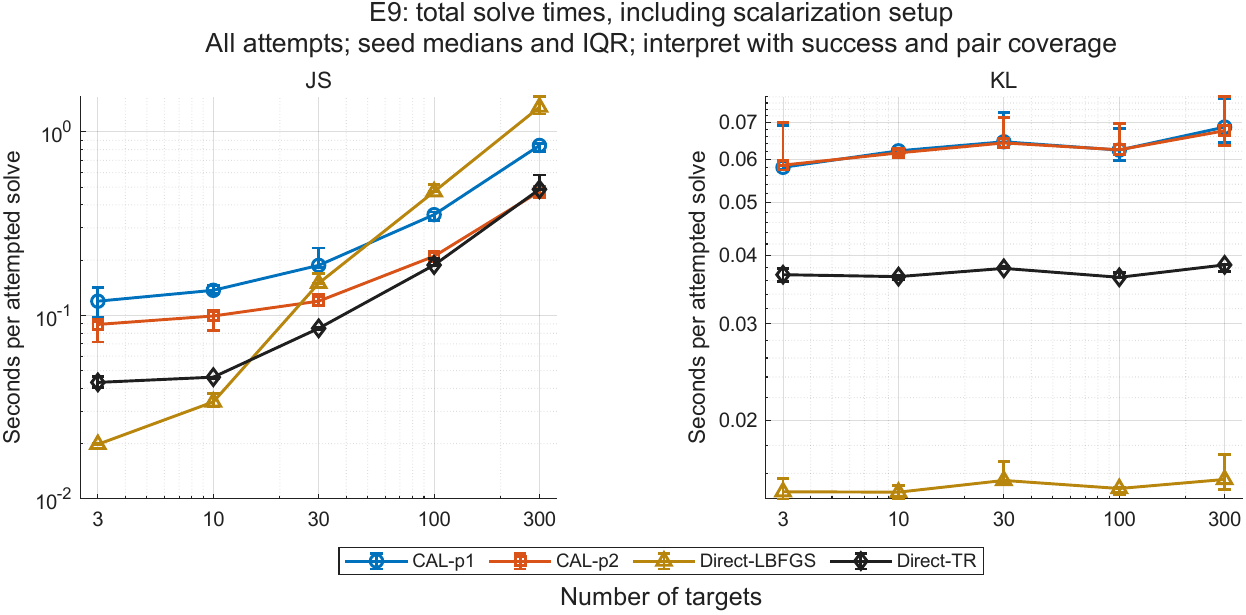}
 \caption{E9 total solve times including preference-dependent scalarization setup, over all attempts. Because different methods can reach different local solutions, these marginal medians supplement rather than replace paired speedups.}
 \label{fig:e9-runtime}
\end{figure}

\paragraph{Paired summaries and local-solution coverage.}
Let $t_{s,k}^{M}$ be the median of three unprofiled times for method $M$, seed $s$, and preference $k$ at a fixed generator and target count. For a baseline $B$, a pair is eligible if both runs pass the endpoint audit and
\begin{equation}
 \abs{F_\lambda(\theta_{s,k}^M)-F_\lambda(\theta_{s,k}^B)}\leq2\times10^{-6}.
 \label{eq:e9-pairing}
\end{equation}
Take median paired ratios $t_{s,k}^{B}/t_{s,k}^{M}$ over eligible preferences within each seed, then the median/IQR across five seeds. IQRs measure dispersion, not confidence. Report pair coverage separately from stationarity. Retain every endpoint in raw tables and signed-gap plots.

At $m=300$, JS CAL-p2 has 23/25 eligible pairs against L-BFGS and 24/25 against Direct-TR. Its median speedups are 3.1498 (IQR 3.0587--3.1510) and 1.1026 (IQR 1.0968--1.1513). Against Direct-TR, 17/24 individual pairs and four of five seed medians are faster. The two excluded L-BFGS pairs have higher CAL-p2 objectives, with gaps approximately 0.0262 and 0.0184. They are not evidence of improved solution quality. Direct methods remain faster in the KL exact-aggregation control. These controls delimit the cost regime in which JS aggregation yields a net benefit.

\begin{figure}[!htbp]
 \centering
 \includegraphics[width=0.817\linewidth]{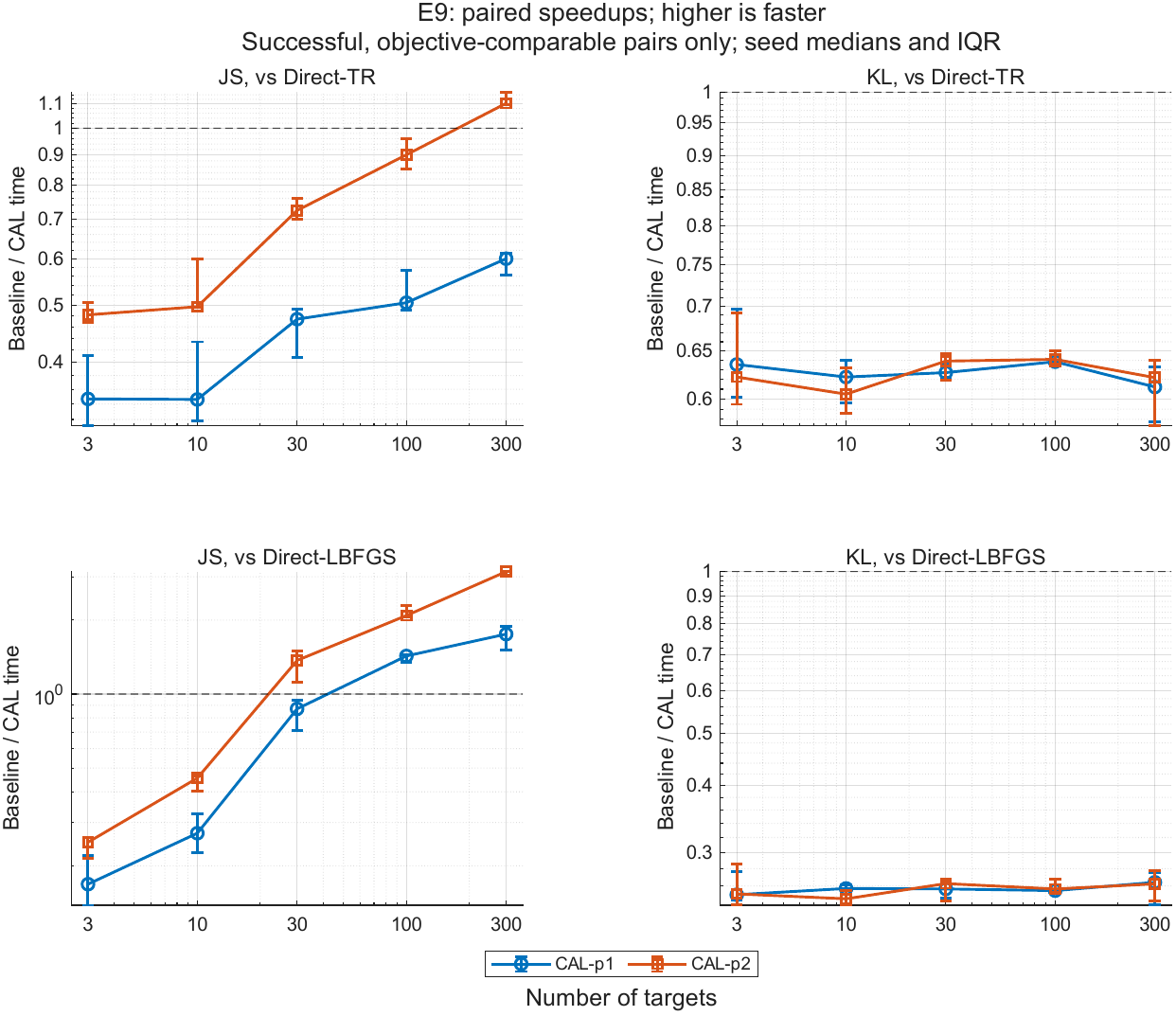}
 \caption{E9 full paired speedups for JS and the exact-aggregation KL control, against Direct-TR and Direct-LBFGS. Curves summarize successful objective-comparable pairs by seed medians and IQRs. CAL-p2 JS pair coverage is given in Table~\ref{tab:e9-main}. Raw records retain every attempted case.}
 \label{fig:e9-all-speedups}
\end{figure}

\paragraph{Diagnostic cost breakdown.}
Separate instrumented replays measure preparation, true/template evaluations, linear algebra, and other work (Figure~\ref{fig:e9-breakdown}). Component medians need not sum to median runtime. Speedups use unprofiled paired times. At $m=300$ on JS, the median CAL-p2 counts are five true FGH calls, four additional true F calls, 56 template F calls, and 14 template Hessian calls, compared with six true FGH and five additional F calls for Direct-TR. The resulting gain combines lower true-evaluation demand with inexpensive repeated template work.

\begin{figure}[!htbp]
 \centering
 \includegraphics[width=0.798\linewidth]{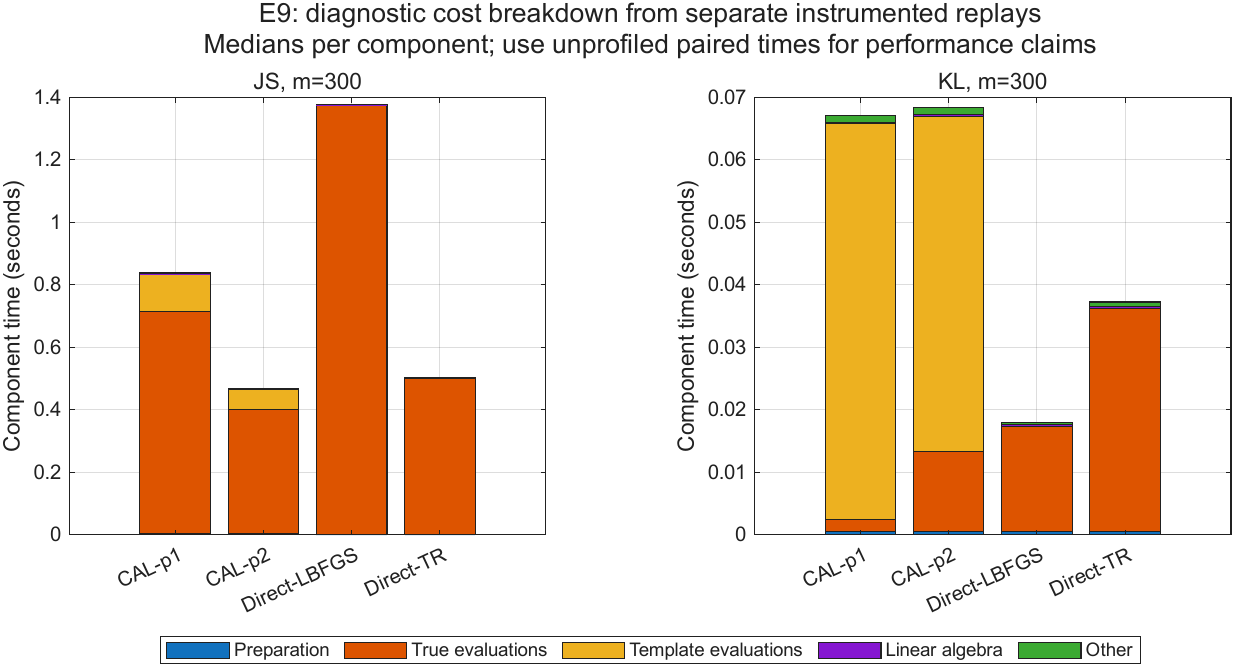}
 \caption{E9 diagnostic cost breakdown at $m=300$ from separate instrumented replays. Each component is summarized by its median. Performance claims use the unprofiled paired timing experiment.}
 \label{fig:e9-breakdown}
\end{figure}
\begin{figure}[!htbp]
 \centering
 \includegraphics[width=0.798\linewidth]{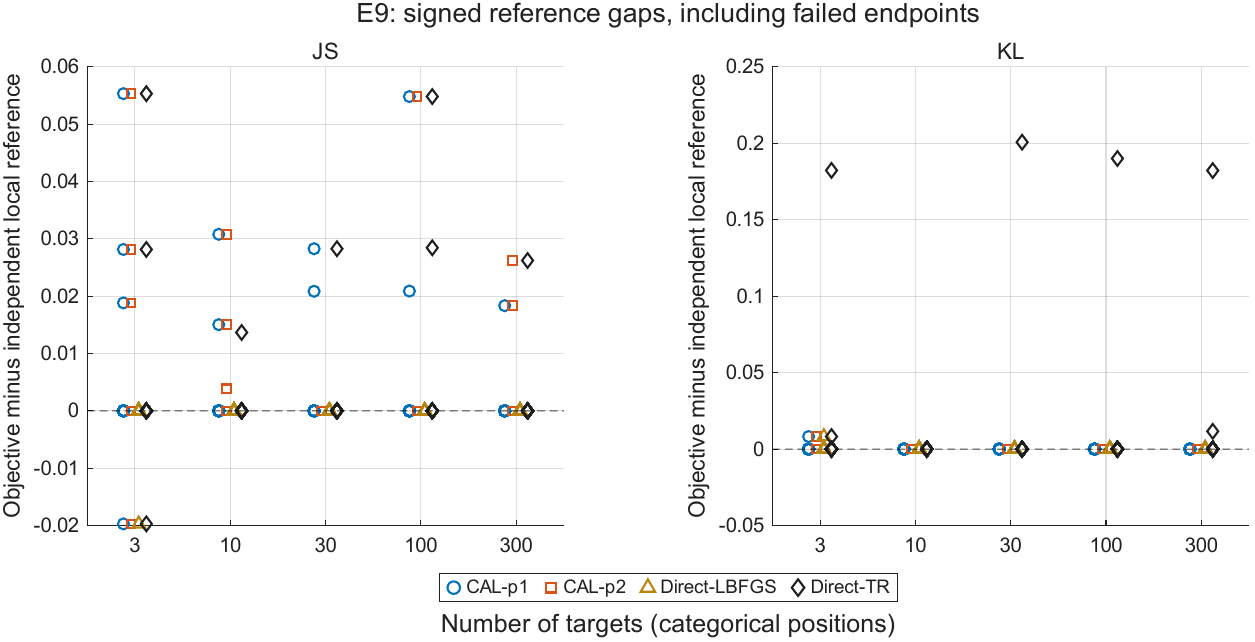}
 \caption{E9 signed objective gaps relative to independent local references, retaining all endpoints. Negative gaps are possible because references are best-found local solutions. These plots report solution variation separately from stationarity success and timing-pair coverage.}
 \label{fig:e9-gaps}
\end{figure}
\subsection{E10: canonical geometry and the leading local error}
\label{app:e10-details}

E10 tests the leading-error cancellation in Proposition~\ref{prop:leading-error} at fixed $p=2$, complementing E5's generic local orders and E7's finite-step errors. Third derivatives are offline diagnostics. All optimizers use only values, gradients, and Hessians.

\paragraph{Problems and target heterogeneity.}
Use $K=6$ categories, $d=2$ parameters, $m=4$ targets, seeds
$s=1,\ldots,5$, and MATLAB's Twister generator with seed $27100+s$.
Draw $b\sim\mathcal N(0,0.35^2I)$, a standard Gaussian $K\times d$
matrix $G_0$, target perturbations $\xi_i\sim\mathcal N(0,0.7^2I)$,
positive raw preferences $w_i=0.25+U_i$ with $U_i\sim U(0,1)$,
and one standard Gaussian direction. Set $q_0=\softmax(b)$, center
$G_0$ by subtracting $\mathbf1(q_0^\top G_0)$, and rescale it to a
matrix $G$ whose largest row norm is $0.4$. The affine family and
preferences are
\[
 q_\theta=q_0+B\theta,\qquad B=\operatorname{diag}(q_0)G,
 \qquad \Theta=\{\theta:q_\theta>0\},\qquad
 \lambda_i=w_i/\sum_jw_j.
\]
The laws are normalized since $\mathbf1^\top B=0$. For
$\eta\in\{0,0.01,0.05,0.2\}$, set
$P_i(\eta)=\softmax(\log q_0+\eta\xi_i)$, explicitly assigning
$P_i(0)=q_0$. Forward JS and KL4 use the generator conventions in Section~\ref{sec:generators}. Thus $\eta=0$ is common matching. Positive
$\eta$ perturbs the targets using shared random draws.

\paragraph{Comparators and common inner accuracy.}
CAL-canonical uses $\alpha_f$. CAL-alpha-minus/plus use
$\alpha_f-1$/$\alpha_f+1$, with the same scale $a_f$ and true residual
value, gradient, and Hessian correction. Taylor-p2 uses the true
quadratic Taylor model. All four models therefore agree through order
two at every anchor. All methods impose
\begin{equation}
 \norm{\nabla m_\theta(s)}\leq0.01\norm{\nabla F(\theta)}^3.
 \label{eq:e10-forcing}
\end{equation}
This strengthens Theorem~\ref{thm:local}'s $p=2$ forcing to reveal the coefficient in~\eqref{eq:leading-local-error}. Newton iterations on
the model start from zero, with at most 30 iterations and 45 backtracking
halvings, retaining feasible steps that reduce the model-gradient norm.
The trial radius is $0.45\min_k q_\theta(k)/\norm{B_{k,:}}$.
True-objective acceptance requires positive actual decrease,
$\rho\geq0.1$, and the Cauchy decrease test with constant $0.1$
and a ballwise model-Hessian bound. No reference solution or third
derivative is passed to the step solver.

\paragraph{Precision, references, and audits.}
The completed \texttt{paper} run uses MATLAB Symbolic Math Toolbox
variable-precision arithmetic at 120 digits. An independent
true-objective Newton solve from zero computes each local reference
$\theta^\star$ to gradient tolerance $10^{-90}$ and checks positive
definiteness of its Hessian. References are used only to position and
audit local probes. Across 40 problems, the reference gradient
norms are below $1.10\times10^{-92}$ and the smallest reference Hessian
eigenvalue exceeds $2.44\times10^{-3}$.
Centered finite differences with step $10^{-20}$ independently check
analytic gradients, Hessians, and third derivatives for the true
objective and three predictors. All 160 checks pass tolerance
$10^{-30}$. The maximum normalized third-derivative discrepancy is
$1.21\times10^{-41}$. Direct and aggregated predictor evaluations
also agree, with maximum normalized discrepancy $7.12\times10^{-128}$.
The inner arithmetic floor is $10^{-105}$. Runs needing smaller
forcing tolerances are flagged rather than counted as successful.

\paragraph{Local maps and prespecified summaries.}
For each seed, use a unit direction $d_s$ and the scale
$r_s=\min\{1,0.2\min_k q_{\theta^\star}(k)/\norm{B_{k,:}}\}$,
initialize at $\theta^\star\pm r_s r d_s$, with
$r\in\{10^{-2},10^{-3},\ldots,10^{-7}\}$. The two signs share one
random direction and are not independent seeds. The design gives
$5\times2\times4\times4\times2\times6=1{,}920$ one-step maps.
An eligible map must be accepted, meet~\eqref{eq:e10-forcing}, pass
the Cauchy test, avoid precision and boundary flags, and contract its
parameter error to a value above the reference error floor $10^{-80}$.
All 1,920 maps qualify. Every attempt is retained, independently of fitted order.

Fit log error against log initial radius on the fixed relative window
$10^{-7}\leq r\leq10^{-3}$, using five points per seed and sign.
Summarize signs within each seed and then report medians and IQRs
across five seeds. The same procedure fits the expansion remainder
\[
 \norm{e_+-b(e)},\qquad
 b(e)=\tfrac12H_\star^{-1}\nabla^3\Ecal(\theta^\star)[e,e].
\]
For Taylor-p2, use $\Ecal=F$. Figures~\ref{fig:e10-maps}
and~\ref{fig:e10-expansion} show the raw scaling and remainder audit.

\paragraph{Cancellation, error constants, and limits.}
At $\eta=0$, the median canonical order is $3.000$ for both generators,
whereas shifted templates and Taylor-p2 have order $2.000$.
The computed canonical residual third tensor has norm below
$4\times10^{-130}$, consistent with the exact cancellation.
At positive $\eta$, all group median orders are within $0.001$ of two.
Canonical one-step errors remain below Taylor-p2 in all 360 paired
positive-heterogeneity probes. Table~\ref{tab:e10-coefficients}
reports directional error-constant ratios, not speedups or a uniform ordering over directions.
All 320 individual expansion-remainder fits lie between
$2.99995$ and $3.00005$, supporting the vector coefficient, not merely
a fitted order (Figure~\ref{fig:e10-expansion}).

\begin{samepage}
Canonical $\alpha$ is not universally optimal: for KL4 at $\eta=0.2$, the negative shift has a smaller residual tensor norm in all five seeds. Common matching is a nonconflicting cancellation limit. Positive-$\eta$ cases test departures from it.
\par\end{samepage}

\paragraph{Actual trajectories and interpretation.}
Separate trajectories begin at $\theta^\star\pm0.01r_s d_s$ and
recalibrate after every accepted step, stopping at true-gradient norm
$10^{-30}$ or ten steps. All 320 trajectories reach that tolerance.
The stored 1,009 steps are accepted and meet the same forcing and
feasibility checks. E10 supports geometry-dependent local accuracy without third-order calibration, not runtime gains, global optimality, or generic cubic convergence for conflicting targets. E9 supplies the runtime evidence.
\par

\begin{figure}[!htbp]
 \centering
 \begin{tabular}{@{}cc@{}}
 \includegraphics[width=.36\linewidth,viewport=0 190 192 381,clip]{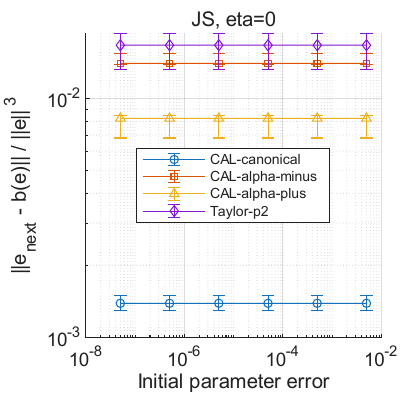}&
 \includegraphics[width=.36\linewidth,viewport=385 190 573 381,clip]{E10_expansion_audit.pdf}\\[3pt]
 \includegraphics[width=.36\linewidth,viewport=0 0 192 189,clip]{E10_expansion_audit.pdf}&
 \includegraphics[width=.36\linewidth,viewport=385 0 573 189,clip]{E10_expansion_audit.pdf}
 \end{tabular}
 \caption{E10 expansion audit: subtracting the predicted quadratic vector leaves an approximately constant $\norm{e_+-b(e)}/\norm e^3$. Panels show $\eta=0,0.05$. Fits cover all four heterogeneities. Third derivatives enter only the offline prediction and audit.}
 \label{fig:e10-expansion}
\end{figure}

\FloatBarrier

\end{document}